\documentclass[11pt,letterpaper]{amsart}

\usepackage[utf8]{inputenc}
\usepackage{mathbbol}
\usepackage{hyperref}
\usepackage{amsmath}
\usepackage{amssymb}
\usepackage[capitalise]{cleveref}
\usepackage{tikz-cd}
\usepackage{aligned-overset}
\usepackage{enumitem}

\DeclareMathOperator{\Coh}{Coh}

\DeclareMathOperator{\id}{id}
\DeclareMathOperator{\pt}{pt}
\DeclareMathOperator{\End}{End}

\DeclareMathOperator{\aff}{aff}
\DeclareMathOperator{\Hom}{Hom}
\DeclareMathOperator{\Irr}{Irr}
\DeclareMathOperator{\Ind}{Ind}
\DeclareMathOperator{\Perv}{Perv}
\DeclareMathOperator{\perf}{perf}
\DeclareMathOperator{\Spec}{Spec}

\title[Geometric realizations of affine Hecke algebras]{Geometric realizations of affine Hecke algebras with unequal parameters}
\author{Jonas Antor}
\address{University of Bonn, Endenicher Allee 60, 53115 Bonn, Germany}
\email{antor@math.uni-bonn.de}

\newtheorem{theorem}{Theorem}[section]
\newtheorem{lemma}[theorem]{Lemma}

\newtheorem{remark}[theorem]{Remark}
\newtheorem{definition}[theorem]{Definition}
\newtheorem{corollary}[theorem]{Corollary}

\newtheorem{proposition}[theorem]{Proposition}
\newtheorem{example}[theorem]{Example}

\newtheorem{mainthm}{Theorem}
\newtheorem*{theorem*}{Theorem}

\begin{document}
\begin{abstract}
We give a $K$-theoretic realization of all affine Hecke algebras with two unequal parameters including exceptional types. This extends the celebrated work of Kazhdan and Lusztig, who gave a $K$-theoretic realization of affine Hecke algebras with equal parameters, and complements results of Kato, who extended this construction to the three-parameter affine Hecke algebra of type $C$. A key idea behind our new construction is to exploit the reducibility of the adjoint representation in small characteristic. We also show that under suitable geometric conditions, our construction leads to a Deligne-Langlands style classification of simple modules. We verify these geometric conditions for $G_2$ thereby obtaining a full geometric classification of the simple modules for the affine Hecke algebra of $G_2$ with two parameters away from roots of unities.
\end{abstract}
\maketitle
\setcounter{tocdepth}{1}
\tableofcontents
\section{Introduction}
A famous theorem of Kazhdan and Lusztig \cite{kazhdan1987proof} states that affine Hecke algebras with equal parameters can be realized geometrically via the equivariant algebraic $K$-theory of the Steinberg variety (see also \cite{chriss2009representation}). This was used to prove the Deligne-Langlands conjecture which gives a geometric parameterization of the simple modules of these algebras away from roots of unities. In later work, Lusztig also gave a geometric construction of certain graded Hecke algebras via equivariant Borel-Moore homology \cite{lusztig1988cuspidal,lusztig1995cuspidal}, leading to a Langlands classification of unipotent representations of $p$-adic groups \cite{lusztig1995classification,lusztig2002classification}. This construction allows some unequal parameters but only those which are certain specific integral powers of a single parameter. Thus, it cannot be used to study affine Hecke algebras with arbitrary unequal parameters.

In type $C$, this limitation was overcome by Kato who gave a $K$-theoretic construction of the corresponding affine Hecke algebra with three arbitrary unequal parameters \cite{kato2009exotic}. This also yields a Deligne-Langlands style classification of simple modules under some mild assumptions on the parameters. The affine Hecke algebra of type $B$ with two parameters can also be studied with these methods since it can be realized as a quotient of the type $C$ affine Hecke algebra. However, Kato's construction doesn't allow to study affine Hecke algebras with unequal parameters of exceptional type.

In this paper, we give a uniform $K$-theoretic construction of all simple affine Hecke algebras with two unequal parameters including exceptional types. Note that, except for type $C$, affine Hecke algebras of simple type admit at most two parameters, so we cover almost all cases of unequal parameter Hecke algebras. Our construction also extends beyond simple types with possibly more parameters, and it recovers the equal parameter construction as a special case. However, for simplicity, we will focus on the simple case with two parameters in the introduction.

A key idea behind our new construction is to exploit certain small characteristic phenomena, in particular the reducibility of the adjoint representation. Kato already observed in \cite{kato2011deformations} that the geometry underlying his construction of unequal parameter Hecke algebras in type $C$ is closely related to certain characteristic $2$ phenomena, though this wasn't used to study affine Hecke algebras outside of classical types. 

Let $G$ be a simple algebraic group of non-simply laced type defined over an algebraically closed field $k$ with root system $\Phi$. Recall that the characteristic $p$ of $k$ is called special for $G$ (c.f. \cite{steinberg1963representations}) if it is equal to the ratio of the lengths of the long and short roots of $G$, i.e. if
\begin{itemize}
    \item $p =2 $ when $G$ is of type $B_n, C_n$ or $F_4$,
    \item $p=3$ when $G = G_2$.
\end{itemize}
Let $\mathfrak{g}$ be the adjoint representation of $G$. Then there is the following elementary but important observation of Hogeweij and Hiss.
\begin{theorem*}\cite{hogeweij1982almost,hiss1984adjungierten}
    If $p$ is special for $G$, then there is a $G$-stable subspace $\mathfrak{g}_s \subset \mathfrak{g}$ whose non-zero weights are the short roots of $G$.
\end{theorem*}
Let $p$ be special for $G$ and pick $\mathfrak{g}_s$ as above. The $G$-representation
\begin{equation*}
    V := \mathfrak{g}_s \oplus \mathfrak{g}/\mathfrak{g}_s
\end{equation*}
admits a canonical $G \times \mathbb{G}_m \times \mathbb{G}_m$-action where each $\mathbb{G}_m$ acts by scaling one of the summands. Fix a Borel subgroup and a maximal torus $T \subset B \subset G$ and let $\Phi^- \subset \Phi$ be the set of roots that appear in $Lie(B)$. Taking $V^- := \bigoplus_{\alpha \in \Phi^-} V_{\alpha}$ we can define an analogue of the Springer resolution
\begin{align*}
    \tilde{V} &:= G \times^B V^-  \overset{\mu}{\rightarrow}V, \quad (g,v) \mapsto gv.
\end{align*}
We denote the image of this map by $\mathfrak{N}(V)$ and write $\mathcal{B}_x = \mu^{-1}(x)$ for the fiber over $x \in V$. There also is an analogue of the Steinberg variety:
\begin{equation*}
    Z := \tilde{V} \times_V \tilde{V}.
\end{equation*}
Using the convolution formalism from \cite{chriss2009representation}, we can equip $K^{G \times \mathbb{G}_m \times \mathbb{G}_m}(Z)$ with an algebra structure. Let $\mathcal{H}^{\aff}_{q_1, q_2} $ be the (extended) affine Hecke algebra with two parameters (see \cref{def: affine Hecke algebra} and the paragraph after \cref{theorem: geo rel of affine Hecke algebras with unequal parameters}). The following theorem gives a geometric realization of $\mathcal{H}^{\aff}_{q_1, q_2}  $.

\begin{mainthm}(\cref{theorem: geo rel of affine Hecke algebras with unequal parameters})
    There is an isomorphism of algebras
    \begin{equation*}
        \mathcal{H}^{\aff}_{q_1, q_2} \cong K^{G \times \mathbb{G}_m \times \mathbb{G}_m}(Z).
    \end{equation*}
\end{mainthm}
Note that the theorem above does not assume that $G$ is simply connected which is the running assumption in \cite{kazhdan1987proof,chriss2009representation}. However, this assumption will come in when we study the representation theory of $\mathcal{H}^{\aff}_{q_1, q_2}$ which is what we discuss next. 

While our geometry is in positive characteristic, we only consider representations over the complex numbers. To classify the simple $\mathcal{H}^{\aff}_{q_1, q_2}$-modules, we first parameterize the central characters and then study the simple modules one central character at a time. Consider the complex torus $\mathcal{T} := \mathbb{C}^{\times} \otimes_{\mathbb{Z}} X_*(T)$ and set
\begin{align*}
    \hat{\mathcal{T}} := \mathcal{T} \times \mathbb{C}^{\times} \times \mathbb{C}^{\times}.
\end{align*}
The (complexified) center of $\mathcal{H}^{\aff}_{q_1, q_2}$ can be identified with the ring of functions on $\hat{\mathcal{T}}/W$. This yields a bijection
\begin{align*}
   \hat{\mathcal{T}}/W \overset{1:1}&{\longleftrightarrow} \{\text{Central characters } Z(\mathcal{H}^{\aff}_{q_1, q_2}) \rightarrow \mathbb{C} \} \\
   a & \longmapsto \chi_a
\end{align*}
classifying the central characters. In the equal parameter setting, the approach in \cite{kazhdan1987proof}, \cite{chriss2009representation} studies representations of affine Hecke algebras with a fixed central character by considering the $a$-fixed point locus of the Steinberg variety and Springer fibers. Since our geometry is defined in positive characteristic and $a$ lives in the complex torus $\hat{\mathcal{T}}$, we cannot take $a$-fixed points directly. Instead, we consider the characteristic $p$ torus
\begin{equation*}
    \hat{T} := T \times \mathbb{G}_m \times \mathbb{G}_m
\end{equation*}
which has the same character lattice as $\hat{\mathcal{T}}$. We then take fixed points with respect to the diagonalizable subgroup scheme
\begin{equation*}
    \hat{T}_a := Spec(k[X^*(\hat{\mathcal{T}}) / \{ \lambda \in X^*(\hat{\mathcal{T}}) \mid \lambda(a) = 1 \}]) \subset \hat{T}.
\end{equation*}

When $G$ is simply connected, the centralizer $G^{\hat{T}_a}$ is connected and reductive and its root system consists of all roots that vanish on $a$ (see \cref{lemma: explicit descrpition of centralizer in reductive group}, \cref{lemma: simply connected implies connected centralizer}). Note however that the group scheme $\hat{T}_a$ may not be reduced and the centralizer and fixed point loci depend on that non-reduced structure.

For each $x \in \mathfrak{N}(V)^{\hat{T}_a}$, the component group $A(a,x) := G^{\hat{T}_a}_x / (G^{\hat{T}_a}_x)^{\circ}$ naturally acts on $H_*(\mathcal{B}^{\hat{T}_a}_x)$. In analogy with the Deligne-Langlands correspondence, one expects that the irreducible representations of $\mathcal{H}^{\aff}_{q_1, q_2}$ with central character $\chi_a$ should be parameterized by the following set of `Langlands parameters':
\begin{equation*}
    \mathcal{P}(G,V,a) := \{(x,\rho) \mid x \in \mathfrak{N}(V)^{\hat{T}_a}, \rho \in \Irr(A(a,x)),  H_*(\mathcal{B}_x^{\hat{T}_a})_{\rho} \neq 0 \}/G^{\hat{T}_a}.
\end{equation*}
One of our main results shows that there indeed is such a parameterization under certain geometric conditions. For this, we will construct in \eqref{eq: def of Sa} another diagonalizable subgroup scheme $S_a \subset T$ associated to $a$ and consider the geometry of the pair $((G^{S_a})^{\circ}, V^{S_a})$. This is motivated by the reduction theorems in \cite{lusztig1989affine,barbasch1993reduction} which replace the affine Hecke algebra by a smaller affine Hecke algebra which is better behaved with respect to the chosen central character. We then define some geometric conditions for the pair $((G^{S_a})^{\circ}, V^{S_a})$: Condition (A1) states that there are finitely many nilpotent orbits associated to this pair. (A2) essentially says that there is a well-behaved associated Springer correspondence. (A3) states that the odd Borel-Moore homology of the associated Springer fibers vanishes. (A2') is a stronger version of (A2) which states that certain component groups act trivially (see \cref{section: geometric conditions} for the precise statements). Assuming (A1)-(A3), we obtain the following parameterization of simple modules.

\begin{mainthm}(\cref{thm: general DL correspondence})
    Let $a = (s,t_1, t_2) \in \hat{\mathcal{T}}$ and assume that
    \begin{itemize}
        \item $G$ is simply connected,
        \item $\langle t_1, t_2 \rangle \subset \mathbb{C}^{\times}$ is torsion-free,
        \item (A1)-(A3) hold for $((G^{S_a})^{\circ}, V^{S_a})$.
    \end{itemize}
    Then the number of $G^{\hat{T}_a}$-orbits on $\mathfrak{N}(V)^{\hat{T}_a}$ is finite and there is a bijection
    \begin{equation*}
        \left\{ L \in \Irr( \mathcal{H}^{\aff}_{q_1, q_2})  \mid L \text{ has central character } \chi_a  \right\}  \overset{1:1}{\leftrightarrow}  \mathcal{P}(G,V,a) .
    \end{equation*}
    If in addition (A2') holds for $((G^{S_a})^{\circ}, V^{S_a})$, then there is a bijection
    \begin{equation*}
        \left\{ L \in \Irr( \mathcal{H}^{\aff}_{q_1, q_2})  \mid L \text{ has central character } \chi_a  \right\} \overset{1:1}{\leftrightarrow}  \mathfrak{N}(V)^{\hat{T}_a}/G^{\hat{T}_a}.
\end{equation*}
\end{mainthm}
Without the assumption that $\langle t_1, t_2 \rangle$ is torsion-free, this correspondence breaks down in general. In fact, it can happen that $\mathfrak{N}(V)^{\hat{T}_a}$ has infinitely many $G^{\hat{T}_a}$-orbits (c.f. \cref{example: example where fixed point space has infinitely many orbits}) but there are always only finitely many irreducible $\mathcal{H}^{\aff}_{q_1, q_2}$-representations with a given central character. This is in contrast to the characteristic $0$ setting where the fixed-point space always has finitely many orbits (c.f \cite[Prop. 8.1.17]{chriss2009representation}).

For $G=G_2$ we verify by direct computation that (A1), (A2), (A3) and (A2') always hold. This yields the following `exotic Deligne-Langlands correspondence' for $G_2$.
\begin{mainthm}(\cref{theorem: dl correspondence for G2})
    Let $G = G_2$. Then for any $a = (s, t_1, t_2) \in \hat{\mathcal{T}}$ such that $\langle t_1, t_2 \rangle \subset \mathbb{C}^{\times}$ is torsion-free, there is a bijection
    \begin{equation*}
        \left\{ L \in \Irr( \mathcal{H}^{\aff}_{q_1, q_2})  \mid L \text{ has central character } \chi_a  \right\}  \overset{1:1}{\leftrightarrow} \mathfrak{N}(V)^{\hat{T}_a}/G^{\hat{T}_a}.
    \end{equation*}
\end{mainthm}
Note that this correspondence does not involve local systems which gives a cleaner picture than the classical Deligne-Langlands correspondence for $G_2$ (see \cref{example: comparison with classical DL corr}). Moreover, for $a = (e, 1,1) \in \hat{\mathcal{T}}$ the theorem above specializes to an `exotic Springer correspondence' for $G_2$:
\begin{equation*}
    \Irr(W) \overset{1:1}{\leftrightarrow} \mathfrak{N}(V)/G.
\end{equation*}
This is again cleaner than the classical Springer correspondence of $G_2$ in characteristic $0$ which involves local systems. The classical Springer correspondence for $G_2$ in characteristic $3$, on the other hand, also doesn't have non-trivial local systems \cite{xue2017springer}.

We expect that the conditions (A1)-(A3) hold for all reductive groups and a class of representations $V$ which we call `rooted' (see \cref{def: rooted representation}) while condition (A2') only holds in certain cases. The class of rooted representations includes the representations $V = \mathfrak{g}_s \oplus \mathfrak{g}/\mathfrak{g}_s$ considered above as well as the adjoint representation. Note that for the adjoint representation condition (A1) states that there are only finitely many nilpotent orbits (in the classical sense). In good characteristic there is a uniform proof of the fact that there are only finitely many nilpotent orbits \cite{richardson1967conjugacy}, however in bad characteristic, this has only been established using a case by case analysis and explicit computation \cite{stuhler1971unipotente,hesselink1979nilpotency,spaltenstein1983unipotent,spaltenstein1984nilpotent,holt1985nilpotent}. Since our main interest is in special characteristic which is always bad, we expect that finding a uniform proof of (A1)-(A3) for the representations $V = \mathfrak{g}_s \oplus \mathfrak{g}/\mathfrak{g}_s$ is difficult while a proof via case by case analysis might be more manageable. We will investigate conditions (A1)-(A3) in the $F_4$ case in separate work. This case is of particular interest since $G_2$ and $F_4$ are the only exceptional types for which the affine Hecke algebra admits unequal parameters.

This paper is organized is follows. In \cref{section: geometric realization} we prove our geometric realization theorem. In \cref{section: simple modules and perverse sheaves} we deduce from this that the affine Hecke algebra at a central character can be identified with the Ext-algebra of a certain semisimple complex of perverse sheaves (see \cref{thm: affine Hecke algebra at central character is Ext algebra}). In \cref{section: a dl correspodence with unequal parameters} we discuss the geometric conditions (A1)-(A3) and (A2') and show that they imply a Deligne-Langlands classification of simple modules. We first show this for `positive real central characters' by adapting arguments from \cite{kato2009exotic} and \cite{reeder2002isogenies} and then show that the general case can be reduced to the positive real case. In \cref{section: affine Hecke for G2} we verify conditions (A1)-(A3) and (A2') for $G_2$ and deduce the Deligne-Langlands classification of simple modules for $G_2$. In \cref{section: appendix} we summarize the necessary background material on equivariant $K$-theory and Borel-Moore homology.

\subsection*{Acknowledgements}
I am thankful to Dan Ciubotaru, Ian Grojnowski and Kevin McGerty for useful discussions.
\section{A geometric realization}\label{section: geometric realization}
Throughout this paper, let $G$ be a reductive group defined over an algebraically closed field $k$. Fix a maximal torus $T\subset G$ and let $(X^*(T), X_*(T), \Phi, \Phi^{\vee})$ be the corresponding root datum. Fix a system of positive roots $\Phi^+ \subset \Phi$ with simple roots $\Pi \subset \Phi^+$ and denote the set of negative roots by $\Phi^- = - \Phi^+$. Let $B$ be the corresponding (geometric) Borel subgroup, i.e. the Borel subgroup whose Lie algebra has non-zero weights $\Phi^-$. We denote by $\mathcal{B} = \mathcal{B}(G) = G/B$ the flag variety of $G$. Let $W = N_G(T)/T$ be the Weyl group of $G$. For any $w \in W$ we denote by $\dot{w} \in N_G(T)$ a lift of $w$. The root group associated to $\alpha$ will be denoted by $U_{\alpha} \subset G$. The $G$-equivariant $K$-theory (i.e. the Grothendieck group of equivariant coherent sheaves) will be denoted by $K^G(-)$ and the Grothendieck group of equivariant perfect complexes by $K_G(-)$ (see \cref{section: app k theory}).
\subsection{$K$-theory of vector bundles and Steinberg varieties}\label{section: k theory of vector bundles}
For any $B$-module $W$ there is a corresponding $G$-equivariant vector bundle $\pi: G \times^B W \rightarrow \mathcal{B}$. Using the Thom isomorphism and the induction isomorphism \eqref{eq: induction iso} we can identify
\begin{equation}\label{eq: identification of character group algebra with K-theory of vector bundle}
    K^G( G \times^B W) \cong K^G(\mathcal{B}) \cong  K^B(\pt) \cong K^T(\pt) \cong \mathbb{Z}[X^*(T)].
\end{equation}
Under this isomorphism, the element $e^{\lambda} \in \mathbb{Z}[X^*(T)]$ corresponds to $\pi^* [\mathcal{E}_{\lambda}] \in K^G( G \times^B W)$ where $\mathcal{E}_{\lambda}$ denotes the line bundle on $\mathcal{B}$ whose fiber over $B$ has weight $\lambda$. From now on we will identify $\mathbb{Z}[X^*(T)] $ and $K^G(G \times^B W)$ via this isomorphism. We will need the following two standard results about the equivariant $K$-theory of vector bundles.
\begin{lemma}\label{lemma: formula for inclusion of vector bundles}
    Let $V$ be a $B$-module, $V' \subset V$ a $B$-stable subspace and denote by $\iota: G \times^B V' \hookrightarrow G \times^B V$ the inclusion. Then the map
    \begin{equation*}
        K^G(G \times^B V') \overset{\iota_*}{\longrightarrow} K^G(G \times^B V) 
    \end{equation*}
    is given by multiplication with $\lambda^{\vee}(V/V') := \prod_{\lambda \in X^*(T)} (1-e^{-\lambda})^{\dim (V/V')_\lambda}$.
\end{lemma}
\begin{proof}
    Consider the projections $\pi : G \times^B V \rightarrow \mathcal{B} $ and $ \pi' : G \times^B V'  \rightarrow \mathcal{B}$. Then, for any $\mu \in X^*(T)$, we have
    \begin{equation*}
        \iota_* [(\pi')^* \mathcal{E}_{\mu} ] \overset{functoriality}{=} \iota_* \iota^* \pi^* [\mathcal{E}_{\mu}] \overset{\eqref{eq: projection formula for coherent sheaves and vector bundles}}{=} \iota_*[\mathcal{O}_{G \times^B V'}] \otimes \pi^*[\mathcal{E}_{\mu}].
    \end{equation*}
    Thus, $\iota_*$ acts by multiplication with $\iota_*[\mathcal{O}_{G \times^B V'}]$, so it suffices to show that $\iota_*[\mathcal{O}_{G \times^B V'}] = \lambda^{\vee}(V/V')$. Using the induction isomorphism \eqref{eq: induction iso} and the fact that $K^B(-) \cong K^T(-)$, we may reduce to the case where $G = B = T$. Choosing a splitting $V \cong V' \oplus V/V'$ as $T$-modules, we can view the inclusion $\iota: V' \hookrightarrow V$ as the zero section of the $T$-equivariant vector bundle $p_1: V \rightarrow V'$ with fiber $V/V'$. Let $\mathcal{E}$ be the $T$-equivariant locally free sheaf on $V'$ corresponding to this vector bundle. It now follows from the Koszul resolution
    \begin{equation*}
        ... \rightarrow \Lambda^1 p_1^* \mathcal{E}^{\vee} \rightarrow \Lambda^0 p_1^* \mathcal{E}^{\vee} \rightarrow \iota_* \mathcal{O}_{V'} \rightarrow 0
    \end{equation*}
    that 
\begin{align*}
    \iota_*[ \mathcal{O}_{V'}]&= \sum_i (-1)^i [\Lambda^i p_1^* \mathcal{E}^{\vee}] \\
    &= \sum_i \sum_{\lambda \in X^*(T)} (-1)^i \dim (\Lambda^i(V/V'))_\lambda \cdot e^{-\lambda}\\
    &= \prod_{\lambda \in X^*(T)} (1-e^{-\lambda})^{\dim (V/V')_\lambda} \\
    &= \lambda^{\vee}(V/V').
\end{align*}
\end{proof}
\begin{lemma}\label{lemma: projection yields demazure-like operator}
    Let $\alpha \in \Pi$ be a simple root and let $P_\alpha= B \sqcup B\dot{s}_\alpha B$ be the corresponding parabolic subgroup. Let $V$ be a $P_\alpha$-module and consider the projection $\pi_\alpha : G \times^B V \rightarrow G \times^{P_{\alpha}} V$. Then we have
    \begin{equation*}
        \pi_{\alpha}^* (\pi_{\alpha})_* e^{\lambda} = \tfrac{e^{\lambda} - e^{s_{\alpha}(\lambda) - \alpha}}{1-e^{-\alpha}} 
    \end{equation*}
    in $K^G(G \times^B V)$.
\end{lemma}
\begin{proof}
    Let $L_{\alpha}$ be the standard Levi in $P_{\alpha}$. There is a commutative diagram
    \begin{equation*}
        \begin{tikzcd}
        K^G(G\times^B V) \arrow[d, "\pi_{\alpha}"] \arrow[r, equal]  & K^G(G/B) \arrow[d] \arrow[r, equal]  & K^{P_{\alpha}}(P_{\alpha}/B) \arrow[d] \arrow[r, equal]  & K^{L_{\alpha}}(P_{\alpha}/B) \arrow[d] \\
        K^G(G\times^{P_{\alpha}} V) \arrow[r, equal]       & K^G(G/P_{\alpha}) \arrow[r, equal]     & K^{P_{\alpha}}(\pt )   \arrow[r, equal]     & K^{L_{\alpha}}(\pt )           
        \end{tikzcd}
    \end{equation*}
    where the horizontal identifications are induced by Thom isomorphisms, induction isomorphisms and restriction along $L_{\alpha} \subset P_{\alpha}$. Note that $P_{\alpha}/B$ is the flag variety of $L_{\alpha}$, so the map $K^{L_{\alpha}}(P_{\alpha}/B) \rightarrow K^{L_{\alpha}}(\pt )$ is given by the (rank 1) Weyl character formula $e^{\lambda} \mapsto \tfrac{e^{\lambda} - e^{s_{\alpha}(\lambda) - \alpha}}{1-e^{-\alpha}}$ which holds in arbitrary characteristic (c.f. \cite[II.5.10]{jantzen2003representations}).
\end{proof}
Let $V$ be a $G$-representation and $V^-$ a $B$-stable subspace. Let
\begin{equation*}
    \tilde{V}  := G \times^B V^- \cong \{ (gB,v) \in \mathcal{B} \times V \mid  g^{-1}\cdot v \in V^- \}
\end{equation*}
which comes with a canonical morphism
\begin{equation*}
    \tilde{V} \overset{\mu}{\longrightarrow} V, \quad (gB, v) \mapsto v.
\end{equation*}
The associated `Steinberg variety'
\begin{equation*}
    Z := \tilde{V} \times_V \tilde{V} = \{ ((g_1B,v_1),(g_2B,v_2)) \in \tilde{V} \times \tilde{V} \mid v_1 = v_2 \}
\end{equation*}
comes with a canonical projection map
\begin{equation*}
    \pi \times \pi: Z \rightarrow \mathcal{B} \times \mathcal{B}.
\end{equation*}
\begin{remark}
    The fiber product $Z = \tilde{V} \times_V \tilde{V}$ is taken in the category of varieties here, i.e. $Z$ is the underlying reduced scheme of the scheme theoretic fiber product $\tilde{V} \times^{sch}_V \tilde{V}$. One can also work with $\tilde{V} \times^{sch}_V \tilde{V}$ which may not be reduced. The difference between the two essentially does not matter for our purposes since $K$-theory does not detect the non-reduced structure.
\end{remark}
Note that we can view $Z$ as a closed subvariety of the smooth variety $\tilde{V} \times \tilde{V}$. Then convolution with respect to the ambient space $\tilde{V} \times \tilde{V}$ equips $K^G(Z)$ with the structure of a $K^G(\pt)$-algebra (see \eqref{eq: convolution def}). The inclusion of the diagonal  $ Z_{\Delta} := \tilde{V} \times_{\tilde{V}} \tilde{V} \subset Z$ induces a map $K^G(Z_{\Delta}) \rightarrow  K^G(Z)$ which is an algebra homomorphism by \cref{lemma: convolution along diagonal}. Let
\begin{equation*}
    \mathcal{B} \times \mathcal{B} = \bigsqcup_{w \in W} Y_w
\end{equation*}
be the decomposition into $G-$orbits where $Y_w = G \cdot (eB, \dot{w}B)$ and let
\begin{equation*}
    Z_w :=(\pi \times \pi)^{-1}(Y_w).
\end{equation*}
We write ${}^w V^- := \dot{w} V^-$ and ${}^w B := \dot{w} B \dot{w}^{-1}$.
\begin{lemma}\label{lemma: basic properties of Z}
    $Z_y \rightarrow Y_w $ is a $G$-equivariant vector bundle and there is an isomorphism $G \times^{B \cap {}^w B} (V^- \cap {}^w V^-) \cong Z_w$. In particular, we have $\dim Z_w = \dim G - \dim B \cap {}^w B + \dim V^- \cap {}^w V^-$.
 \end{lemma}
 \begin{proof}
     The fact that $Z_w \rightarrow Y_w \cong G/ B \cap {}^w B$ is a vector bundle can easily be checked with the help of the Bruhat decomposition. The fiber over $(B, \dot{w}B) \in Y_w$ is $V^- \cap {}^w V^-$ which implies that the canonical map of vector bundles $G \times^{B \cap {}^w B} (V^- \cap {}^w V^-) \rightarrow Z_w$ is an isomorphism. The equality $\dim Z_w = \dim G - \dim B \cap {}^w B + \dim V^- \cap {}^w V^-$ follows directly from this.
 \end{proof}
We pick a total order on $W$ extending the Bruhat order and define
\begin{align*}
    Z_{\le w} &:= \bigsqcup_{y \le w} Z_y \\
    Z_{< w} &:= \bigsqcup_{y < w} Z_y.
\end{align*}
Note that $Z_{<w} = Z_{\le w'}$ where $w'$ is the maximal element with $w'<w$ and $Z_{\Delta} = Z_{\le e }  = Z_e$.
\begin{lemma}\label{lemma: K-theory of Z is free over the diagonal} Let $H \subset G$ be diagonalizable subgroup scheme or $H= G$. Then $K^H(Z_{\le w})$ is a free left $K^H(Z_{\Delta})$-module with basis $\{ [\mathcal{O}_{\overline{Z}_y}] \mid y \le w \}$. In particular, $K^H(Z)$ is a free left $K^H(Z_{\Delta})$-module with basis $\{ [\mathcal{O}_{\overline{Z}_w}] \mid w \in W \}$.
\end{lemma}
\begin{proof}
    For any $G$-stable locally closed subvariety $X \subset \tilde{V} \times \tilde{V}$, there is a canonical $K_H(\mathcal{B})$-module structure on $K^H(X)$ via
    \begin{equation*}
        [\mathcal{E}] \cdot [\mathcal{F}] := [p^* \mathcal{E} \otimes \mathcal{F}]
    \end{equation*}
    where $p$ is the composition $X \hookrightarrow \tilde{V} \times \tilde{V} \overset{p_1}{\rightarrow} \tilde{V} \rightarrow \mathcal{B}$. Note that we can identify
    \begin{equation*}
        K^H(Z_{\Delta}) \cong K^H(\mathcal{B})\cong K_H(\mathcal{B})
    \end{equation*}
    via the Thom isomorphism. By \cref{lemma: convolution along diagonal}, the $K^H(Z_{\Delta})$-module structure on $K^H(X)$ via convolution agrees with the canonical $K_H(\mathcal{B})$-module structure just defined. Thus, it suffices to show that $K^H(Z_{\le w})$ is a free $K_H(\mathcal{B})$-module with basis $\{ [\mathcal{O}_{\overline{Z}_y}] \mid y \le w  \}$. We prove this by induction on $l(w)$. Note that for any $y \in W$, the Thom isomorphism together with \cref{lemma: basic properties of Z} implies that $K^H(Z_y)$ is a free $K_H(\mathcal{B})$-module of rank $1$ with basis $[\mathcal{O}_{Z_y}]$. This proves the claim for $Z_{\le e} = Z_e$. We have a short exact sequence of $K_H(\mathcal{B})$-modules
    \begin{equation*}
        \begin{tikzcd}
        0 \arrow[r] & K^H(Z_{<w}) \arrow[r] & K^H(Z_{\le w}) \arrow[r] & K^H(Z_w) \arrow[r] & 0
    \end{tikzcd}
    \end{equation*}
    where the exactness on the left is proved exactly as in \cite[(5.4.4)]{chriss2009representation}. We know by the induction hypothesis that the outer terms are free $K_H(\mathcal{B})$-modules with bases $\{ [\mathcal{O}_{\overline{Z}_y}] \mid y < w\}$ and $\{[\mathcal{O}_{Z_w}] \}$ respectively. Thus, the sequence splits and $K^H(Z_{\le w})$ is also a free $K_H(\mathcal{B})$-module. Note that $[\mathcal{O}_{\overline{Z}_w}] \in K^H(Z_{\le w}) $ lifts the element $[\mathcal{O}_{Z_w}] \in K^H(Z_{w})$. Thus, $K^H(Z_{\le w})$ has the basis $\{ [\mathcal{O}_{\overline{Z}_y}] \mid y \le w\}$ as a $K_H(\mathcal{B})$-module. Since $Z = Z_{\le w_0}$ we get that $K^H(Z)$ is a free left $K^H(Z_{\Delta})$-module with basis $\{ [\mathcal{O}_{\overline{Z}_w}] \mid w \in W \}$.
\end{proof}
The result above implies that the algebra homomorphism $K^G(Z_{\Delta}) \rightarrow K^G(Z)$ is injective. Hence, we can view $K^G(Z_{\Delta})$ as a subalgebra of $K^G(Z)$.

\subsection{Affine Hecke algebras with unequal parameters}
Fix a root datum $\mathcal{R} = (X^*, X_*, \Phi, \Phi^{\vee})$ with a system of positive roots $\Phi^+ \subset \Phi$. Let $I$ be a finite set and denote the ring of Laurent polynomials with $|I|$ variables by
\begin{equation*}
    \mathcal{A} := \mathbb{Z}[q_i^{\pm 1} \mid i \in I].
\end{equation*}
\begin{definition}
    A parameter function is a function
\begin{equation*}
    \textbf{q} : \Phi \rightarrow \{ q_i \mid i \in I\} \subset \mathcal{A}
\end{equation*}
which is constant on $W$-orbits.
\end{definition}

\begin{definition}\label{def: affine Hecke algebra}
    The affine Hecke algebra $\mathcal{H}^{\aff} = \mathcal{H}^{\aff}_{\textbf{q}}$ is the $\mathcal{A}$-algebra with basis $\{ T_w e^{\lambda} \mid w \in W, \lambda \in X^*\} $ and multiplication rules
    \begin{enumerate}[label=(\roman*)]
        \item \makebox[8cm]{ $T_{s_{\alpha}}^2 = (\textbf{q}(\alpha)-1) T_{s_{\alpha}} + \textbf{q}(\alpha)$\hfill} $\alpha \in \Pi$,
        \item \makebox[8cm]{ $T_w T_{w'} = T_{ww'}$ \hfill} $l(ww') = l(w) + l(w')$,
        \item \makebox[8cm]{ $e^{\lambda} e^{\mu} = e^{\lambda + \mu}$ \hfill} $\lambda, \mu \in X^*$,
        \item \makebox[8cm]{$e^{\lambda} T_{s_{\alpha}}- T_{s_{\alpha}} e^{s_{\alpha}(\lambda )} = (\textbf{q}(\alpha) -1) \frac{e^{\lambda} - e^{s_{\alpha}(\lambda)}}{1- e^{-\alpha}}$ \hfill} $\alpha \in \Pi$, $\lambda \in X^*$,
    \end{enumerate}
    where $T_w := T_we^{0}$ and $e^{\lambda} := T_e e^{\lambda}$.
\end{definition}
\begin{remark}
    Our definition of the affine Hecke algebra is similar to the one given in \cite{chriss2009representation}. In the literature, the affine Hecke algebra is also often defined replacing the character lattice with the cocharacter lattice in our definition. This is natural from a $p$-adic groups perspective but would introduce a lot of Langlands dual notation in our geometry. Since we never consider $p$-adic groups in this paper, we stick to the convention in \cite{chriss2009representation}. 
\end{remark}
\begin{remark}
    If $\mathcal{R}$ is simple, then $\textbf{q}(\alpha) = \textbf{q}(\beta)$ whenever $\alpha$ and $\beta$ have the same length. In particular, there is only one parameter in simply laced types and at most two in non-simply laced types. This covers all the simple affine Hecke algebras with unequal parameters in the literature except for a third parameter which can occur for the affine Hecke algebra of $SO(2n+1)$ (or $Sp(2n)$ in the dual cocharacter lattice convention). In our definition this third parameter is equal to the parameter of the short roots (resp. the parameter of the long roots in the dual convention). In terms of the parameter functions in \cite[3.1(c)]{lusztig1989affine} this corresponds to the case where $\lambda^* = \lambda|_{\{ \alpha \in \Pi \mid \check{\alpha} \in 2 X_* \}}$.
\end{remark}
Here are some elementary algebraic properties of $\mathcal{H}^{\aff}_{\textbf{q}}$ which easily follow from the definition (c.f. \cite{lusztig1989affine}).
\begin{lemma}\label{lemma: alg properties of aff heck with unequal parmaeters}
    \begin{enumerate}[label=(\roman*)]
        \item The $\{ T_w \mid w \in W \}$ span a subalgebra $\mathcal{H}_{\textbf{q}} \subset \mathcal{H}^{\aff}_{\textbf{q}}$ which is the (finite) Hecke algebra with unequal parameters.
        \item The $e^{\lambda}$ span a commutative subalgebra of $\mathcal{H}^{\aff}_{\textbf{q}}$ which is isomorphic to the group algebra $\mathcal{A}[X^*]$.
        \item $\mathcal{H}^{\aff}_{\textbf{q}}$ is a free right (and left) $\mathcal{A}[X^*]$-module with basis $\{ T_w \mid w \in W \}$.
        \item $\mathcal{H}^{\aff}_{\textbf{q}}$ is generated as an algebra by $\mathcal{A}[X^*]$ together with the $T_{s_{\alpha}}$ for $\alpha \in \Pi$.
        \item $Z(\mathcal{H}^{\aff}_{\textbf{q}}) = \mathcal{A}[X^*]^W$.
    \end{enumerate}
\end{lemma}
Let $sgn$ be the sign representation of $\mathcal{H}_{\textbf{q}}$, i.e. the free $\mathcal{A}$-module of rank $1$ on which $T_{s_{\alpha}}$ acts as $-1$. The antispherical module is the $\mathcal{H}^{\aff}_{\textbf{q}}$-module
\begin{equation*}
    M_{asph} := \mathcal{H}^{\aff}_{\textbf{q}} \otimes_{\mathcal{H}_{\textbf{q}}} sgn
\end{equation*}
which is a free $\mathcal{A}$-module with standard basis $\{ e^{\lambda} \otimes 1 \mid \lambda \in X^* \}$.
\begin{lemma}\label{lemma: explicit formulas for antispherical module}
$M_{asph}$ is the unique $\mathcal{H}^{\aff}_{\textbf{q}}$-module with $\mathcal{A}$-basis $\{ e^{\lambda} \otimes 1 \mid \lambda \in X^* \}$ and multiplication rules
\begin{enumerate}[label=(\roman*)]
    \item $e^\mu \cdot ( e^{\lambda} \otimes 1) = e^{\mu + \lambda} \otimes 1$
    \item $-( T_{s_{\alpha}} + \textbf{q}(\alpha)e^{\alpha}) \cdot  e^{\lambda} \otimes 1 = (1-\textbf{q}(\alpha)e^{\alpha})\frac{e^{\lambda}- e^{s_{\alpha}(\lambda) - \alpha}}{1- e^{-\alpha}}  \otimes 1$.
\end{enumerate}
\end{lemma}
\begin{proof}
    The uniqueness follows from the fact that the $e^{\lambda}$ together with the $T_{s_{\alpha}}$ generate $\mathcal{H}^{\aff}_{\textbf{q}}$ (\cref{lemma: alg properties of aff heck with unequal parmaeters}). The equality $e^\mu \cdot ( e^{\lambda} \otimes 1) = e^{\mu + \lambda} \otimes 1$ follows from the definition. Using the relation
    \begin{equation*}
       -T_{s_{\alpha}} e^{\lambda} = -e^{s_{\alpha}(\lambda)} T_{s_{\alpha}}  +(\textbf{q}(\alpha) -1) \frac{e^{s_{\alpha}(\lambda)}- e^{\lambda}}{1- e^{-\alpha}}
    \end{equation*}
    we get
    \begin{align*}
        & -( T_{s_{\alpha}} + \textbf{q}(\alpha)e^{\alpha}) \cdot  e^{\lambda} \otimes 1 \\
        & = \left( -e^{s_{\alpha}(\lambda)} T_{s_{\alpha}} + (\textbf{q}(\alpha) -1) \frac{e^{s_{\alpha}(\lambda)}- e^{\lambda}}{1- e^{-\alpha}} - \textbf{q}(\alpha)e^{\alpha} e^{\lambda} \right) \otimes 1 \\
        &=  \left( e^{s_{\alpha}(\lambda)} +(\textbf{q}(\alpha) -1) \frac{e^{s_{\alpha}(\lambda)}- e^{\lambda}}{1- e^{-\alpha}} - \textbf{q}(\alpha) e^{\lambda + \alpha} \right) \otimes 1 \\
        &=  \left( \frac{ e^{s_{\alpha}(\lambda)} - e^{s_{\alpha}(\lambda) - \alpha} }{1-e^{-\alpha}}  + (\textbf{q}(\alpha) -1) \frac{e^{s_{\alpha}(\lambda)}- e^{\lambda}}{1- e^{-\alpha}} - \textbf{q}(\alpha) \frac{e^{\lambda + \alpha} - e^{\lambda}}{1-e^{-\alpha}}\right) \otimes 1 \\
        &=  (1-\textbf{q}(\alpha)e^{\alpha})\frac{e^{\lambda} - e^{s_{\alpha}(\lambda) - \alpha}}{1- e^{-\alpha}}  \otimes 1. 
    \end{align*}
\end{proof}
\begin{lemma}\label{lemma: antispherical module is faithful}\cite[(4.3.10)]{macdonald2003affine}
    $M_{asph}$ is a faithful $\mathcal{H}^{\aff}_{\textbf{q}}$-module, i.e. the $\mathcal{H}^{\aff}_{\textbf{q}}$-module structure on $M_{asph}$ induces an injective algebra homomorphism $\mathcal{H}^{\aff}_{\textbf{q}} \hookrightarrow \End_{\mathcal{A}}(M_{asph})$.
\end{lemma}

\subsection{Rooted representations}
We now introduce a special class of $G$-representations.
\begin{definition}\label{def: rooted representation}
    A $G$-representation $V$ is called rooted if the non-zero weights of $V$ are precisely the roots of $G$ and $\dim V_{\alpha} = 1$ for all $\alpha \in \Phi$. Let $I$ be a finite set. A rooted $G$-representation $V$ together with a decomposition of $G$-modules $V = \bigoplus_{i \in I}V_i$ is called in $I$-rooted $G$-representation.
\end{definition}
Fix an $I$ rooted $G$-representation $V$. There is an associated parameter function
\begin{equation*}
    \textbf{q}_V : \Phi \rightarrow \{ q_i \mid i \in I \}
\end{equation*}
with $\textbf{q}_V(\alpha) = q_{i_{\alpha}}$ where $i_{\alpha} \in I$ is the unique element with $(V_i)_{\alpha} \neq \{ 0\}$. Moreover, we can equip $V$ with the structure of a $G \times (\mathbb{G}_m)^I$-module via
\begin{equation*}
    (g,(t_i)_{i \in I}) \cdot(v_i)_{i \in I} = (t_i^{-1} gv_i)_{i \in I}.
\end{equation*}
Consider the $B$-stable subspace
    \begin{equation*}
        V^- := \bigoplus_{\alpha \in \Phi^-} V_{\alpha} \subset V. 
    \end{equation*}
Then we can consider the varieties $\tilde{V} = G \times^B V^-, Z = \tilde{V} \times_V \tilde{V}, ...$ as in \cref{section: k theory of vector bundles}. For any simple root $\alpha \in \Pi$, we denote the corresponding parabolic subgroup by $ P_{\alpha} = B \sqcup B\dot{s}_{\alpha}B$.
\begin{lemma}
    For each simple root $\alpha \in \Pi$, the subspace $V^- \cap {}^{s_{\alpha}} V^-$ is $P_{\alpha}$-stable.
\end{lemma}
\begin{proof}
    Note that $V^- \cap {}^{s_{\alpha}} V^- = \bigoplus_{\gamma \in \Phi^- \backslash \{- \alpha \}} V_{\mathfrak{\gamma}}$. Since $U_\beta \cdot V_{\lambda} \subset \bigoplus_{n \ge 0} V_{\lambda + n \beta}$, the subspace $V^- \cap {}^{s_{\alpha}} V^-$ is $U_{\beta}$ stable for all $\beta \in \Phi^-$. It is also $T$-stable and thus also $B$-stable. It is also stable under $\dot{s}_{\alpha}$ and hence also under $P_{\alpha}$.
\end{proof}
We define varieties
\begin{align*}
    \tilde{V}_{\alpha} &:= G \times^B (V^- \cap {}^{s_{\alpha}} V^- )  \\
    \bar{V}_{\alpha} &:= G \times^{P_{\alpha}} (V^- \cap {}^{s_{\alpha}} V^-)\\
    \Lambda_{\alpha} &:= \tilde{V}_{\alpha} \times_{\bar{V}_{\alpha}} \tilde{V}_{\alpha}.
\end{align*}
Note that $\tilde{V}_{\alpha} \subset \tilde{V}$ is a subbundle and $\Lambda_{\alpha} \subset Z $ is closed.
\begin{lemma}\label{lemma: Lambda alpha is closure of Zs}
    We have $\Lambda_{\alpha} = \overline{Z}_{s_{\alpha}}$.
\end{lemma}
\begin{proof}
    The fiber of $(B , \dot{s}_{\alpha}B) \in \mathcal{B} \times \mathcal{B}$ under the projection $ Z \rightarrow \mathcal{B} \times \mathcal{B}$ is precisely $V^- \cap {}^{s_{\alpha}}V^-$ which is also the fiber under $\Lambda_{\alpha} \rightarrow \mathcal{B} \times \mathcal{B}$. By $G$-equivariance, this implies $Z_{s_{\alpha}} \subset \Lambda_{\alpha}$ and hence $\overline{Z}_{s_{\alpha}} \subset \Lambda_{\alpha}$. Note that the map $\tilde{V}_{\alpha} \rightarrow \bar{V}_{\alpha}$ is a locally trivial $\mathbb{P}^1$-fibration. By base change, $\Lambda_{\alpha} \rightarrow \tilde{V}_{\alpha}$ is also a locally trivial $\mathbb{P}^1$-fibration. In particular, this map is flat and hence open \cite[\href{https://stacks.math.columbia.edu/tag/01UA}{Tag 01UA}]{stacks-project} which implies that $\Lambda_{\alpha}$ is irreducible by \cite[\href{https://stacks.math.columbia.edu/tag/004Z}{Tag 004Z}]{stacks-project}. Moreover, we have
    \begin{align*}
        \dim \Lambda_{\alpha} &= \dim \tilde{V}_{\alpha} +1  \\
        &= \dim G - \dim B  + \dim V^-\cap {}^{s_{\alpha}}V^- + 1\\
        &= \dim Y_{s_{\alpha}}  + \dim V^-\cap {}^{s_{\alpha}}V^- \\
        &= \dim Z_{s_{\alpha}}\\
        &= \dim \overline{Z}_{s_{\alpha}}.
    \end{align*}
    This implies that $\Lambda_{\alpha} = \overline{Z}_{s_{\alpha}}$ by the irreducibility of $\Lambda_{\alpha}$.
\end{proof}
We can equip $K^{G \times (\mathbb{G}_m)^I} ( \tilde{V})$ with the structure of a left $K^{G \times (\mathbb{G}_m)^I} ( Z)$-module via convolution (with respect to the smooth ambient varieties $X_1 = X_2 = \tilde{V}$, $X_3 = \{ pt \} $). Given $[\mathcal{F}] \in K^{G \times (\mathbb{G}_m)^I} ( Z)$ and $ [\mathcal{G}] \in K^{G \times (\mathbb{G}_m)^I} ( \tilde{V})$, their convolution is given by
\begin{equation*}
    [\mathcal{F}] \star [\mathcal{G}] = (p_1)_* ([\mathcal{F}] \otimes^L_{\tilde{V} \times \tilde{V}} p_2^* [\mathcal{G}]) \in K^{G \times (\mathbb{G}_m)^I} ( \tilde{V}).
\end{equation*}
Recall from \eqref{eq: identification of character group algebra with K-theory of vector bundle} that we have a canonical identification
\begin{equation}\label{eq: identification of K(tilde V) with polynomial ring}
    \mathcal{A}[X^*(T)]\cong K^{G \times (\mathbb{G}_m)^I} ( \tilde{V}) .
\end{equation}
\begin{proposition}\label{prop: formula for convolution with Lambda alpha}
    For any $\lambda \in X^*(T)$, $\alpha \in \Pi$, convolution of $[\mathcal{O}_{\overline{Z}_{s_{\alpha}}}] \in K^{G \times (\mathbb{G}_m)^I}(Z)$ with $e^{\lambda} \in K^{G \times (\mathbb{G}_m)^I} ( \tilde{V})$ is given by
    \begin{equation*}
        [\mathcal{O}_{\overline{Z}_{s_{\alpha}}}] \star e^{\lambda} = (1 - \textbf{q}_V(\alpha) e^{\alpha})\frac{e^{\lambda} - e^{s_{\alpha}(\lambda) - \alpha}}{1-e^{-\alpha}} \in K^{G \times (\mathbb{G}_m)^I } ( \tilde{V}).
    \end{equation*}
    
\end{proposition}
\begin{proof}
    There is a commutative diagram
    \begin{equation*}
    \begin{tikzcd}
    \tilde{V} \times \tilde{V} \arrow[rr, "p_2"] &  & \tilde{V}\\
    Z \arrow[d, "p_1"'] \arrow[u, "\iota_Z", hook]  & \Lambda_{\alpha} \arrow[r, "p_2"] \arrow[d, "p_1"'] \arrow[l, "\tilde{\iota}_{\alpha}"', hook] & \tilde{V}_{\alpha} \arrow[u, "\iota_{\alpha}", hook] \arrow[d, "\pi_{\alpha}"'] \\
    \tilde{V}  & \tilde{V}_{\alpha} \arrow[r, "\pi_{\alpha}"] \arrow[l, "\iota_{\alpha}"', hook]                & \bar{V}_{\alpha}                   \end{tikzcd}
    \end{equation*}
    where the bottom right square is cartesian. Using these maps, we have
    \begin{align*}
        [\mathcal{O}_{\Lambda_{\alpha}}] \star e^{\lambda} \overset{def}&{=} (p_1)_*( (\tilde{\iota}_{\alpha})_*[\mathcal{O}_{\Lambda_{\alpha}}] \otimes^L_{\tilde{V} \times \tilde{V}} p_2^*  e^{\lambda} )\\
        \overset{\cref{lemma: cap product with support for vector bundle}}&{=} (p_1)_*( (\tilde{\iota}_{\alpha})_*[\mathcal{O}_{\Lambda_{\alpha}}] \otimes \iota_Z^*p_2^*  e^{\lambda} ) \\
        \overset{\eqref{eq: projection formula for coherent sheaves and vector bundles}}&{=}(p_1)_*(\tilde{\iota}_{\alpha})_* \tilde{\iota}_{\alpha}^*\iota_Z^*p_2^* e^{\lambda} \\
        \overset{functoriality}&{=} (\iota_{\alpha})_* (p_1)_* p_2^* (\iota_{\alpha})^* e^{\lambda} \\
        \overset{base\text{ }change}&{=} (\iota_{\alpha})_* (\pi_{\alpha})^* (\pi_{\alpha})_* (\iota_{\alpha})^* e^{\lambda}.
    \end{align*}
By \cref{lemma: projection yields demazure-like operator} $(\pi_{\alpha})^* (\pi_{\alpha})_*$ acts as $e^{\lambda} \mapsto \frac{e^{\lambda} - e^{s_{\alpha}(\lambda) - \alpha}}{1-e^{-\alpha}}$. Note that $V^-/(V^- \cap {}^{s_{\alpha}}V^-) \cong V_{-\alpha}$ is one-dimensional of weight $\textbf{q}_V(\alpha)^{-1}e^{-\alpha}$. Thus, by \cref{lemma: formula for inclusion of vector bundles} $(\iota_{\alpha})_* $ acts as multiplication with $1- \textbf{q}_V(\alpha) e^{\alpha}$. Finally, $(\iota_{\alpha})^*$ is given by $e^{\lambda} \mapsto e^{\lambda}$. Hence, we have
\begin{equation*}
    [\mathcal{O}_{\Lambda_{\alpha}}] \star e^{\lambda} =  (1 - \textbf{q}_V(\alpha) e^{\alpha})\frac{e^{\lambda} - e^{s_{\alpha}(\lambda) - \alpha}}{1-e^{-\alpha}}.
\end{equation*}
This implies the claim since $\overline{Z}_{s_{\alpha}} = \Lambda_{\alpha}$ by \cref{lemma: Lambda alpha is closure of Zs}.
\end{proof}
Our next goal is to determine a nice generating set of the algebra $K^{G \times (\mathbb{G}_m)^I}(Z)$. For this, we need the following transversality result.
\begin{lemma}\label{lemma: Zs and Zw intersection transversally}
    Let $w \in W$ and $\alpha \in \Pi$ such that $l(s_{\alpha}w) = l(w) +1$. Then the varieties $Z_{s_{\alpha}} \times \tilde{V}$ and $\tilde{V} \times Z_w$ intersect transversally in $\tilde{V}^3$. Moreover, the projection onto the first and third factor induces an isomorphism
    \begin{equation}\label{eq: iso for intersection when lengths add}
        p_{13} : (Z_{s_{\alpha}} \times \tilde{V}) \cap (\tilde{V} \times Z_w) \overset{\sim}{\rightarrow} Z_{s_{\alpha}w}.
    \end{equation}
\end{lemma}
\begin{proof}
    We first prove that \eqref{eq: iso for intersection when lengths add} is an isomorphism. Note that there is a commutative diagram
    \begin{equation}\label{eq: restriction to intersection on flag variety}
        \begin{tikzcd}
            (Z_{s_{\alpha}} \times \tilde{V}) \cap (\tilde{V} \times Z_w)  \arrow[r, "p_{13}"] \arrow[d, "\pi^3"] & Z_{s_{\alpha}w} \arrow[d,"\pi^2"] \\
            (Y_{s_{\alpha}} \times \mathcal{B}) \cap (\mathcal{B} \times Y_w)  \arrow[r, "p_{13}"]       & Y_{s_{\alpha}w}   .       
            \end{tikzcd}
    \end{equation}
    It is well-known that horizontal the map in the bottom row an isomorphism. Moreover, the vertical maps are vector bundles. Thus, to prove that the map from \eqref{eq: iso for intersection when lengths add} is an isomorphism, it suffices to show that the map in the top row in \eqref{eq: restriction to intersection on flag variety} induces an isomorphism on fibers. By equivariance, it suffices to do this for the fibers over $(\dot{s}_{\alpha}B,eB,\dot{w}B) \in (Y_{s_{\alpha}} \times \mathcal{B}) \cap (\mathcal{B} \times Y_w) $. There the fibers are given by ${}^{s_{\alpha}}V^- \cap V^- \cap {}^w V^-$ and ${}^{s_{\alpha}} V^- \cap {}^{w} V^-$ respectively. These fibers are in fact equal since
    \begin{equation*}
        {}^{s_{\alpha}}\Phi^-  \cap  \Phi^-   \cap {}^w \Phi^-  = ({}^{s_{\alpha}} \Phi^- \backslash \{ \alpha \} )\cap {}^{w} \Phi^- = {}^{s_{\alpha}} \Phi^- \cap {}^{w} \Phi^-
    \end{equation*}
    using that $ \alpha \not\in {}^w \Phi^-$ by the assumption that $l(s_{\alpha} w ) = l(w)+ 1$.

    Next, we show transversality. Note that $Z_{s_{\alpha}} \times \tilde{V} \rightarrow Y_{s_{\alpha}} \times \mathcal{B} $ (resp.  $\tilde{V} \times Z_w \rightarrow Y_{s_{\alpha}} \times \mathcal{B}$) is a subbundle of the vector bundle $\tilde{V}^3 \rightarrow \mathcal{B}^3$ restricted to $Y_{s_{\alpha}} \times \mathcal{B}$ (resp. to $Y_{s_{\alpha}} \times \mathcal{B}$). Clearly $(Y_{s_{\alpha}} \times \mathcal{B})$ and $(\mathcal{B} \times Y_w)$ intersect transversally in $\mathcal{B}^3$. Thus, to prove that $Z_{s_{\alpha}} \times \tilde{V}$ and $\tilde{V} \times Z_w$ intersect transversally, it suffices to show that their fibers intersect transversally in the corresponding fiber of $\tilde{V} \rightarrow \mathcal{B}^3$. By equivariance and the isomorphism in the bottom row of \eqref{eq: restriction to intersection on flag variety} it suffices to prove this for the fiber over $(\dot{s}_{\alpha}B,eB,\dot{w}B) \in (Y_{s_{\alpha}} \times \mathcal{B}) \cap (\mathcal{B} \times Y_w) $. There the transversality claim is equivalent to showing that
    \begin{equation*}
        V^- = (V^- \cap  {}^{s_{\alpha}}V^-) + (V^- \cap  {}^wV^-)
    \end{equation*}
    which follows again from the observation that $- \alpha \in \Phi^- \cap {}^w \Phi^-$.
\end{proof}

\begin{corollary}\label{corollary: generated for equivariant K-theory of Steinberg}
    $K^{G \times (\mathbb{G}_m)^I}(Z)$ is generated as an algebra by $K^{G \times (\mathbb{G}_m)^I}(Z_{\Delta})$ together with the $[\mathcal{O}_{\overline{Z}_{s_{\alpha}}}]$ for $\alpha \in \Pi$.
\end{corollary}
\begin{proof}
    Let $\mathfrak{A}$ be the subalgebra of $K^{G \times (\mathbb{G}_m)^I}(Z)$ generated by $K^{G \times (\mathbb{G}_m)^I}(Z_{\Delta})$ together with the $[\mathcal{O}_{\overline{Z}_{s_{\alpha}}}]$ for $\alpha \in \Pi$. We prove by induction along the length of $w$ that $[\mathcal{O}_{\overline{Z}_w}] \in \mathfrak{A}$ for all $w \in W$. If $l(w)  \in \{ 0, 1 \}$, the claim is true by construction. Now let $w \in W$ and $\alpha \in \Pi$ such that $l(s_{\alpha} w) = l(w) +1$ and assume we have shown the claim for all $y \in W$ with $l(y) \le l(w)$.
    Let
    \begin{align*}
        Z_{s_{\alpha}, w} &:= (\overline{Z}_{s_{\alpha}} \times \tilde{V}) \cap (\tilde{V} \times \overline{Z}_{w}) \\
        Z_{s_{\alpha}, w}^{\circ} &:= (Z_{s_{\alpha}} \times \tilde{V}) \cap (\tilde{V} \times Z_{w}).
    \end{align*}
    The projection onto the first and third factor restricts to a map
    \begin{equation*}
        p_{13} : Z_{s_{\alpha}, w} \rightarrow Z_{\le s_{\alpha}w}.
    \end{equation*}
    Using that $ K^{G \times (\mathbb{G}_m)^I}(Z_{\le s_{\alpha}w}) \subset K^{G \times (\mathbb{G}_m)^I}(Z)$ by \cref{lemma: K-theory of Z is free over the diagonal}, we get
    \begin{equation*}
        [\mathcal{O}_{\overline{Z}_{s_{\alpha}}}] \star [\mathcal{O}_{\overline{Z}_{w}}] \in K^{G \times (\mathbb{G}_m)^I}(Z_{\le s_{\alpha}w}).
    \end{equation*}
    Let
    \begin{equation*}
        U := \tilde{V}^3 \backslash ( (Z_{< s_{\alpha}}  \times \tilde{V}) \cup (\tilde{V} \times Z_{< w}))
    \end{equation*}
    which is an open subset of $\tilde{V}^3$. Note that
    \begin{equation*}
        \begin{aligned}
        U \cap (\overline{Z}_{s_{\alpha}} \times \tilde{V}) & = U \cap (Z_{s_{\alpha}} \times \tilde{V} )\\
        U \cap ( \tilde{V} \times \overline{Z}_w) & = U \cap (  \tilde{V} \times Z_{w}).
        \end{aligned}
    \end{equation*}
    In particular, we have
    \begin{equation*}
        U \cap Z_{s_{\alpha}, w} = U \cap Z_{s_{\alpha}, w}^{\circ} = Z_{s_{\alpha}, w}^{\circ}.
    \end{equation*}
    Thus, $Z_{s_{\alpha}, w}^{\circ} \subset Z_{s_{\alpha}, w}$ is open and the flat pullback
    \begin{equation*}
        \iota_{Z_{s_{\alpha}, w}^{\circ}}^* : K^{G \times (\mathbb{G}_m)^I}(Z_{s_{\alpha}, w}) \rightarrow K^{G \times (\mathbb{G}_m)^I}(Z_{s_{\alpha}, w}^{\circ})
    \end{equation*}
    agrees with the restriction with support map
    \begin{equation*}
        \iota_U^* : K^{G \times (\mathbb{G}_m)^I}(Z_{s_{\alpha}, w}) \rightarrow K^{G \times (\mathbb{G}_m)^I}(Z_{s_{\alpha}, w}^{\circ})
    \end{equation*}
    coming from the inclusion $\iota_U : U \hookrightarrow \tilde{V}^3$. Using this, we can compute the restriction of $[\mathcal{O}_{\overline{Z}_{s_{\alpha}}}] \star [\mathcal{O}_{\overline{Z}_{w}}] \in K^{G \times (\mathbb{G}_m)^I}(Z_{\le s_{\alpha}w})$ to $Z_{s_{\alpha}w} \subset Z_{\le s_{\alpha}w}$:
    \begin{align*}
        \iota_{Z_{s_{\alpha}w}}^* ([\mathcal{O}_{\overline{Z}_{s_{\alpha}}}] \star [\mathcal{O}_{\overline{Z}_{w}}]) \overset{def}&{=} \iota_{Z_{s_{\alpha}w}}^*(p_{13})_* (p_{12}^* [\mathcal{O}_{\overline{Z}_{s_{\alpha}}}]\otimes^L_{\tilde{V}^3} p_{23}^*[\mathcal{O}_{\overline{Z}_{w}}] ) \\
        \overset{base \text{ } change}&{=}(p_{13})_* \iota_{Z_{s_{\alpha}, w}^{\circ}}^*(p_{12}^* [\mathcal{O}_{\overline{Z}_{s_{\alpha}}}]\otimes^L_{\tilde{V}^3} p_{23}^*[\mathcal{O}_{\overline{Z}_{w}}] ) \\
        &= (p_{13})_* \iota_U^*(p_{12}^* [\mathcal{O}_{\overline{Z}_{s_{\alpha}}}]\otimes^L_{\tilde{V}^3} p_{23}^*[\mathcal{O}_{\overline{Z}_{w}}] ) \\
        \overset{\cref{lemma: K theory restriction with supp commutes with tensor product with supp}}&{=} (p_{13})_* (\iota_U^*p_{12}^* [\mathcal{O}_{\overline{Z}_{s_{\alpha}}}]\otimes^L_{U} \iota_U^*p_{23}^*[\mathcal{O}_{\overline{Z}_{w}}] ) \\ 
        &= (p_{13})_* ([\mathcal{O}_{U \cap (\overline{Z}_{s_{\alpha}} \times \tilde{V})}]\otimes^L_{U} [\mathcal{O}_{U \cap ( \tilde{V} \times \overline{Z}_w )}] ) \\
        \overset{\cref{lemma: transversal derived tensor product is trivial}}&{=} (p_{13})_* [\mathcal{O}_{Z_{s_{\alpha},w}^{\circ}}] \\
        \overset{\cref{lemma: Zs and Zw intersection transversally}}&{=} [\mathcal{O}_{Z_{s_{\alpha}w}}].
    \end{align*}
    This implies that
    \begin{equation*}
        [\mathcal{O}_{\overline{Z}_{s_{\alpha}}}] \star [\mathcal{O}_{\overline{Z}_{w}}]  \in [\mathcal{O}_{\overline{Z}_{s_{\alpha}w}}] + K^{G \times (\mathbb{G}_m)^I}(Z_{< s_{\alpha}w}).
    \end{equation*}
    By \cref{lemma: K-theory of Z is free over the diagonal} we have that $K^{G \times (\mathbb{G}_m)^I}(Z_{< s_{\alpha}w})$ is a free $K^{G \times (\mathbb{G}_m)^I}(Z_{\Delta})$-module with basis given by the $[\mathcal{O}_{\overline{Z}_y}]$ with $y < s_{\alpha}w$. Thus, by the induction hypothesis, we get $[\mathcal{O}_{\overline{Z}_{s_{\alpha}w}}] \in \mathfrak{A}$. Hence, $\mathfrak{A}$ contains all $[\mathcal{O}_{\overline{Z}_{w}}]$ for $w \in W$ and thus $\mathfrak{A} = K^{G \times (\mathbb{G}_m)^I}(Z)$ by \cref{lemma: K-theory of Z is free over the diagonal}. This completes the proof.
\end{proof}
We are now ready to prove the main result of this section.
\begin{theorem}\label{theorem: geo rel of affine Hecke algebras with unequal parameters}
    There is an isomorphism of algebras $K^{G \times (\mathbb{G}_m)^I} (Z) \cong \mathcal{H}^{\aff}_{\textbf{q}_V}$.
\end{theorem}
\begin{proof}
    We can identify
    \begin{equation*}
        M_{asph} \overset{\cref{lemma: explicit formulas for antispherical module}}{\cong} \mathcal{A}[X^*(T)]  \overset{\eqref{eq: identification of K(tilde V) with polynomial ring}}{\cong} K^{G \times (\mathbb{G}_m)^I} (\tilde{V}).
    \end{equation*}
    Thus, we can view $M_{apsh}$ both as a $K^{G \times (\mathbb{G}_m)^I} (Z)$-module and an $\mathcal{H}^{\aff}_{\textbf{q}_V}$-module which yields algebra homomorphisms
    \begin{align*}
        \alpha_1 &: K^{G \times (\mathbb{G}_m)^I} (Z)  \rightarrow \End_{\mathcal{A}}(M_{asph}) \\
        \alpha_2 &: \mathcal{H}^{\aff}_{\textbf{q}_V}  \rightarrow \End_{\mathcal{A}}(M_{asph}).
    \end{align*}
    By \cref{prop: formula for convolution with Lambda alpha}, \cref{lemma: convolution along diagonal} and \cref{lemma: explicit formulas for    antispherical module} we have
    \begin{align*}                            \alpha_1([\mathcal{O}_{\overline{Z}_{s_{\alpha}}}]) &= -\alpha_2( T_{s_{\alpha}} + \textbf{q}_V(\alpha)e^{\alpha}) \\
    \alpha_1( (\Delta_{\tilde{V}})_* e^{\lambda}) &= \alpha_2(e^{\lambda})
    \end{align*}
    where $\Delta_{\tilde{V}} : \tilde{V} \hookrightarrow Z$ is the diagonal embedding. Thus, by \cref{corollary: generated for equivariant K-theory of Steinberg} and \cref{lemma: alg properties of aff heck with unequal parmaeters}, $\alpha_1$ and $\alpha_2$ have the same image in $\End_{\mathcal{A}}(M_{asph})$. Moreover, $\alpha_2$ is injective by \cref{lemma: antispherical module is faithful}. Hence, we get a surjective homomorphism of algebras
    \begin{equation*}
        \alpha_2^{-1} \circ \alpha_1 : K^{G \times (\mathbb{G}_m)^I} (Z) \rightarrow \mathcal{H}^{\aff}_{\textbf{q}_V}
    \end{equation*}
    which restricts to an isomorphism on polynomial subalgebras
    \begin{equation*}
        \alpha_2^{-1} \circ \alpha_1: K^{G \times (\mathbb{G}_m)^I} (Z_{\Delta}) \overset{\sim}{\rightarrow} \mathcal{A}[X^*].
    \end{equation*}
    Note that $K^{G \times (\mathbb{G}_m)^I} (Z)$ and $\mathcal{H}^{\aff}_{\textbf{q}_V}$ are free modules of rank $|W|$ over this commutative noetherian subalgebra (by \cref{lemma: alg properties of aff heck with unequal parmaeters,lemma: K-theory of Z is free over the diagonal}). Surjective maps between free modules of the same (finite) rank over a commutative noetherian ring are always isomorphisms (c.f. \cite[Exercise 6.1]{atiyah1994introduction}), so $\alpha_2^{-1} \circ \alpha_1$ is an isomorphism.
\end{proof}
We briefly explain how the construction above specializes to the setting from the introduction. Let $G$ be a simple algebraic group of non simply-laced type defined over an algebraically closed field $k$ of special characteristic. Let $\mathfrak{g}_s \subset \mathfrak{g}$ be a $G$-stable subspace whose non-zero weights are the short roots. Then $V := \mathfrak{g}_s \oplus \mathfrak{g}/\mathfrak{g}_s$ is an $I$-rooted $G$-representation for $I = \{ 1,2 \}$. The associated parameter function is given by
\begin{equation*}
    q_{\textbf{V}}(\alpha) = \begin{cases}
        q_1 & \alpha  \text{ short}, \\
        q_2 & \alpha  \text{ long}.
    \end{cases}
\end{equation*}
\cref{theorem: geo rel of affine Hecke algebras with unequal parameters} then yields an isomorphism
\begin{equation*}
    \mathcal{H}^{\aff}_{q_1, q_2} = \mathcal{H}^{\aff}_{\textbf{q}_V}\cong K^{G \times \mathbb{G}_m \times \mathbb{G}_m}(Z).
\end{equation*}
\section{Simple modules and perverse sheaves}\label{section: simple modules and perverse sheaves}
\subsection{Centralizers and fixed points}
For any variety $Y$ with an action of a diagonalizable group scheme $\mathcal{D}$, we can consider the fixed point locus
\begin{equation*}
    Y^{\mathcal{D}} = \{ y \in Y(k) \mid t_A y_A = y_A \text{ for all } k\text{-algebras } A \text{ and } t_A \in \mathcal{D}(A) \} .
\end{equation*}
Note that $Y^{\mathcal{D}} \subset Y(k)$ is always a closed subset, so we can consider $Y^{\mathcal{D}}$ as a subvariety of $Y$. When $\mathcal{D}$ is reduced, $Y^{\mathcal{D}}$ agrees with the naive fixed point locus $\{y \in Y(k) \mid ty = y \text{ for all } t \in \mathcal{D}(k) \}$.
\begin{remark}
    One can also work with the fixed-point scheme $(Y^{\mathcal{D}})^{sch}$ whose underlying reduced scheme is the subvariety $Y^{\mathcal{D}} \subset Y$ we consider. For our purposes, the difference between the two will not matter, since we eventually always apply $K(-)$, $H_*(-)$ or $D^b_c(-)$ to our fixed-point varieties/schemes which give the same result on the underlying reduced scheme. Note however, that for non-reduced diagonalizable group schemes $\mathcal{D}$ it may happen that $Y^{\mathcal{D}} \neq Y^{\mathcal{D}_{red}}$, so the non-reduced structure on $\mathcal{D}$ is important.
\end{remark}
We need the following well-known fact (c.f \cite[A.8.10]{conrad2015pseudo}).
\begin{lemma}\label{lemma: taking fixed points preserves smoothness}
        Let $Y$ be a smooth variety with an action of a diagonalizable group scheme $\mathcal{D}$. Then $Y^{\mathcal{D}}$ is smooth and $T_y (Y^{\mathcal{D}}) = (T_y Y)^{\mathcal{D}}$ for any $y \in Y^{\mathcal{D}}$.
\end{lemma}
If $\mathcal{D}$ is a closed subgroup scheme of the reductive group $G$, the fixed point locus $G^{\mathcal{D}}$ is the centralizer of $\mathcal{D}$ in $G$. Moreover, the centralizer $G^{\mathcal{D}}$ has the following concrete description.
\begin{lemma}\label{lemma: explicit descrpition of centralizer in reductive group}
    Let $\mathcal{D} \subset G$ be a diagonalizable subgroup scheme. Then $(G^{\mathcal{D}})^{\circ}$ is reductive. Moreover, if $\mathcal{D}$ is contained in the maximal torus $T \subset G$ with $K_{\mathcal{D}} := \ker(X^*(T) \rightarrow X^*(D))$, then
    \begin{align*}
        (G^{\mathcal{D}})^{\circ} & = \langle T, U_{\alpha} \mid \alpha \in \Phi \cap K_{\mathcal{D}} \rangle \\
        G^{\mathcal{D}} &= \langle (G^{\mathcal{{D}}})^\circ , \dot{w} \mid \dot{w} \in N_G(T), w(K_{\mathcal{D}}) = K_{\mathcal{D}} , w|_{X^*(T)/K_{\mathcal{D}}} = \id \rangle.
    \end{align*}
\end{lemma}
\begin{proof}
    The fact that $(G^{\mathcal{D}})^{\circ}$ is reductive is contained in \cite[A.8.12]{conrad2015pseudo}. Now let $\mathcal{D}$ be contained in a maximal torus $T \subset G$. Then $T$ is a maximal torus in $ (G^{\mathcal{D}})^{\circ} $ and by \cref{lemma: taking fixed points preserves smoothness} we have $Lie(G^{\mathcal{D}}) = Lie(G)^{\mathcal{D}}$. Note that $Lie(G)^{\mathcal{D}}$ is spanned by $Lie(T)$ together with the root spaces on which $\mathcal{D}$ acts trivially. These are precisely the root spaces of those $\alpha$ with $\alpha \in K_{\mathcal{D}}$. This proves that $ (G^{\mathcal{D}})^{\circ} = \langle T, U_{\alpha} \mid \alpha \in \Phi \cap K_{\mathcal{D}} \rangle  $. 
    
    Let $g \in G^{\mathcal{D}}$. Then $gTg^{-1} \subset (G^{\mathcal{D}})^{\circ}$ is a maximal torus. Since all maximal tori in a reductive group are conjugate, we can find $h \in (G^{\mathcal{D}})^{\circ}$ such that $hgTg^{-1}h^{-1} = T$ and thus $hg \in N_G(T)$. This shows that $G^{\mathcal{D}}$ is generated by $(G^{\mathcal{{D}}})^\circ $ together with $N_G(T)^{\mathcal{D}}$. An element $\dot{w} \in N_G(T)$ normalizes $\mathcal{D}$ if and only if the corresponding Weyl group element $w \in W$ normalizes $K_{\mathcal{D}}$. Moreover, if $\dot{w}$ normalizes $\mathcal{D}$, then it centralizes $\mathcal{D}$ if and only if $w$ acts trivially on $X^*(\mathcal{D}) \cong X^*(T)/ K_{\mathcal{D}}$. This proves that $G^{\mathcal{D}} = \langle (G^{\mathcal{{D}}})^\circ , \dot{w} \mid \dot{w} \in N_G(T), w(K_{\mathcal{D}}) = K_{\mathcal{D}} , w|_{X^*(T)/K_{\mathcal{D}}} = \id \rangle$.
\end{proof}
\begin{lemma}\cite[Theorem A]{richardson1982orbits}\label{lemma: orbit finiteness for fixed points}
    Let $Y$ be a $G$-variety, $\mathcal{D} \subset G$ a smooth (i.e. reduced) diagonalizable subgroup scheme and $H \subset G$ a closed connected normal subgroup. If $Y$ has finitely many $H$-orbits, then $Y^{\mathcal{D}}$ has finitely many $H^{\mathcal{D}}$-orbits. Moreover, if $Y$ is a single $H$-orbit, then any $H^{\mathcal{D}}$-orbit in $Y^{\mathcal{D}}$ is closed.
\end{lemma}
\begin{remark}
    The assumption that $\mathcal{D}$ is smooth in \cref{lemma: orbit finiteness for fixed points} cannot be dropped (see \cref{example: example where fixed point space has infinitely many orbits}).
\end{remark}

\begin{lemma}\label{lemma: fixed points of flag variety}
    Let $\mathcal{D} \subset T$ be a diagonalizable subgroup scheme. Then for any $gB \in \mathcal{B}^{\mathcal{D}}$ the following are true
    \begin{enumerate}[label = (\roman*)]
        \item There is a maximal torus of $gBg^{-1}$ containing $\mathcal{D}$,
        \item $(gBg^{-1})^{\mathcal{D}}$ is a Borel subgroup of $(G^{\mathcal{D}})^{\circ}$,
        \item Each connected component of $\mathcal{B}^{\mathcal{D}}$ is $(G^{\mathcal{D}})^{\circ}$-equivariantly isomorphic to the flag variety of $(G^{\mathcal{D}})^\circ$.
    \end{enumerate}
\end{lemma}
\begin{proof}
    Let $gB \in \mathcal{B}$. By the Bruhat decomposition we may assume $g = u\dot{w} \in G$ with $u \in \prod_{\alpha \in \Phi^- \cap {}^w \Phi^+} U_{\alpha}$. Then $gB \in \mathcal{B}^{\mathcal{D}}$ if and only if $ u \in G^{\mathcal{D}}$. Hence, $\mathcal{D} = u\mathcal{D}u^{-1}$ is contained in the maximal torus $S:= uTu^{-1} = u\dot{w} T(u\dot{w})^{-1} \subset gBg^{-1}$ which proves the first claim.
    
    Using the first claim, it suffices to prove the second claim for $gBg^{-1} = B$. Let $K_{\mathcal{D}} := ker(X^*(T) \rightarrow X^*(\mathcal{D}))$. We then have
    \begin{equation*}
        B^{\mathcal{D}} \cong (T \times \prod_{\alpha \in \Phi^-} U_{\alpha})^{\mathcal{D}} = T \times  \prod_{\alpha \in \Phi^- \cap K_{\mathcal{D}}} U_{\alpha}.
    \end{equation*}
    It now follows from \cref{lemma: explicit descrpition of centralizer in reductive group} that $B^{\mathcal{D}}$ is a Borel subgroup of $(G^{\mathcal{D}})^{\circ}$. This proves the second claim.

    The orbit map gives rise to a bijective morphism
    \begin{equation}\label{eq: orbit is flag variety}
        (G^{\mathcal{D}})^{\circ} /B^{\mathcal{D}} \rightarrow (G^{\mathcal{D}})^{\circ} \cdot B/B \subset  \mathcal{B}^{\mathcal{D}}.
    \end{equation}
    One can easily check that this is an isomorphism: By equivariance it suffices to check this after replacing $(G^{\mathcal{D}})^{\circ}$ with the open subset $B^{\mathcal{D}} \dot{w}_0 B^{\mathcal{D}} \subset (G^{\mathcal{D}})^{\circ}$ where $w_0$ is the longest element of the Weyl group of $(G^{\mathcal{D}})^{\circ}$. Then both sides naturally identify with $\prod_{\alpha \in \Phi^- \cap K_{\mathcal{D}}}U_{\alpha}$ and the map is the identity. By \cref{lemma: taking fixed points preserves smoothness} we have $T_{B} \mathcal{B}^{\mathcal{D}} =(\mathfrak{n}^+)^{\mathcal{D}}$. Note that $(\mathfrak{n}^+)^{\mathcal{D}}$ is also the positive part of $Lie((G^{\mathcal{D}})^{\circ})$ by \cref{lemma: explicit descrpition of centralizer in reductive group}. Thus, using the isomorphism from \eqref{eq: orbit is flag variety}, we see that the $(G^{\mathcal{D}})^{\circ}$-orbit through $B \in \mathcal{B}^{\mathcal{D}}$ has the same dimension as $\mathcal{B}^{\mathcal{D}}$. The same argument applies to all other orbits, which shows that the orbits all have full dimension. Thus, they have to be the connected components. Note that by \eqref{eq: orbit is flag variety} the orbits are $(G^{\mathcal{D}})^{\circ}$-equivariantly isomorphic to the flag variety of $(G^{\mathcal{D}})^{\circ}$.
\end{proof}

\subsection{The affine Hecke algebra at a central character}
For any algebraic group (or diagonalizable group scheme) $H$, we write $R(H) = K^H(pt)$ for the representation ring. Let $V$ be an $I$-rooted $G$-representation. Recall from \cref{theorem: geo rel of affine Hecke algebras with unequal parameters} that there is an isomorphism $K^{G\times (\mathbb{G}_m)^I}(Z) \cong \mathcal{H}^{\aff}_{\textbf{q}_V}$. Under this isomorphism, the center of $\mathcal{H}^{\aff}_{\textbf{q}_V}$ corresponds to
\begin{equation*}
    R(G \times (\mathbb{G}_m)^I) = \mathbb{Z}[X^*(\hat{T})]^W = Z(\mathcal{H}^{\aff}_{\textbf{q}_V})
\end{equation*}
where
\begin{equation*}
    \hat{T} := T \times (\mathbb{G}_m)^I.
\end{equation*}
Setting $\mathcal{T} := Spec(\mathbb{C}[X^*(T)])$ and 
\begin{equation*}
    \hat{\mathcal{T}} := Spec(\mathbb{C}[X^*(\hat{T})]) = \mathcal{T} \times (\mathbb{C}^{\times})^I
\end{equation*}
we can identify
\begin{equation*}
    \mathbb{C} \otimes_{\mathbb{Z}} Z(\mathcal{H}^{\aff}_{\textbf{q}_V}) = \mathbb{C}\otimes_{\mathbb{Z}} \mathcal{A}[X^*(T)]^W \cong \mathcal{O}(\hat{\mathcal{T}})^W \cong \mathcal{O}(\hat{\mathcal{T}}/W).
\end{equation*}
This yields a bijection
\begin{align*}
    \hat{\mathcal{T}}/W  &\overset{1:1}{\leftrightarrow} \{ \text{Central characters } Z(\mathcal{H}^{\aff}_{\textbf{q}_V})\rightarrow \mathbb{C} \} \\
    a&\mapsto \chi_a.
\end{align*}
We denote the $1$-dimensional $Z(\mathcal{H}^{\aff}_{\textbf{q}_V})$-module corresponding to $\chi_a$ by $\mathbb{C}_a$. The goal of this section is to give a geometric description of
\begin{equation*}
    \mathcal{H}^{\aff}_{\chi_a} := \mathcal{H}^{\aff}_{\textbf{q}_V} \otimes_{Z(\mathcal{H}^{\aff}_{\textbf{q}_V})} \mathbb{C}_a = (\mathbb{C} \otimes_{\mathbb{Z}} \mathcal{H}^{\aff}_{\textbf{q}_V})/(\ker(\chi_a)).
\end{equation*}

For this, we will need the following consequence of the Pittie-Steinberg theorem \cite{steinberg1975theorem}.
\begin{lemma}\label{lemmma: Kunneth formula for flag variety}
    Assume that $G$ has simply connected derived subgroup. Then there is an isomorphism of $R(T)$-modules $ R(T) \otimes_{R(G)}K^{G}(\mathcal{B}) \cong K^{T}(\mathcal{B})$ induced by restriction along $T \subset G$.
    
\end{lemma}
\begin{proof}
    The proof in characteristic $0$ in \cite[§6.1]{chriss2009representation} applies word by word in positive characteristic.
\end{proof}
For any $a \in \hat{\mathcal{T}}$ we get a subgroup
\begin{equation*}
    X^*_a := \{ \lambda \in X^*(\hat{\mathcal{T}}) \mid \lambda(a) = 1 \}\subset X^*(\hat{\mathcal{T}}) \cong X^*(\hat{T})
\end{equation*}
and a corresponding diagonalizable subgroup scheme
\begin{equation*}
    \hat{T}_a = Spec(k[X^*(\hat{T})/X^*_a])\subset \hat{T}.
\end{equation*}
\begin{lemma}\label{lemma: restriction of K theory on Z to Ta}
    Assume that $G$ has simply connected derived subgroup. Then restriction along $ \hat{T}_a \hookrightarrow G \times (\mathbb{G}_m)^I$ induces an isomorphism of $R(\hat{T}_a)$-algebras $ R(\hat{T}_a) \otimes_{R(G\times (\mathbb{G}_m)^I)} K^{G \times (\mathbb{G}_m)^I}(Z) \cong K^{\hat{T}_a}(Z)$.
\end{lemma}
\begin{proof}
    Recall from \cref{lemma: K-theory of Z is free over the diagonal} that $K^{\hat{T}_a}(Z)$ (resp. $K^{G \times (\mathbb{G}_m)^I}(Z)$) is a free $K^{\hat{T}_a}(\Delta_Z)$-module (resp. $K^{G \times (\mathbb{G}_m)^I}(\Delta_Z)$-module) with basis $\{ [\mathcal{O}_{\overline{Z}_w}]\}$. Hence, it suffices to show that the canonical map
    \begin{equation*}
        R(\hat{T}_a) \otimes_{R(G\times (\mathbb{G}_m)^I)} K^{G \times (\mathbb{G}_m)^I}(\Delta_Z) \rightarrow K^{\hat{T}_a}(\Delta_Z)
    \end{equation*}
    is an isomorphism. Since $\Delta_Z \cong \tilde{V} \rightarrow \mathcal{B}$ is a vector bundle, the Thom isomorphism reduces this to showing that the canonical map
    \begin{equation*}
        R(\hat{T}_a) \otimes_{R(G\times (\mathbb{G}_m)^I)} K^{G \times (\mathbb{G}_m)^I}(\mathcal{B}) \rightarrow K^{\hat{T}_a}(\mathcal{B})
    \end{equation*}
    is an isomorphism. For this it suffices to show that the maps
    \begin{align*}
        R(\hat{T}) \otimes_{R(G\times (\mathbb{G}_m)^I)} K^{G \times (\mathbb{G}_m)^I}(\mathcal{B}) &\rightarrow K^{\hat{T}}(\mathcal{B}) \\
        R(\hat{T}_a) \otimes_{R(\hat{T})} K^{\hat{T}}(\mathcal{B}) &\rightarrow K^{\hat{T}_a}(\mathcal{B})
    \end{align*}
    are isomorphisms. The first map is an isomorphism by \cref{lemmma: Kunneth formula for flag variety}. The second is an isomorphism since $K^{\hat{T}}(\mathcal{B})$ and $K^{\hat{T}_a}(\mathcal{B})$ are both free modules over $R(\hat{T})$ (resp. $R(\hat{T}_a)$) with the Schubert basis $\{ [\mathcal{O}_{\overline{B\dot{w}B/B}}] \mid w \in W \}$.
\end{proof}

We define
\begin{equation*}
    \mathfrak{N}(V) := im (\tilde{V} \rightarrow V).
\end{equation*}
Note that $\mathfrak{N}(V) \subset V$ is closed since $ \tilde{V} \rightarrow V$ is proper.
\begin{lemma}\label{lemma: basic results about V le 0}
    Let $V^{\le 0} := V^- \oplus V_0$. Then we have:
    \begin{enumerate}[label = (\roman*)]
        \item $V = G \cdot V^{\le 0}$;
        \item $ \mathfrak{N}(V) \cap V^{\le 0} = V^-$.
    \end{enumerate}
\end{lemma}
\begin{proof}
    Consider the map $\tilde{\mu}: G \times^B V^{\le 0} \rightarrow V$, $(g,v) \mapsto gv$. Let $v \in V^- = \bigoplus_{\alpha \in \Phi^+} V_{- \alpha}$ be an element with non-zero component in $V_{-\alpha}$ for each simple root $\alpha$ and a zero component in each $V_{-\alpha}$ for $\alpha \in \Phi^+$ not simple. Let $g \in G$ such that $gv \in V^{\le 0}$ and write $g = b\dot{w} b'$ with $b,b' \in B$ and $\dot{w} \in N_G(T)$. Then $\dot{w}b'v \in V^{\le 0}$. Now $b'v \in V^{\le 0}$ has a non-zero component in $V_{-\alpha}$ for each simple $\alpha$ and thus $w(-\alpha) \in \Phi^-$ for all simple $\alpha$. This implies $w = e$ and thus $g \in B$. Hence, $\tilde{\mu}^{-1}(v) = \{ eB\}$ is a singleton. By a standard theorem on fiber dimensions of dominant morphisms (c.f. \cite[\href{https://stacks.math.columbia.edu/tag/0B2L}{Tag 0B2L}]{stacks-project}), we get
    \begin{equation*}
        0 = \dim \tilde{\mu}^{-1}(v) \ge \dim G \times ^B V^{\le 0} -  \dim G \cdot V^{\le 0}.
    \end{equation*}
    Since $\dim G \times ^B V^{\le 0} = \dim V^{\le 0} + |\Phi^+| = \dim V$, this implies  $\dim  G \cdot V^{\le 0}  \ge \dim V$. By the irreducibility of $V$ this implies $G \cdot V^{\le 0} =  V$.

    Clearly, $V^- \subset \mathfrak{N}(V) \cap V^{\le 0}$. Let $v \in \mathfrak{N}(V) \cap V^{\le 0}$. Then we can find $g \in G$ such that $gv \in V^-$. Let $g = b\dot{w}b' \in G$ with $b,b' \in B$. Then $b'v = \dot{w}^{-1} b^{-1} gv \in {}^{w^{-1}} V^- \cap V^{\le 0} $. But ${}^{w^{-1}}V^- \cap V^{\le 0} \subset V^-$ since ${}^{w^{-1}}V^- = \bigoplus_{\alpha \in \Phi^+} V_{- w^{-1}(\alpha)}$. Hence, $b'v \in V^-$ and thus also $v\in V^-$. This shows that $V^- = \mathfrak{N}(V) \cap V^{\le 0}$.
\end{proof}
Taking $\hat{T}_a$-fixed points in $\mu : \tilde{V} \rightarrow \mathfrak{N}(V)$, we obtain a map $\mu^a : \tilde{V}^{\hat{T}_a} \rightarrow \mathfrak{N}(V)^{\hat{T}_a}$. We can then consider
\begin{equation}\label{eq: def of fixed point Springer sheaf}
    \textbf{S}^{a} := \mu^a_* \textbf{1}_{\tilde{V}^{\hat{T}_a}} \in D^b_c(\mathfrak{N}(V)^{\hat{T}_a})
\end{equation}
where $D^b_c(-)$ denotes the constructible derived category (see \cref{section: app Borel Moore}). Let $\Perv_{G^{\hat{T}_a}}(\mathfrak{N}(V)^{\hat{T}_a})$ be the abelian category of $G^{\hat{T}_a}$-equivariant perverse sheaves on $\mathfrak{N}(V)^{\hat{T}_a}$. The map $\mu^a$ is proper, so $ \textbf{S}^{a} $ is a $G^{\hat{T}_a}$-equivariant semisimple complex, i.e. it can be written as a direct sum of shifts of elements in $\Irr(\Perv_{G^{\hat{T}_a}}(\mathfrak{N}(V)^{\hat{T}_a}))$. We set
\begin{equation*}
    \Irr(\textbf{S}^a) := \left\{  \mathcal{F} \in \Irr(\Perv_{G^{\hat{T}_a}}(\mathfrak{N}(V)^{\hat{T}_a})) \mid \substack{ \mathcal{F}[n] \text{ is a direct summand} \\\text{of } \textbf{S}^a \text{ for some } n\in \mathbb{Z}} \right\}.
\end{equation*}
By \cite[Thm 8.6.12]{chriss2009representation} there is a canonical bijection
\begin{equation}\label{eq: irreps of Ext algebra are simple perv sheaves}
    \Irr(\Hom_{D^b_c(\mathfrak{N}(V)^{\hat{T}_a})}^* (\textbf{S}^a, \textbf{S}^a)) \overset{1:1}{\leftrightarrow} \Irr(\textbf{S}^a).
\end{equation}

\begin{theorem}\label{thm: affine Hecke algebra at central character is Ext algebra}
    Let $G$ be a reductive group and $V$ an $I$-rooted $G$-representation. Assume that $G$ has simply connected derived subgroup. Then, for any $a \in \hat{\mathcal{T}}$, there is an isomorphism of $\mathbb{C}$-algebras
    \begin{equation*}
        \mathcal{H}^{\aff}_{\chi_a} \cong  H_*(Z^{\hat{T}_a}) \cong \Hom_{D^b_c(\mathfrak{N}(V)^{\hat{T}_a})}^* (\textbf{S}^a, \textbf{S}^a).
    \end{equation*}
    In particular, there is a canonical bijection
    \begin{equation*}
        \Irr(\mathcal{H}^{\aff}_{\chi_a}) \overset{1:1}{\longleftrightarrow} \Irr(\textbf{S}^a).
    \end{equation*}
\end{theorem}
\begin{proof}
    Note that
    \begin{equation}\label{eq: k theory for trivial action}
        K^{\hat{T}_a}(Z^{\hat{T}_a}) \cong  K(Z^{\hat{T}_a}) \otimes_{\mathbb{Z}} R(\hat{T}_a)
    \end{equation}
    since $\hat{T}_a$ acts trivially on $ Z^{\hat{T}_a}$ (see \cite[Lemme  5.6]{thomason1986lefschetz}). Moreover, we claim that algebra homomorphism
    \begin{equation}\label{eq: RR map for Z fixed point space}
        \mathbb{C} \otimes_{\mathbb{Z}} K(Z^{\hat{T}_a}) \overset{RR}{\rightarrow}H_*(Z^{\hat{T}_a})
    \end{equation}    
    from \cref{lemma: convolution compatible with Chern character} is an isomorphism. By \cref{thm: chern character} this can be reduced to showing that the cycle class map $c_{Z^{\hat{T}_a}}: \mathbb{C} \otimes_{\mathbb{Z}} A(Z^{\hat{T}_a}) \rightarrow H_*(Z^{\hat{T}_a})$ is an isomorphism. We will show in \cref{lemma: connnected components of fixed point resolution} that each connected component of $\tilde{V}^{\hat{T}_a}$ is of the form $(G^{\hat{T}_a})^{\circ} \times^{B^{\hat{T}_a}} (V^-)^{\hat{T}_a}$. It then follows from \cite[Prop. 2.7]{antor2024formality} that $c_{Z^{\hat{T}_a}}$ is an isomorphism (see also \cite[§6.2]{chriss2009representation}). Also, note that the central character $\chi_a : R(G \times (\mathbb{G}_m)^I) \rightarrow \mathbb{C}_a$ factors through a map $R(\hat{T}_a) = \mathbb{Z}[X^*(\hat{\mathcal{T}})/X^*_a]\overset{ev_a}{\rightarrow} \mathbb{C}_a$. Using this, we get 
    \begin{align*}
        \mathcal{H}^{\aff}_{\chi_a} \overset{def}&{=} \mathcal{H}^{\aff}_{\textbf{q}_V} \otimes_{Z(\mathcal{H}^{\aff}_{\textbf{q}_V})} \mathbb{C}_a \\
        \overset{\cref{theorem: geo rel of affine Hecke algebras with unequal parameters}}&{\cong} K^{G \times (\mathbb{G}_m)^I}(Z) \otimes_{R(G\times (\mathbb{G}_m)^I)} \mathbb{C}_a \\
        &\cong K^{G \times (\mathbb{G}_m)^I}(Z) \otimes_{R(G\times (\mathbb{G}_m)^I)} R(\hat{T}_a) \otimes_{R(\hat{T}_a)} \mathbb{C}_a \\
        \overset{\cref{lemma: restriction of K theory on Z to Ta}}&{\cong} K^{\hat{T}_a}(Z) \otimes_{R(\hat{T}_a)} \mathbb{C}_a \\
        \overset{\cref{lemma: convolution compatible with K theory localization}}&{\cong} K^{\hat{T}_a}(Z^{\hat{T}_a}) \otimes_{R(\hat{T}_a)} \mathbb{C}_a \\
        &\overset{\eqref{eq: k theory for trivial action}}{\cong}  K(Z^{\hat{T}_a}) \otimes_{\mathbb{Z}} R(\hat{T}_a) \otimes_{R(\hat{T}_a)} \mathbb{C}_a\\
        &\cong  K(Z^{\hat{T}_a}) \otimes_{\mathbb{Z}} \mathbb{C}\\
        \overset{\eqref{eq: RR map for Z fixed point space}}&{\cong} H_*(Z^{\hat{T}_a}) .
    \end{align*}
    The isomorphism $H_*(Z^{\hat{T}_a}) \cong \Hom_{D^b_c(\mathfrak{N}(V)^{\hat{T}_a})}^* (\textbf{S}^a, \textbf{S}^a)$ follows from \eqref{eq: identification of BM homology and Ext algebra} and the parameterization of simples follows from \eqref{eq: irreps of Ext algebra are simple perv sheaves}.
\end{proof}
\begin{remark}
    The correspondence from \cref{thm: affine Hecke algebra at central character is Ext algebra} can also be lifted to an equivalence of triangulated categories. More specifically, it follows from \cite[Theorem 4.26, Corollary 5.6]{antor2024formality} that the full triangulated subcategory of $D^b_c(\mathfrak{N}(V)^{\hat{T}_a}) $ generated by the direct summands of $\textbf{S}^a$ is equivalent to the perfect derived category of dg-modules $D_{\text{perf}}(\mathcal{H}^{\aff}_{\chi_a}-\text{dgMod})$ where we view $\mathcal{H}^{\aff}_{\chi_a} \cong \Hom_{D^b_c(\mathfrak{N}(V)^{\hat{T}_a})}^* (\textbf{S}^a, \textbf{S}^a)$ as a dg-algebra with vanishing differential and the $\Hom^*$-grading.
\end{remark}

\section{A Deligne-Langlands correspondence for unequal parameters}\label{section: a dl correspodence with unequal parameters}
Throughout this section, we fix an $I$-rooted $G$-representation $V$.
\subsection{The geometric conditions}\label{section: geometric conditions}
Let $a = (s, (t_i)_{i \in I} ) \in \hat{\mathcal{T}}$. By \cref{thm: affine Hecke algebra at central character is Ext algebra} parameterizing the simple $\mathcal{H}^{\aff}$-modules with central character $\chi_a$ comes down to determining which simple perverse sheaves occur in $\textbf{S}^a \in D^b_c(\mathfrak{N}(V)^{\hat{T}_a})$. A favorable situation for this is when there are only finitely many $G^{\hat{T}_a}$-orbits in $\mathfrak{N}(V)^{\hat{T}_a}$. In this case, there is a bijection
\begin{align*}
    \Irr(\Perv_{G^{\hat{T}_a}}(\mathfrak{N}(V)^{\hat{T}_a})) &\overset{1:1}{\leftrightarrow} \{(x,\rho) \mid x \in \mathfrak{N}(V)^{\hat{T}_a},  \rho \in \Irr( A(a,x))\}/G^{\hat{T}_a} \\
    IC(G^{\hat{T}_a} \cdot x, \rho) & \mapsto (x,\rho)
\end{align*}
where
\begin{equation*}
    A(a,x) =A_G(a,x):= G_x^{\hat{T}_a} / (G_x^{\hat{T}_a})^\circ.
\end{equation*}
Let
\begin{equation*}
    \mathcal{B}_x := \mu^{-1}(x) \cong \{ gB \in \mathcal{B} \mid g^{-1}x \in V^-\}
\end{equation*}
be the `exotic Springer fiber' of $x$. Then there is a canonical $A(a,x)$-action on $H_*(\mathcal{B}_x^{\hat{T}_a})$ and we can consider the set
\begin{equation*}
    \mathcal{P}(G, V, a) := \{(x,\rho) \mid x \in \mathfrak{N}(V)^{\hat{T}_a},  \rho \in \Irr( A(a,x)),  H_*(\mathcal{B}_x^{\hat{T}_a})_{\rho} \neq 0\}/G^{\hat{T}_a}
\end{equation*}
where for any $A(a,x)$-representation $M$, we write
\begin{equation*}
    M_{\rho} := \Hom_{A(a,x)}(\rho, M).
\end{equation*}
For each $(x,\rho) \in \mathcal{P}(G, V, a)$ there is a canonical $H_*(Z^{\hat{T}_a})$-module structure on
\begin{equation*}
    M(a,x,\rho) := H_*(\mathcal{B}_x^{\hat{T}_a})_{\rho}
\end{equation*}
via convolution. Using the isomorphism $\mathcal{H}^{\aff}_{\chi_a} \cong H_*(Z^{\hat{T}_a})$ from \cref{thm: affine Hecke algebra at central character is Ext algebra}, we can view $M(a,x,\rho)$ as an $\mathcal{H}^{\aff}_{\chi_a} $-module. The $M(a,x,\rho)$ are also called `standard modules'.
\begin{lemma}\label{lemma: description of DL corr with standard modules}
    If there are only finitely many $G^{\hat{T}_a}$-orbits on $\mathfrak{N}(V)^{\hat{T}_a}$, then there is a canonical injection
    \begin{equation}\label{eq: injective DL parametrization}
        \begin{aligned}
            \Irr(\mathcal{H}^{\aff}_{\chi_a}) &\hookrightarrow \mathcal{P}(G, V, a) \\
            L &\mapsto (x_L, \rho_L). 
        \end{aligned}    
        \end{equation}
        The pair $(x_L, \rho_L)$ is uniquely determined by the following two properties:
        \begin{enumerate}[label = (\roman*)]
            \item $L$ appears as a composition factor of $M(a,x_L,\rho_L)$,
            \item If $L$ appears as a composition factor of $M_{a,x',\rho'}$, then $x' \in \overline{G^{\hat{T}_a} \cdot x_L}$.
        \end{enumerate}
\end{lemma}
\begin{proof}
    By base change there is an isomorphism of $A(a,x)$-modules $H^*(\iota_x^!\textbf{S}^a)  \cong H_*(\mathcal{B}_x^{\hat{T}_a})$. Moreover, there are isomorphisms of $A(a,x)$-modules
    \begin{align*}
    H^{ \dim G^{\hat{T}_a}\cdot x }(\iota_x^! IC(G^{\hat{T}_a}\cdot x, \rho))& \cong H^{ \dim G^{\hat{T}_a}\cdot x } (\mathbb{D} \iota_x^* \mathbb{D} IC(G^{\hat{T}_a}\cdot x, \rho)) \\
    &\cong H^{- \dim G^{\hat{T}_a}\cdot x }( \iota_x^* IC(G^{\hat{T}_a}\cdot x, \rho^{\vee}))^{\vee} \\
    &\cong (\rho^{\vee})^{\vee} \\
    &\cong \rho .
    \end{align*}
    Thus, if $ IC(G^{\hat{T}_a}\cdot x, \rho)$ appears in $\textbf{S}^a$, then $\rho$ appears in $H_*(\mathcal{B}_x^{\hat{T}_a})$. It now follows from \cref{thm: affine Hecke algebra at central character is Ext algebra} that there is an injection $\Irr(\mathcal{H}^{\aff}_{\chi_a}) \hookrightarrow \mathcal{P}(G, V, a)$. The second claim follows from \cite[Thm. 8.6.23]{chriss2009representation}
\end{proof}
\begin{remark}
    It can be shown that $L$ appears as a quotient of the standard module $M(a,x_L,\rho_L)$ using the filtration constructed in the proof of \cite[Thm. 8.6.23]{chriss2009representation}.
\end{remark}
Note that for $a_e = (e, (1)_{i \in I}) \in \hat{\mathcal{T}}$, we have $\hat{T}_{a_e} = \{ e\}$. Moreover, it follows from the definition of the affine Hecke algebra that
\begin{equation*}
    \mathcal{H}^{\aff}_{\chi_{a_e}} \cong \mathbb{C}[W]  \# ( \mathbb{C}[X^*(T)] \otimes_{\mathbb{C}[X^*(T)]^W} \mathbb{C}_e) = \mathbb{C}[W] \# H^*(\mathcal{B}). 
\end{equation*}
Thus, \cref{thm: affine Hecke algebra at central character is Ext algebra} yields an isomorphism (see also \cite{douglass2009homology,kwon2009borel})
\begin{equation}\label{eq: affine Hecke algebra at trivial central character}
    \mathbb{C}[W] \# H^*(\mathcal{B})  \cong H_*(Z).
\end{equation}
We also write
\begin{equation*}
    A(x) := A(a_e, x)=G_x/(G_x)^{\circ}
\end{equation*}
and consider the following conditions for the pair $(G,V)$:
\begin{enumerate}
    \item[(A1)] $G$ acts on $\mathfrak{N}(V)$ with finitely many orbits.
    \item[(A2)] $\# \Irr(W) = \# \mathcal{P}(G,V,a_e)$.
    \item [(A3)] $H_{odd}(\mathcal{B}_x) = 0$ for all $x \in \mathfrak{N}(V)$.
\end{enumerate}
Note that the maximal semisimple quotient of $\mathbb{C}[W] \# H^*(\mathcal{B})$ is $\mathbb{C}[W]$. Hence, combining \eqref{eq: affine Hecke algebra at trivial central character} and \eqref{eq: injective DL parametrization} with (A1) and (A2) yield the following 'exotic Springer correspondence'.
\begin{corollary}\label{cor: exotic springer correspondence}
    If (A1) and (A2) hold for $(G,V)$, then there is a canonical bijection
    \begin{equation*}
    \Irr(W) \overset{1:1}{\longleftrightarrow} \mathcal{P}(G,V,a_e).
\end{equation*} 
\end{corollary}
We will also sometimes consider the condition:
\begin{enumerate}
    \item[(A2')] $A(x)$ acts trivially on $H_*(\mathcal{B}_x)$ for all $x \in \mathfrak{N}(V)$ and $\# \Irr(W) = \# \mathfrak{N}(V)/G$.
\end{enumerate}
Clearly (A2') implies (A2). Under the stronger condition (A2') we can identify $\mathcal{P}(G,V,a_e)$ with $\mathfrak{N}(V)/G $, so the exotic Springer correspondence from \cref{cor: exotic springer correspondence} becomes
\begin{equation*}
    \Irr(W) \overset{1:1}{\longleftrightarrow} \mathfrak{N}(V)/G.
\end{equation*}
The following two lemmas show that conditions (A1)-(A3) and (A2') are invariant under isogenies and certain changes of representations.
\begin{lemma}\label{lemma: Geometric conditions compatible with isogeny}
    \begin{enumerate}[label = (\roman*)]
        \item Let $f: G' \rightarrow G$ be a central isogeny. Then each of (A1), (A2), (A3) and (A2') holds for $(G',V)$ if and only if it holds for $(G,V)$.
        \item Let $\mathcal{D}(G)$ be the derived subgroup of $G$. Then each of (A1), (A2), (A3) and (A2') holds for $(G,V)$ if and only if it holds for $(\mathcal{D}(G),V)$.
    \end{enumerate}
\end{lemma}
\begin{proof}
    Since the $G'$ action on $\mathfrak{N}(V)$ factors through $G$, (A1) for $(G,V)$ and $(G',V)$ are equivalent. For any $x \in \mathfrak{N}(V)$, the isogeny $f$ restricts to a surjection $f: G_x \rightarrow G_{x'}$ and thus to a surjection $A(x) \rightarrow A(x')$. Moreover, the canonical map on flag varieties $\mathcal{B}(G) \rightarrow \mathcal{B}(G')$ is an isomorphism which restricts to an isomorphism $\mathcal{B}(G)_x \cong \mathcal{B}(G')_x$ for each $x \in \mathfrak{N}(V)$. It follows from this that $\# \mathcal{P}(G,V,a_e) = \# \mathcal{P}(G',V,a_e') $. This shows that (A2), (A3) and (A2') for $(G,V)$ and $(G',V)$ are equivalent.

    Since all roots vanish on $Z(G)$, we see that $Z(G)$ acts trivially on $V^-$ and thus also on $\mathfrak{N}(V)$. By the surjectivity of $\mathcal{D}(G) \times Z(G)^{\circ} \rightarrow G$, this implies that (A1) for $(G,V)$ and $(\mathcal{D}(G),V)$ are equivalent. We also get that $G_x = Z(G)^{\circ} \mathcal{D}(G)_x $ which implies that $A_{G}(x) \cong A_{\mathcal{D}(G)}(x)$. Moreover, the canonical map $\mathcal{B}(\mathcal{D}(G)) \rightarrow \mathcal{B}(G)$ is an isomorphism which restricts to an isomorphism $\mathcal{B}(G)_x \cong \mathcal{B}(\mathcal{D}(G))_x$ for each $x \in \mathfrak{N}(V)$. Hence, $\# \mathcal{P}(G,V,a_e) = \# \mathcal{P}(\mathcal{D}(G),V,a_e) $. It follows from this that (A2), (A3) and (A2') for $(G,V)$ and $(\mathcal{D}(G),V)$ are equivalent.
\end{proof}
\begin{lemma}\label{lem: conditions Ai are equivalent for rooted morphisms}
    Let $f:V_1 \rightarrow V_2$ be a morphism of rooted $G$-representations which restricts to an isomorphism $(V_1)_{\alpha} \overset{\sim}{\rightarrow} (V_2)_{\alpha}$ for all $\alpha \in \Phi$. Then $f$ restricts to a $G$-equivariant bijection $\mathfrak{N}(V_1) \rightarrow \mathfrak{N}(V_2)$. Moreover, each of (A1), (A2), (A3) and (A2') holds for $(G,V_1)$ if and only if it holds for $(G,V_2)$.
\end{lemma}
\begin{proof}
    Since $f$ restricts to an isomorphism $V_1^- \rightarrow V_2^-$, the $G$-equivariant map $f: \mathfrak{N}(V_1) \rightarrow \mathfrak{N}(V_2)$ is surjective. To prove injectivity assume that we have $x, x' \in \mathfrak{N}(V_1)$ with $f(x) = f(x')$. We can then find $g \in G$ such that $gx \in V_1^-$. It then follows that $f( gx') = gf(x') = gf(x) = f(gx) \in V_2^-$ and thus
    \begin{equation*}
        gx' \in f^{-1}(V_2^-) \cap \mathfrak{N}(V_1) \subset  V_1^{\le 0} \cap \mathfrak{N}(V_1) \overset{\cref{lemma: basic results about V le 0}}{=} V_1^-.
    \end{equation*}
    Since $f$ restricts to an isomorphism $V_1^- \rightarrow V_2^-$, this implies $gx = gx'$ and thus $x = x'$. This proves the injectivity of $\mathfrak{N}(V_1) \rightarrow \mathfrak{N}(V_2)$.
    
    If $x \in \mathfrak{N}(V_1)$ such that $g^{-1} f(x) \in V_2^-$, then we get by a similar argument as above
    \begin{equation*}
        g^{-1}x \in f^{-1}( V_2^{-1}) \cap \mathfrak{N}(V_1) \subset V_1^{\le 0 } \cap \mathfrak{N}(V_1) \overset{\cref{lemma: basic results about V le 0}}{=} V_1^-.
    \end{equation*}
    This shows that $\mathcal{B}_x = \mathcal{B}_{f(x)}$ for any $x \in \mathfrak{N}(V_1)$. In particular, we have $\# \mathcal{P}(G,V_1,a_e) = \# \mathcal{P}(G,V_2,a_e)$. It follows from these results that (A1), (A2), (A3) and (A2') for $(G,V_1)$ and $(G,V_2)$ are equivalent.
\end{proof}

\subsection{Positive real central characters}
The complex torus $\hat{\mathcal{T}}$ has a polar decomposition
\begin{equation*}
    \hat{\mathcal{T}}= \hat{\mathcal{T}}_{S^1} \times \hat{\mathcal{T}}_{\mathbb{R}_{>0}}
\end{equation*}
where $\hat{\mathcal{T}}_{S^1} := S^1 \otimes_{\mathbb{Z}} X_*(\hat{\mathcal{T}})$ and $\hat{\mathcal{T}}_{\mathbb{R}_{>0}} := \mathbb{R}_{>0} \otimes_{\mathbb{Z}} X_*(\hat{\mathcal{T}})$. Thus, for any $a \in \hat{\mathcal{T}}$ there is a corresponding polar decomposition $a=a_{S^1} \cdot a_{\mathbb{R}_{>0}}$.
\begin{definition}
    We say that $a$ is positive real if $a_{S^1} = 1$.
\end{definition}
\begin{lemma}\label{lemma: positive real implies Ta is a torus}
     If $a$ is positive real, then the diagonalizable subgroup scheme $\hat{T}_a \subset \hat{T}$ is a torus.
 \end{lemma}
 \begin{proof}
     The evaluation at $a$ defines a group homomorphism $X^*(\hat{\mathcal{T}}) \rightarrow \mathbb{R}_{>0}$ with kernel $X^*_{a}$. Hence, $X^*(\hat{\mathcal{T}})/X^*_{a} $ embeds into $\mathbb{R}_{>0}$ which is torsion-free. Thus, $\hat{T}_{a} = Spec(k[X^*(\hat{\mathcal{T}})/X^*_a])$ is a torus.
\end{proof}
\begin{corollary}
    If $a$ is positive real, then the group $G^{\hat{T}_a}$ is connected.
\end{corollary}
\begin{proof}
    Centralizers of tori in a reductive group are always connected.
\end{proof}
\begin{corollary}\label{cor: a positive real implies orbit finiteness}
     If $a$ is positive real and (A1) holds for $(G,V)$, then there are only finitely many $G^{\hat{T}_a}$-orbits in $\mathfrak{N}(V)^{\hat{T}_a}$.
\end{corollary}
\begin{proof}
     This follows from \cref{lemma: positive real implies Ta is a torus} and \cref{lemma: orbit finiteness for fixed points}.
\end{proof}

\begin{lemma}\label{lemma: top Borel Moore is group algebra of W}
    There is an isomorphism of algebras $\mathbb{C}[W] \cong H_{2 \dim Z}^{\hat{T}_a}(Z)$.
\end{lemma}
\begin{proof}
    Let $\textbf{S} := \textbf{S}^{a_e} = \mu_* \textbf{1}_{\tilde{V}}$ where $\mu :\tilde{V} \rightarrow V$ is the canonical map. By \cref{lemma: basic properties of Z} we have
    \begin{equation*}
        \dim Z_w = \dim G - \dim B \cap {}^w B + \dim V^- \cap {}^w V^- = \dim G - \dim T
    \end{equation*}
    for all $w \in W$. Hence, $\dim Z = \dim G - \dim T = \dim \tilde{V}$ which implies that the space $H_{2 \dim Z}^{\hat{T}_a}(Z)$ is stable under convolution (see \eqref{eq: convol in BM homology}). Moreover, the isomorphism of algebras from \eqref{eq: identification of BM homology and Ext algebra} restricts to an isomorphism (c.f. \cite[Prop. 8.6.1]{chriss2009representation}).
    \begin{align*}
        H_{2 \dim Z}^{\hat{T}_a} (Z) \cong \Hom_{D^b_{c,\hat{T}_a}(\mathfrak{N}(V))}^0(\textbf{S}, \textbf{S}).
    \end{align*}
    Note that $\dim Z = \dim \tilde{V}$ also implies that $\tilde{V} \rightarrow V$ is semismall (c.f. \cite[3.8.1]{achar2021perverse}). Hence, $\textbf{S}$ is a semisimple perverse sheaves. Since the forgetful functor $\Perv_{\hat{T}_a}(V) \rightarrow \Perv(V)$ is fully faithful (c.f. \cite[Prop. 6.2.15]{achar2021perverse}), we get an isomorphism of algebras
    \begin{equation*}
        H_{2 \dim Z}(Z) \cong \Hom^0_{D^b_{c}(\mathfrak{N}(V))}(\textbf{S}, \textbf{S}) \cong  \Hom_{D^b_{c,\hat{T}_a}(\mathfrak{N}(V))}^0(\textbf{S}, \textbf{S}) \cong H_{2 \dim Z}^{\hat{T}_a} (Z).
    \end{equation*}
    Note that $\Hom^0_{D^b_c(\mathfrak{N}(V))}(\textbf{S}, \textbf{S})$ is the maximal semisimple quotient of the algebra $\Hom^*_{D^b_{c}(\mathfrak{N}(V))} (\textbf{S}, \textbf{S})$. Since $\mathbb{C}[W]\# H^*(\mathcal{B}) \cong \Hom^*_{D^b_c(\mathfrak{N}(V))}(\textbf{S}, \textbf{S})$ by \cref{thm: affine Hecke algebra at central character is Ext algebra} this maximal semisimple quotient is isomorphic to $\mathbb{C}[W]$. This shows that
    \begin{equation*}
        \mathbb{C}[W] \cong \Hom^0_{D^b_c(\mathfrak{N}(V))}(\textbf{S}, \textbf{S}) \cong H_{2 \dim Z}^{\hat{T}_a} (Z).
    \end{equation*}
\end{proof}
Let us now fix a positive real element $a \in \hat{\mathcal{T}}_{\mathbb{R}_{>0}}$ and set
\begin{equation*}
    \hat{\mathfrak{t}} := Lie(\hat{\mathcal{T}}) = \mathbb{C} \otimes X_*(\hat{T}).
\end{equation*}
Then the exponential map $exp: \hat{\mathfrak{t}}  \rightarrow \hat{\mathcal{T}}$ restricts to an isomorphism of abelian groups
\begin{equation*}
    \hat{\mathfrak{t}}_{\mathbb{R}} = \mathbb{R} \otimes_{\mathbb{Z}} X_*(\hat{T}) \rightarrow \mathbb{R}_{>0} \otimes X_*(\hat{T}) = \hat{\mathcal{T}}_{\mathbb{R}_{>0}}. 
\end{equation*}
We denote the inverse of this map by $log: \hat{\mathcal{T}}_{\mathbb{R}_{>0}} \rightarrow \hat{\mathfrak{t}}_{\mathbb{R}}$. Note that for any $\lambda \in X^*(\hat{T})$, we have $\lambda(log(a)) = log(\lambda(a))$. In particular, we have
\begin{equation*}
    X^*_a = \{\lambda\in X^*(\hat{T}) \mid \lambda(log(a)) = 0 \}
\end{equation*}
and evaluation at $log(a)$ defines a group homomorphism $X^*(\hat{T}_a) = X^*(\hat{T})/X^*_a \rightarrow \mathbb{C}$. This induces an algebra homormophism
\begin{equation*}
    ev_{log(a)} : H^*_{\hat{T}_a}(pt) \cong Sym^*(\mathbb{C} \otimes X^*(\hat{T}_a )) \rightarrow \mathbb{C}
\end{equation*}
and thus a 1-dimensional $H^*_{\hat{T}_a}(pt)$-module $\mathbb{C}_{log(a)}$. By the localization theorem (\cref{thm: localization for BM homology}), we have an isomorphism of convolution algebras
\begin{equation*}
    H_*(Z^{\hat{T}_a}) \cong H_*^{\hat{T}_a}(Z) \otimes_{H^*_{\hat{T}_a} (pt)} \mathbb{C}_{log(a)}.
\end{equation*}
Thus, we obtain an algebra homomorphism 
\begin{equation*}
    \mathbb{C}[W] \overset{\cref{lemma: top Borel Moore is group algebra of W}}{\hookrightarrow} H^{\hat{T}_a}_*(Z) \rightarrow H_*(Z^{\hat{T}_a}).
\end{equation*}
In particular, $H_*^{\hat{T}_a}( \mathcal{B}_x)$ and $H_*(\mathcal{B}_x^{\hat{T}_a})$ carry a natural $\mathbb{C}[W\times A(a,x)]$-module structure where the $W$-action comes from the convolution actions of $H_*^{\hat{T}_a}(Z)$ and $H_*(Z^{\hat{T}_a})$. Note that for $a = a_e$, this yields a $\mathbb{C}[W\times A(a,x)]$-module structure on $H_*(\mathcal{B}_x)$.
\begin{lemma}\label{lemma: deformation result for H(Bx)}
    Assume that (A3) hold for $(G,V)$. Then for any $x \in \mathfrak{N}(V)^{\hat{T}_a}$ we have $H_*(\mathcal{B}_x) \cong H_*(\mathcal{B}_x^{\hat{T}_a})$ as $\mathbb{C}[W \times A(a,x)]$-modules.
\end{lemma}
\begin{proof}
    By the localization theorem (\cref{thm: localization for BM homology} applied with $T= \hat{T}_a$) there is an isomorphism
    \begin{equation*}
        H_*(\mathcal{B}_x^{\hat{T}_a}) \cong H_*^{\hat{T}_a}(\mathcal{B}_x^{\hat{T}_a} ) \otimes_{H^*_{\hat{T}_a}(pt)} \mathbb{C}_{log(a)} \cong H_*^{\hat{T}_a}(\mathcal{B}_x) \otimes_{H^*_{\hat{T}_a}(pt)} \mathbb{C}_{log(a)}
    \end{equation*}
    as modules over $H_*(Z^{\hat{T}_a}) \cong H_*^{\hat{T}_a}(Z) \otimes_{H^*_{\hat{T}_a} (pt)} \mathbb{C}_{log(a)}$. Moreover, it follows from (A3) and \cref{lemma: odd homology vanishing implies equivariantly formal} that $H_*^{\hat{T}_a}(Z)$ is a free $H^*_{\hat{T}_a}(pt)$-module of finite rank and the canonical forgetful map
    \begin{equation*}
        H_*^{\hat{T}_a}(\mathcal{B}_x) \otimes_{H^*_{\hat{T}_a}(pt)} \mathbb{C}_0 \rightarrow H_*(\mathcal{B}_x)
    \end{equation*}
    is an isomorphism. This is in fact an isomorphism of $\mathbb{C}[W \times A(a,x)]$-modules since the $W$-action on $H_*^{\hat{T}_a}(\mathcal{B}_x)$ comes from the forgetful isomorphism $H_{2 \dim Z}^{\hat{T}_a}(Z) \cong H_{2 \dim Z} (Z) \cong \mathbb{C}[W] $. Thus, $H_*^{\hat{T}_a}(\mathcal{B}_x)$ is a formal deformation of $\mathbb{C}[W \times A(a,x)]$-modules over the polynomial ring $H^*_{\hat{T}_a}(pt)$ which specializes to $H_*(\mathcal{B}_x)$ at $0$ and to $H_*(\mathcal{B}_x^{\hat{T}_a})$ at $log(a)$. It is well-known that representations of finite groups cannot be deformed in a non-trivial way. Hence, $H_*(\mathcal{B}_x) \cong H_*(\mathcal{B}_x^{\hat{T}_a})$ as $\mathbb{C}[W \times A(a,x)]$-modules.
\end{proof} 
\begin{lemma}\label{lemma: orbits in fixed point orbit are disjoint}
    Let $\mathcal{O} \subset \mathfrak{N}(V)$ be a $G$-orbit. Then for any two distinct $G^{\hat{T}_a}$-orbits $\mathcal{O}_1, \mathcal{O}_2 \subset \mathcal{O}^{\hat{T}_a}$, we have $\mathcal{O}_2 \not \subset  \overline{\mathcal{O}}_1$ where $\overline{\mathcal{O}}_1$ is the closure of $\mathcal{O}_1$ in $\mathfrak{N}(V)$.
\end{lemma}
\begin{proof}
    By \cref{lemma: orbit finiteness for fixed points} the orbits $\mathcal{O}_1$ and $ \mathcal{O}_2$ are both closed in $\mathcal{O}^{\hat{T}_a}$. Thus, $\overline{\mathcal{O}}_1 \cap \mathcal{O}^{\hat{T}_a} = \mathcal{O}_1$ which implies $\mathcal{O}_2 \cap \overline{\mathcal{O}}_1 = \emptyset$.
\end{proof}
\begin{theorem}\label{thm: DL correspondence for simply connected and positive real central character}
    Let $V$ be an $I$-rooted $G$-representation and let $a \in \hat{\mathcal{T}}$. Assume that
    \begin{itemize}
        \item $G$ has simply connected derived subgroup,
        \item $a$ is positive real,
        \item (A1), (A2) and (A3) hold for $(G,V)$.
    \end{itemize}
    Then there is a bijection
    \begin{equation*}
        \Irr(\mathcal{H}^{\aff}_{\chi_a}) \overset{1:1}{\leftrightarrow}  \mathcal{P}(G,V,a).
    \end{equation*}
    If in addition (A2') holds for $(G,V)$, then there is a bijection
    \begin{equation*}
        \Irr(\mathcal{H}^{\aff}_{\chi_a})  \overset{1:1}{\leftrightarrow}  \mathfrak{N}(V)^{\hat{T}_a} / G^{\hat{T}_a}.
    \end{equation*}
\end{theorem}
\begin{proof}
    Recall from \cref{thm: affine Hecke algebra at central character is Ext algebra} that we have an isomorphism $\mathcal{H}^{\aff}_{\chi_a} \cong H_*(Z^{\hat{T}_a})$. The number of $G^{\hat{T}_a}$-orbits on $\mathfrak{N}(V)^{\hat{T}_a}$ is finite by \cref{cor: a positive real implies orbit finiteness}. Thus, it suffices to show that the injective map $\Irr(\mathcal{H}^{\aff}_{\chi_a}) \hookrightarrow \mathcal{P}(G,V,a)$ from \eqref{eq: injective DL parametrization} is surjective. For $a = a_e$, this follows from (A2) and the isomorphism $\mathbb{C}[W] \# H^*(\mathcal{B}) \cong \mathcal{H}^{\aff}_{\chi_{a_e}}$. Thus, it follows from \cref{lemma: description of DL corr with standard modules} that for any $x \in \mathfrak{N}(V)$ and $\rho \in A(x)$, there is a unique irreducible representation $L_{x, \rho}$ of $W$ such that $L_{x, \rho}$ appears in $H_*(\mathcal{B}_x)_{\rho}$ but not in $H_*(\mathcal{B}_{x'})$ unless $x' \in \overline{G \cdot x}$. 
    
    Now let $\rho \in \Irr( A(a,x))$ such that $H_*(\mathcal{B}^{\hat{T}_a}_x)_{\rho} \neq 0$. By \cref{lemma: deformation result for H(Bx)} we have an isomorphism of $W \times A(a,x)$-representations $H_*(\mathcal{B}_x) \cong H_*(\mathcal{B}_x^{\hat{T}_a})$. Since the $A(a,x)$-action on $H_*(\mathcal{B}_x)$ factors through $A(x)$, we can find $\tilde{\rho} \in \Irr( A(x))$ such that $\rho$ appears in the restriction of $\tilde{\rho}$ to $A(a,x)$ and $H_*(\mathcal{B}_x)_{\tilde{\rho}} \neq 0$. Note that any non-zero morphism of $A(a,x)$-representations $\tilde{\rho} \rightarrow H_*({\mathcal{B}_x})$ is injective since $\tilde{\rho}$ is irreducible. In particular, the restriction of such a map along $\rho \hookrightarrow \tilde{\rho}$ is still non-zero. Thus, we get an injection
    \begin{equation*}
        H_*({\mathcal{B}_x})_{\tilde{\rho}} \hookrightarrow H_*(\mathcal{B}_x)_{\rho} \cong H_*(\mathcal{B}_x^{\hat{T}_a})_{\rho}.
    \end{equation*}
    Hence, the simple $W$-module $L_{x, \tilde{\rho}}$ appears in $H_*(\mathcal{B}_x^{\hat{T}_a})_{\rho}$. Now assume that $L_{x, \tilde{\rho}}$ appears in $H_*(\mathcal{B}_{x'}^{\hat{T}_a})$ with $x' \in \mathfrak{N}(V)^{\hat{T}_a}$ and $x \in \overline{G^{\hat{T}_a} \cdot x'}$. Since $H_*(\mathcal{B}_{x'}) \cong H_*(\mathcal{B}_{x'}^{\hat{T}_a})$ as $W$-modules, the defining property of $L_{x, \tilde{\rho}}$ implies that $x' \in \overline{G \cdot x}$ and thus $G\cdot x = G \cdot x'$. By \cref{lemma: orbits in fixed point orbit are disjoint} this implies that $x$ and $x'$ lie in the same $G^{\hat{T}_a}$-orbit. Let $M$ be any of the simple $H_*(Z^{\hat{T}_a})$-modules in a composition series of $ H_*(\mathcal{B}_x^{\hat{T}_a})_{\rho}$ containing the $W$-representation $L_{x, \tilde{\rho}}$. It then follows that $M$ does not appear in $H_*(\mathcal{B}_{x'}^{\hat{T}_a})$ for any $x' \in \mathfrak{N}(V)^{\hat{T}_a}$ with $x \in \overline{G^{\hat{T}_a} \cdot x'}$ unless $x$ and $x'$ lie in the same $G^{\hat{T}_a}$-orbit. By \cref{lemma: description of DL corr with standard modules} this implies that $M$ maps to $ (x, \rho)$ under the map $\Irr(\mathcal{H}^{\aff}_{\chi_a}) \hookrightarrow \mathcal{P}(G,V,a)$. Since $M$ was arbitrary, this map surjective.

    If in addition (A2') holds, then $A(x)$ and thus also $A(a,x)$ act trivially on $H_*(\mathcal{B}_x) \cong H_*(\mathcal{B}_x^{\hat{T}_a})$. Hence, the projection onto the first component $\mathcal{P}(G,V,a) \rightarrow \mathfrak{N}(V)^{\hat{T}_a} / G^{\hat{T}_a}$ is injective. It is also surjective since $H_*(\mathcal{B}_x) \cong H_*(\mathcal{B}_x^{\hat{T}_a})$ is always non-zero. Hence, we get bijections $\Irr(\mathcal{H}^{\aff}_{\chi_a}) \overset{1:1}{\leftrightarrow}  \mathcal{P}(G,V,a) \overset{1:1}{\leftrightarrow}  \mathfrak{N}(V)^{\hat{T}_a} / G^{\hat{T}_a}$.
\end{proof}
Our next goal is to show that the simply connectedness assumption in the theorem above can be removed. Given a central isogeny $G \rightarrow G'$ we get corresponding isogenies $ \hat{T} \rightarrow \hat{T}'$ and $\hat{\mathcal{T}} \rightarrow \hat{\mathcal{T}}'$. For any $a \in \hat{\mathcal{T}}$, we write $a' \in \hat{\mathcal{T}}'$ for the image of $a$ under $\hat{\mathcal{T}} \rightarrow \hat{\mathcal{T}}'$. Note that if $a \in \hat{\mathcal{T}}_{\mathbb{R}_{>0}}$ then $a' \in \hat{\mathcal{T}}'_{\mathbb{R}_{>0}}$.
\begin{lemma}\label{lemma: central isogeny surjective on fixed points}
    Let $G \rightarrow G'$ be a central isogeny of connected reductive groups and $V$ a rooted $G'$-representation. Then for any $a \in \hat{\mathcal{T}}_{\mathbb{R}_{>0}}$ the canonical map $G^{\hat{T}_a} \rightarrow (G')^{\hat{T}'_{a'}}$ is surjective and $\mathfrak{N}(V)^{\hat{T}_a} = \mathfrak{N}(V)^{\hat{T}'_{a'}}$.
\end{lemma}
\begin{proof}
    Since $\hat{T}_a$ is a subtorus of $T$ (\cref{lemma: positive real implies Ta is a torus}), the group $G^{\hat{T}_a}$ is connected reductive and generated by $T$ together with the $\{ U_{\alpha}  \mid \alpha(a) = 1\}$. Similarly, since $a' \in \hat{\mathcal{T}}_{\mathbb{R}_{>0}}'$ the group $\hat{T}'_{a'}$ is a subtorus of $T'$ and $(G')^{\hat{T}'_{a'}}$ is connected reductive and generated by $T'$ together with the $\{ U'_{\alpha} \mid \alpha(a') = 1 \}$. Note that $\alpha(a') = \alpha(a)$ and the canonical maps $T \rightarrow T'$ and $U_{\alpha} \rightarrow U_{\alpha}'$ are surjective. Thus, the map $G^{\hat{T}_a} \rightarrow (G')^{\hat{T}'_{a'}}$ is also surjective. The equality $\alpha(a') = \alpha(a)$ also implies that $V^{\hat{T}_a} = V^{\hat{T}'_{a'}}$ and thus $\mathfrak{N}(V)^{\hat{T}_a} = \mathfrak{N}(V)^{\hat{T}'_{a'}}$.
\end{proof}
\begin{corollary}\label{cor: central isogeny preserves parameter sets}
    Let $G \rightarrow G'$ be a central isogeny of connected reductive groups and $V$ a rooted $G'$-representation. For any $x \in \mathfrak {N}(V)^{\hat{T}_a}$ the canonical map $A_G(a,x) \rightarrow A_{G'}(a',x)$ is surjective. In particular, there is a canonical bijection $\mathcal{P}(G,V,a) \overset{1:1}{\leftrightarrow} \mathcal{P}(G',V,a')$.
\end{corollary}
\begin{proof}
    The map $G^{\hat{T}_a} \rightarrow (G')^{\hat{T}'_{a'}}$ is surjective (\cref{lemma: central isogeny surjective on fixed points}) so $G^{\hat{T}_a}_x \rightarrow (G')^{\hat{T}'_{a'}}_x$ is also surjective. This implies that is also $A_G(a,x) \rightarrow A_{G'}(a',x)$ surjective. Note that the canonical map $\mathcal{B}(G) \rightarrow \mathcal{B}(G')$ and thus also the canonical map $\mathcal{B}(G)_x \rightarrow \mathcal{B}(G')_x$ is an isomorphism. It follows from this that there is a bijection $\mathcal{P}(G,V,a) \overset{1:1}{\leftrightarrow} \mathcal{P}(G',V,a')$ which inflates along $A_G(a,x) \rightarrow A_{G'}(a',x)$.
\end{proof}
The following lemma can be thought of as a special case of the results in \cite[§3]{reeder2002isogenies}.
\begin{lemma}\label{lemma: clifford theory at positive real central character}
    Let $G \rightarrow G'$ be a central isogeny and $a \in \hat{\mathcal{T}}_{\mathbb{R}_{>0}}$. Then restriction along the canonical map $\mathcal{H}^{' \aff}_{\chi_{a'}} \rightarrow \mathcal{H}^{\aff}_{\chi_a}$ induces a bijection $\Irr(\mathcal{H}^{\aff}_{\chi_a}) \overset{1:1}{\leftrightarrow} \Irr(\mathcal{H}^{' \aff}_{\chi_{a'}})$
\end{lemma}
\begin{proof}
    The group $C:= ker(\mathcal{T} \rightarrow \mathcal{T}')$ acts on $\mathcal{H}^{\aff}_{q_{\textbf{V}}} $ via $c \cdot T_{s_{\alpha}} = T_{s_{\alpha}}$ and $c \cdot e^{\lambda} = \lambda(c) e^{\lambda}$. Note that $e^{\lambda}$ is fixed under the action of $C$ if and only if $\lambda \in X^*(\mathcal{T}')$. Hence, we can identify $\mathcal{H}^{' \aff} = (\mathcal{H}^{\aff})^C$. If $M$ is an $\mathcal{H}^{\aff}$-module, we denote by ${}^cM$ the $\mathcal{H}^{\aff}$-module with the action twisted by $c$ (i.e. $h \in \mathcal{H}^{\aff}$ acts on ${}^cM$ by the action of $c(h)$ on $M$). If $M$ is a simple $\mathcal{H}^{\aff}$-module with central character $\chi_a$, then ${}^c M$ has central character $\chi_{ca}$. Since $a \in \hat{\mathcal{T}}_{\mathbb{R}_{>0}}$ and $C \backslash \{ e\} \subset \mathcal{T}_{S^1}$, it follows from the polar decomposition that $W\cdot a \neq W \cdot ca$ for all $c \in C \backslash \{ e \}$. Thus, $M \not\cong {}^cM$ for all $c \in C \backslash \{e\}$. It then follows by the Clifford theory of \cite[Theorem A.6, Lemma A.9]{ram2003affine} that $\Ind^{\mathcal{H}^{\aff} \rtimes C}_{\mathcal{H}^{\aff}} M $ is a simple $\mathcal{H}^{\aff}\rtimes C$-module and $e\Ind^{\mathcal{H}^{\aff} \rtimes C}_{\mathcal{H}^{\aff}} M$ is a simple $ e(\mathcal{H}^{\aff}\rtimes C)e$-module where $e \in \mathbb{C}[C]$ is the idempotent corresponding to the trivial representation. Moreover, by \cite[Lemma A.9]{ram2003affine} we have
    \begin{equation*}
        \mathcal{H}^{' \aff} = (\mathcal{H}^{\aff})^C \cong e(\mathcal{H}^{\aff} \rtimes C)e
    \end{equation*}
    and under this isomorphism the simple $e(\mathcal{H}^{\aff} \rtimes C)e$-module $e\Ind^{\mathcal{H}^{\aff} \rtimes C}_{\mathcal{H}^{\aff}} M$ corresponds to the $\mathcal{H}^{' \aff} $-module $M$ obtained via restriction along the inclusion $\mathcal{H}^{' \aff}  \hookrightarrow \mathcal{H}^{\aff}$. Thus, restriction induces a map
    \begin{equation}\label{eq: isogeny map on simples}
        \Irr(\mathcal{H}^{\aff}_{\chi_a}) \rightarrow  \Irr(\mathcal{H}^{' \aff}_{\chi_{a'}}).
    \end{equation}
    This map is surjective since any simple $\mathcal{H}^{' \aff}$-module $M$ is a submodule of $\mathcal{H}^{\aff} \otimes_{\mathcal{H}^{' \aff}} M$ and thus appears in some simple $\mathcal{H}^{\aff}$-module. If $M_1$ and $M_2$ are two simple $\mathcal{H}^{\aff}$-modules with central character $\chi_a$ and $e\Ind^{\mathcal{H}^{\aff} \rtimes C}_{\mathcal{H}^{\aff}} M_1  \cong e\Ind^{R \rtimes C}_R M_2$, then by standard results about idempotents (c.f. \cite[6.2f]{green2007polynomial}) we have $\Ind^{\mathcal{H}^{\aff} \rtimes C}_{\mathcal{H}^{\aff}} M_1  \cong \Ind^{\mathcal{H}^{\aff} \rtimes C}_{\mathcal{H}^{\aff}} M_2$ which implies ${}^cM_1 \cong M_2$ for some $c \in C$ by \cite[Theorem A.6]{ram2003affine}. Since $M_1$ and $M_2$ were both assumed to have central character $\chi_a$, this is only possible if $c = e$ and $M_1 \cong M_2$. This shows that the map from \eqref{eq: isogeny map on simples} is also injective and thus a bijection.
\end{proof}
\begin{remark}
    The lemma above can also be deduced from the reduction theorems in \cite{lusztig1989affine,barbasch1993reduction} which match up the irreducible representations of $\mathcal{H}^{\aff}$ with central character ${\chi_a}$ (resp. of $\mathcal{H}^{' \aff}$ with central character at ${\chi_{a'}}$) with those of a corresponding graded Hecke algebra at a certain central character. For positive real central characters the resulting graded Hecke algebras for $\mathcal{H}^{\aff}_{\chi_a} $ and $\mathcal{H}^{' \aff}_{\chi_{a'}}$ and the respective central characters are the same which recovers the bijection $\Irr(\mathcal{H}^{\aff}_{\chi_a}) \overset{1:1}{\leftrightarrow}  \Irr(\mathcal{H}^{' \aff}_{\chi_{a'}})$.
\end{remark}
\begin{corollary}\label{cor: DL correspondonce without simply connected for positive real central character}
    The result from \cref{thm: DL correspondence for simply connected and positive real central character} holds without the simply connectedness assumption, i.e. if $V$ is an $I$-rooted $G$-representation such that (A1), (A2) and (A3) hold for $(G,V)$ and $a$ is positive real, there is a bijection
    \begin{equation*}
        \Irr(\mathcal{H}^{\aff}_{\chi_a}) \overset{1:1}{\leftrightarrow}  \mathcal{P}(G,V,a).
    \end{equation*}
    If in addition (A2') holds, then there is a bijection
    \begin{equation*}
        \Irr(\mathcal{H}^{\aff}_{\chi_a}) \overset{1:1}{\leftrightarrow}   \mathfrak{N}(V)^{\hat{T}_a} / G^{\hat{T}_a}.
    \end{equation*}
\end{corollary}
\begin{proof}
    Let $\tilde{G} \rightarrow G$ be a central isogeny such that $\tilde{G}$ has simply connected derived subgroup (e.g. take $\tilde{G} = \mathcal{D}(G)^{sc} \times Z(G)^{\circ}$ where $\mathcal{D}(G)^{sc}$ is the simply connected cover of the semisimple group $\mathcal{D}(G)$). Let $\tilde{a} \in \hat{\tilde{\mathcal{T}}}_{\mathbb{R}_{>0}} \subset \tilde{G}$ be the unique real element lifting $a \in \hat{\mathcal{T}}_{\mathbb{R}_{>0}}$. Then we have bijections
    \begin{equation*}
        \Irr(\mathcal{H}^{\aff}_{\chi_a}) \overset{\cref{lemma: clifford theory at positive real central character}}{\leftrightarrow} \Irr(\tilde{\mathcal{H}}^{\aff}_{\chi_{\tilde{a}}})\overset{\cref{thm: DL correspondence for simply connected and positive real central character}}{\leftrightarrow} \mathcal{P}(\tilde{G}, V , \tilde{a}) \overset{\cref{cor: central isogeny preserves parameter sets}}{\leftrightarrow} \mathcal{P}(G,V, a).
    \end{equation*}
    If in addition (A2') holds, we get bijections
    \begin{equation*}
        \mathcal{P}(\tilde{G}, V , \tilde{a}) \overset{\cref{thm: DL correspondence for simply connected and positive real central character}}{\leftrightarrow} \mathfrak{N}(V)^{\hat{\tilde{T}}_{\tilde{a}}} / (\tilde{G})^{\hat{\tilde{T}}_{\tilde{a}}} \overset{\cref{lemma: central isogeny surjective on fixed points}}{\leftrightarrow} \mathfrak{N}(V)^{\hat{T}_a} / G^{\hat{T}_a}.
    \end{equation*}
    Combining these bijections yields the bijection $\Irr(\mathcal{H}^{\aff}_{\chi_a}) \overset{1:1}{\leftrightarrow }\mathfrak{N}(V)^{\hat{T}_a} / G^{\hat{T}_a}$.
\end{proof}

\subsection{Reduction to positive real central character}\label{section: reduction to pos real central char}
In this section, we show how the study of a general central character $\chi_a$ can be reduced to a positive real central character. This can be thought of as a geometric version of the reduction theorems in \cite{lusztig1989affine,barbasch1993reduction}. Let $\mathcal{D} \subset T$ be a diagonalizable subgroup scheme and $K_{\mathcal{D}} = \ker(X^*(T) \rightarrow X^*(\mathcal{D}))$.
\begin{lemma}
    $V^{\mathcal{D}}$ is a rooted $(G^{\mathcal{D}})^{\circ}$-representation.
\end{lemma}
\begin{proof}
    We have $V^{\mathcal{D}} = V_0 \oplus \bigoplus_{\alpha \in \Phi \cap K_{\mathcal{D}}} V_{\alpha}$. By \cref{lemma: explicit descrpition of centralizer in reductive group} the set of roots in $(G^{\mathcal{D}})^{\circ}$ is $\Phi \cap K_{\mathcal{D}}$. Hence, $V^{\mathcal{D}}$ is a rooted $(G^{\mathcal{D}})^{\circ}$-representation.
\end{proof}
Recall that we have morphisms
\begin{align*}
    \tilde{V} & \cong G \times^B V^- \overset{\mu}{\rightarrow} V\\
    \widetilde{V^{\mathcal{D}}} &\cong (G^{\mathcal{D}})^{\circ} \times^{B^{\mathcal{D}}} (V^{\mathcal{D}})^- \overset{\mu_{\mathcal{D}}}{\rightarrow} V^{\mathcal{D}}
\end{align*}
with nilcones $\mathfrak{N}(V) = im(\mu)$ and $\mathfrak{N}(V^{\mathcal{D}})=im(\mu^{\mathcal{D}})$.
\begin{lemma}\label{lemma: taking fixed points is compatible with nilcones}
    We have $\mathfrak{N}(V)^{\mathcal{D}} = \mathfrak{N}(V^{\mathcal{D}})$.
\end{lemma}
\begin{proof}
    Clearly $ \mathfrak{N}(V^{\mathcal{D}}) \subset \mathfrak{N}(V)^{\mathcal{D}}$. If $x \in \mathfrak{N}(V)^{\mathcal{D}} \subset V^{\mathcal{D}}$ we can find by \cref{lemma: basic results about V le 0} a $g \in (G^{\mathcal{D}})^{\circ}$ such that $gx \in (V^{\mathcal{D}})^{\le 0}$. Hence,
    \begin{equation*}
        gx \in V^{\mathcal{D}} \cap V^{\le 0} \cap \mathfrak{N}(V) \overset{\cref{lemma: basic results about V le 0}}{=} V^{\mathcal{D}} \cap V^- = (V^{\mathcal{D}})^-.
    \end{equation*}
    This shows that $x \in \mathfrak{N}(V^{\mathcal{D}})$.
\end{proof}
Restricting $\mu$ to $\mathcal{D}$-fixed points also yields a morphism
\begin{equation*}
    \tilde{V}^{\mathcal{D}} \overset{\mu^{\mathcal{D}}}{\rightarrow} V^{\mathcal{D}}.
\end{equation*}
\begin{lemma}\label{lemma: connnected components of fixed point resolution}
    Each connected component of $\tilde{V}^{\mathcal{D}} $ is $(G^{\mathcal{D}})^{\circ}$-equivariantly isomorphic to $\widetilde{V^{\mathcal{D}}}$ over $V^{\mathcal{D}}$.
\end{lemma}
\begin{proof}
    We identify $\mathcal{B} = G/B$ with the variety of all Borel subgroups of $G$. By \cref{lemma: fixed points of flag variety} each connected component $C \subset \mathcal{B}^{\mathcal{D}}$ is $(G^{\mathcal{D}})^{\circ}$-equivariantly isomorphic to $(G^{\mathcal{D}})^{\circ}/B^{\mathcal{D}}$. Hence, there is a unique element $B_C \in C$ with $(B_C)^{\mathcal{D}} = B^{\mathcal{D}}$. Note that $T \subset B^{\mathcal{D}} = (B_C)^{\mathcal{D}} \subset B_C$ and thus $B_C = \dot{w} B\dot{w} ^{-1}$ for some $w \in W$. Comparing the weights of $Lie(B^{\mathcal{D}}) = Lie(B)^{\mathcal{D}}$ and $Lie(B_C^{\mathcal{D}}) = Lie(\dot{w} B \dot{w}^{-1})^{\mathcal{D}}$ yields
    \begin{equation*}
        \Phi^- \cap K_{\mathcal{D}} = w(\Phi^-)\cap K_{\mathcal{D}}.
    \end{equation*}
    Using the Bruhat decomposition, one can easily check that $\pi^{\mathcal{D}}: \tilde{V}^{\mathcal{D}} \rightarrow \mathcal{B}^{\mathcal{D}}$ is a $(G^{\mathcal{D}})^{\circ}$-equivariant vector bundle. Note that the fiber of $B_C = \dot{w} B\dot{w} ^{-1}$ under $\pi^{\mathcal{D}}$ is
    \begin{equation*}
        (\dot{w} V^-)^{\mathcal{D}} = \bigoplus_{w(\Phi^-)\cap K_{\mathcal{D}}}V_{\alpha} = \bigoplus_{\Phi^-\cap K_{\mathcal{D}}}V_{\alpha} = (V^-)^{\mathcal{D}} = (V^{\mathcal{D}})^-.
    \end{equation*} Hence, we get an isomorphism
    \begin{equation*}
        (\pi^{\mathcal{D}})^{-1}(C) \cong(G^{\mathcal{D}})^{\circ} \times^{B^{\mathcal{D}}} (V^{\mathcal{D}})^-  =  \widetilde{V^{\mathcal{D}}}  .
    \end{equation*}
    Since $C$ is a connected component of $\mathcal{B}^{\mathcal{D}}$ and $(\pi^{\mathcal{D}})^{-1}(C)$ is connected, we see that $(\pi^{\mathcal{D}})^{-1}(C)$ is a connected component of $\tilde{V}^{\mathcal{D}}$. Thus, all the connected components of $ \tilde{V}^{\mathcal{D}}$ are isomorphic to $\widetilde{V^{\mathcal{D}}} $.
\end{proof}

\begin{lemma}\label{lemma: simply connected implies connected centralizer}
    Assume that $G$ has simply connected derived subgroup. Then for any $a \in \hat{\mathcal{T}}$, the reductive group $G^{\hat{T}_a}$ is connected.
\end{lemma}
\begin{proof}
    Let $\mathcal{G}$ be the complex reductive group with the same root datum as $G$. By \cref{lemma: explicit descrpition of centralizer in reductive group} the group $G^{\hat{T}_a}$ is connected if and only if for any $w \in W$ such that
    \begin{align*}
        w (X^*_a) &= X^*_a  \quad \text{and} \quad w|_{X^*(\hat{T})/X^*_a} = \id_{X^*(\hat{T})/X^*_a}
    \end{align*}
    we have $w \in \langle s_{\alpha } \mid \alpha \in \Phi \cap X^*_a \rangle$. Applying \cref{lemma: explicit descrpition of centralizer in reductive group} again, this is also equivalent to $\mathcal{G}^{\hat{\mathcal{T}}_a}$ being connected where $\hat{\mathcal{T}}_a = Spec(\mathbb{C}[X^*(\hat{\mathcal{T}})/X^*_a ])$. However, $\mathcal{G}^{\hat{\mathcal{T}}_a}= \mathcal{G}^a$ since $\hat{\mathcal{T}}_a$ is the closed subgroup of $\hat{\mathcal{T}}$ generated by $a$. A well-known result of Steinberg (c.f. \cite[3.9]{springer1970conjugacy}) states that $\mathcal{G}^a$ is connected if $\mathcal{G}$ has simply-connected derived subgroup. Hence, $G^{\hat{T}_a}$ is connected.
\end{proof}
For any $a = (s,(t_i)_{i \in I}) \in \hat{\mathcal{T}}$ we define
\begin{equation}\label{eq: def of Sa}
    \begin{aligned}
    Q_a &:= \{ \lambda \in X^*(T) \mid \lambda(s) \in \langle t_i \mid i \in I \rangle \} \\
        S_a &:= \Spec(k[X^*(T) /Q_a] ) \subset T.
    \end{aligned}
\end{equation}
Note that $X^*_a \subset Q_a \oplus X^*((\mathbb{G}_m)^I)$ and thus $S_a \subset \hat{T}_a$. In particular, for any $\hat{T}$-variety $Y$ we have
\begin{equation*}
    (Y^{S_a})^{\hat{T}_a} = Y^{\hat{T}_a}.
\end{equation*}
\begin{lemma}\label{lemma: simply connected implies double fixed points well behaved}
    If $G$ has simply connected derived subgroup, then $((G^{S_a})^{\circ})^{\hat{T}_a} = G^{\hat{T}_a}$.
\end{lemma}
\begin{proof}
    There is an injection $ (G^{S_a})^{\hat{T}_a} /  ((G^{S_a})^{\circ})^{\hat{T}_a} \hookrightarrow G^{S_a} / (G^{S_a})^{\circ}$ so the quotient $ (G^{S_a})^{\hat{T}_a} /  ((G^{S_a})^{\circ})^{\hat{T}_a} $ is discrete. Since $(G^{S_a})^{\hat{T}_a} = G^{\hat{T}_a}$ is connected by \cref{lemma: simply connected implies connected centralizer}, this implies $G^{\hat{T}_a}= (G^{S_a})^{\hat{T}_a} =  ((G^{S_a})^{\circ})^{\hat{T}_a}$.
\end{proof}
For any $a \in \hat{T}$, we can take $\hat{T}_a$-fixed points to get morphisms
\begin{align*}
        \tilde{V}^{\hat{T}_a} \overset{\mu^a}&{\rightarrow} V^{\hat{T}_a} \\
        (\widetilde{V^{S_a}})^{\hat{T}_a} \overset{(\mu_{S_a})^a}&{\rightarrow} (V^{S_a})^{\hat{T}_a} = V^{\hat{T}_a}
\end{align*}
and associated sheaves as in \eqref{eq: def of fixed point Springer sheaf}:
\begin{align*}
    \textbf{S}^a &= \mu^a_* \textbf{1}_{\tilde{V}^{\hat{T}_a}}  \in D^b_c(\mathfrak{N}(V)^{\hat{T}_a})  \\
    \bar{\textbf{S}}^a &= (\mu_{S_a})^a_* \textbf{1}_{(\widetilde{V^{S_a}})^{\hat{T}_a} } \in D^b_c(\mathfrak{N}(V^{S_a})^{\hat{T}_a}).
\end{align*}
Note that $\textbf{S}^a$ and $\bar{\textbf{S}}^a$ live on the same space:
\begin{equation}\label{eq: Ta fixed points on nilcone for Sa and without}
    \mathfrak{N}(V)^{\hat{T}_a} = (\mathfrak{N}(V)^{S_a})^{\hat{T}_a}\overset{\cref{lemma: taking fixed points is compatible with nilcones}}{=}\mathfrak{N}(V^{S_a})^{\hat{T}_a}.
\end{equation}
\begin{lemma}\label{lemma: Sa and bar Sa are the same up to dir sum}
    Assume that $G$ has simply-connected derived subgroup. Then
    \begin{equation}\label{eq: iso of springer sheaves under taking fixed points}
        \textbf{S}^a \cong ( \bar{\textbf{S}}^a)^{\oplus n}
    \end{equation}
    for some $n \ge 1$ and
    \begin{equation}\label{eq: quality of parameter space under taking fixed points}
        \mathcal{P}(G,V,a) = \mathcal{P}((G^{S_a})^{\circ}, V^{S_a}, a).
    \end{equation}
\end{lemma}
\begin{proof}
    By \cref{lemma: connnected components of fixed point resolution} $\tilde{V}^{S_a}$ is a disjoint union of several copies of $\widetilde{V^{S_a}}$ over $V^{S_a}$. Taking $\hat{T}_a$-fixed points, we see that $\tilde{V}^{\hat{T}_a} = (\tilde{V}^{S_a})^{\hat{T}_a}$ is a disjoint union of several copies of $(\widetilde{V^{S_a}})^{\hat{T}_a}$ over $ (V^{S_a})^{\hat{T}_a} = V^{\hat{T}_a}$. This implies \eqref{eq: iso of springer sheaves under taking fixed points}.
    
    Since $G^{\hat{T}_a} = ((G^{S_a})^{\circ})^{\hat{T}_a}$ by \cref{lemma: simply connected implies double fixed points well behaved} we have $A_G(a,x) = A_{(G^{S_a})^{\circ}}(a,x)$ for any $x \in \mathfrak{N}(V)^{\hat{T}_a}$. Using \eqref{eq: iso of springer sheaves under taking fixed points} and the isomorphisms of $A(a,x)$-modules
    \begin{align*}
        H_*(\mathcal{B}(G)^{\hat{T}_a}_x) &\cong H^*(\iota_x^!\textbf{S}^a) \\
        H_*(\mathcal{B}((G^{S_a})^{\circ})^{\hat{T}_a}_x) &\cong H^*(\iota_x^!\bar{\textbf{S}}^a)
    \end{align*}
    we see that $\rho \in \Irr(A(a,x))$ appears in $H_*(\mathcal{B}(G)^{\hat{T}_a}_x)$ if and only if it appears in $H_*(\mathcal{B}(G^{S_a})^{\hat{T}_a}_x)$. This proves \eqref{eq: quality of parameter space under taking fixed points}.
\end{proof}
\begin{lemma}\label{lemma: condition when a can be replaced with real element}
    Let $a = (s, (t_i)_{i \in I}) \in \hat{\mathcal{T}}$ such that $\langle t_i \mid i \in I \rangle$ is torsion-free and $\alpha(s) \in \langle t_i \mid i \in I \rangle$ for all $\alpha \in \Phi$. Then there is an element $a' \in \hat{\mathcal{T}}_{\mathbb{R}_{>0}}$ with $V^{\hat{T}_a} = V^{\hat{T}_{a'}}$, $\tilde{V}^{\hat{T}_a} = \tilde{V}^{\hat{T}_{a'}}$ and $(G^{\hat{T}_a})^{\circ} = (G^{\hat{T}_{a'}})^{\circ} =  G^{\hat{T}_{a'}}$. In particular, we have $\textbf{S}^a = \textbf{S}^{a'}$. If $G^{\hat{T}_a}$ is connected, we also have $\mathcal{P}(G,V,a) = \mathcal{P}(G,V,a')$.
\end{lemma}
\begin{proof}
    Pick free generators $c_1,...,c_k$ of the torsion-free abelian group $\langle t_i \mid i \in I \rangle \cong \mathbb{Z}^k$. We can then pick $d_1 ,...,d_k \in \mathbb{C}$ such that $exp(d_i) = c_i$. Then the exponential map restricts to an isomorphism of abelian groups $\langle d_1,..,d_k \rangle \overset{\sim}{\rightarrow} \langle c_1,...,c_k \rangle = \langle t_i \mid i \in I \rangle$. Pick elements $r_1, ...,r_k \in \mathbb{R}$ which are $\mathbb{Z}$-linearly independent. We then get a commutative diagram of isomorphisms of abelian groups
    \begin{equation*}
        \begin{tikzcd}
            {\langle exp(r_1), ... ,  exp(r_k) \rangle} \arrow[r, "f"]        & {\langle t_i \mid i \in I \rangle}                  \\
            {\langle r_1,...,r_k \rangle} \arrow[r, "g"] \arrow[u, "exp"] & {\langle d_1,...,d_k\rangle} \arrow[u, "exp"]
        \end{tikzcd}
    \end{equation*}
    where $g(r_i) = d_i$ and $f(exp(r_i)) = c_i$. For each $\alpha \in \Phi$ let $r_{\alpha} \in \langle r_1, ..., r_k \rangle \subset \mathbb{R}$ be the unique element with $f(exp(r_{\alpha})) = \alpha(s)$. Since $ev_s : \mathbb{Z}\Phi \rightarrow \langle t_i \mid i \in I \rangle$ is a group homomorphism, the map $\alpha \mapsto r_{\alpha}$ induces a group homomorphism $\mathbb{Z}\Phi \rightarrow \mathbb{R}$. Hence, the map
    \begin{equation*}
        \mathbb{R} \otimes_{\mathbb{Z}} \mathbb{Z}\Phi \rightarrow \mathbb{R}, \quad \alpha \mapsto r_{\alpha}
    \end{equation*}
    is $\mathbb{R}$-linear. In particular, this map is given by the evaluation at some element $x_s \in \mathbb{R} \otimes_{\mathbb{Z}} \mathbb{Z}\Phi^{\vee} \subset \mathbb{R} \otimes_{\mathbb{Z}} X_*(T)$. Applying the exponential map to $x_s$ yields an element $s' \in \mathbb{R}_{>0} \otimes_{\mathbb{Z}} X_*(T) = \mathcal{T}_{\mathbb{R}_{>0}}$. Note that by construction, we have
    \begin{equation*}
        f(\alpha(s')) = f(\alpha(exp(x_s))) = f(exp(\alpha(x_s))) = f(exp(r_{\alpha})) = \alpha(s).
    \end{equation*}
    For each $i \in I$ let $t_i' \in \langle exp(r_1), ..., exp(r_k) \rangle \subset \mathbb{R}_{>0}$ be the unique element with $f(t_i') = t_i$. Then the bijection $f$ makes the following statements equivalent:
    \begin{align*}
        \alpha(s') = 1 &\Leftrightarrow \alpha(s) = 1, \\
        \alpha(s') = t_i' & \Leftrightarrow \alpha(s) = t_i.
    \end{align*}
    Set $a' := (s', (t_i')_{i \in I}) \in \hat{\mathcal{T}}_{\mathbb{R}_{>0}}$. Then the first equivalence implies $(G^{\hat{T}_a})^{\circ} = (G^{\hat{T}_{a'}})^{\circ}$ by \cref{lemma: explicit descrpition of centralizer in reductive group}. It also implies that $\mathcal{B}^{\hat{T}_a} = \mathcal{B}^{\hat{T}_{a'}}$ by the Bruhat decomposition. The second equivalence implies $V^{\hat{T}_a} = V^{\hat{T}_{a'}}$. Since $\tilde{V} \subset \mathcal{B} \times V$, this also shows that $\tilde{V}^{\hat{T}_a} = \tilde{V}^{\hat{T}_{a'}}$. It follows from this that $\textbf{S}^a = \textbf{S}^{a'}$.
    
    Note that $\hat{T}_{a'}$ is a torus by \cref{lemma: positive real implies Ta is a torus} so $(G^{\hat{T}_{a'}})^{\circ} = G^{\hat{T}_{a'}}$. If $G^{\hat{T}_a}$ is connected, we get $G^{\hat{T}_a} = (G^{\hat{T}_a})^{\circ} \overset{\cref{lemma: explicit descrpition of centralizer in reductive group}}{=} (G^{\hat{T}_{a'}})^{\circ} = G^{\hat{T}_{a'}}$ and thus $A(a,x) = A(a',x)$ for all $x \in \mathfrak{N}(V)^{\hat{T}_a} = \mathfrak{N}(V)^{\hat{T}_{a'}}$. Hence, we have $\mathcal{P}(G,V,a) = \mathcal{P}(G,V,a')$.
\end{proof}
\begin{theorem}\label{thm: general DL correspondence}
    Let $V$ be an $I$-rooted $G$-representation and let $a = (s,(t_i)_{i \in I}) \in \hat{\mathcal{T}}$. Assume that
    \begin{itemize}
        \item $G$ has simply connected derived subgroup,
        \item $\langle t_i \mid i \in I \rangle \subset \mathbb{C}^{\times}$ is torsion-free,
        \item (A1)-(A3) hold for $((G^{S_a})^{\circ}, V^{S_a})$.
    \end{itemize}
    Then the number of $G^{\hat{T}_a}$-orbits on $\mathfrak{N}(V)^{\hat{T}_a}$ is finite and there is a canonical bijection
    \begin{equation*}
        \Irr(\mathcal{H}^{\aff}_{\chi_a}) \overset{1:1}{\leftrightarrow} \mathcal{P}(G,V,a).
    \end{equation*}
    If in addition (A2') holds for $((G^{S_a})^{\circ}, V^{S_a})$, then there is a bijection
    \begin{equation*}
        \Irr(\mathcal{H}^{\aff}_{\chi_a}) \overset{1:1}{\leftrightarrow}  \mathcal{P}(G,V,a) \overset{1:1}{\leftrightarrow} \mathfrak{N}(V)^{\hat{T}_a}/ G^{\hat{T}_a}.
    \end{equation*}
\end{theorem}
\begin{proof}
    By \cref{lemma: explicit descrpition of centralizer in reductive group} the roots in $(G^{S_a})^{\circ}$ are precisely the $\alpha \in \Phi$ with $\alpha(s) \in \langle t_i \mid i \in I \rangle$. Hence, \cref{lemma: condition when a can be replaced with real element} applies to the triple $((G^{S_a})^{\circ},V^{S_a},a)$ which yields an element $a' \in \hat{\mathcal{T}}_{\mathbb{R}_{>0}}$ such that
    \begin{equation}\label{eq: G and nilcone at a is the same when taking Sa fixed points}
        \begin{aligned}
        ((G^{S_a})^{\circ})^{\hat{T}_{a'}} &= ((G^{S_a})^{\circ})^{\hat{T}_a} \overset{\cref{lemma: simply connected implies double fixed points well behaved}}{=}  G^{\hat{T}_a}  \\
        \mathfrak{N}(V^{S_a})^{\hat{T}_{a'}} &= \mathfrak{N}(V^{S_a})^{\hat{T}_a} \overset{\eqref{eq: Ta fixed points on nilcone for Sa and without}}{=} \mathfrak{N}(V)^{\hat{T}_a}  .
        \end{aligned}
    \end{equation}
    Since (A1) holds for $((G^{S_a})^{\circ},V^{S_a})$ there are only finitely  many $((G^{S_a})^{\circ})^{\hat{T}_{a'}}$-orbits on $\mathfrak{N}(V^{S_a})^{\hat{T}_{a'}}$ by \cref{cor: a positive real implies orbit finiteness}. By \eqref{eq: G and nilcone at a is the same when taking Sa fixed points} this implies that there are only finitely many $G^{\hat{T}_a}$-orbits on $\mathfrak{N}(V)^{\hat{T}_a}$.

    The group $((G^{S_a})^{\circ})^{\hat{T}_a}=G^{\hat{T}_a}$ is connected by \cref{lemma: simply connected implies connected centralizer}. Hence, \cref{lemma: condition when a can be replaced with real element} yields that
    \begin{equation}\label{eq: parameter spaces at Sa for a and a prime}
        \mathcal{P}((G^{S_a})^{\circ}, V^{S_a}, a)  =\mathcal{P}((G^{S_a})^{\circ}, V^{S_a}, a').
    \end{equation}
    Let $\bar{\mathcal{H}}^{\aff}_{\textbf{q}_{V^{S_a}}}$ be the affine Hecke algebra associated to $(G^{S_a})^{\circ}$ with parameter function $\textbf{q}_{V^{S_a}}$ coming from $V^{S_a}$. Since we assume that (A1), (A2) and (A3) hold for $((G^{S_a})^{\circ}, V^{S_a})$ and $a'$ is positive real, we can apply \cref{cor: DL correspondonce without simply connected for positive real central character} which yields a bijection
    \begin{equation}\label{eq: Irr for bar hecke}
        \Irr(\bar{\mathcal{H}}^{\aff}_{\chi_{a'}}) \overset{1:1}{\leftrightarrow}\mathcal{P}((G^{S_a})^{\circ}, V^{S_a}, a').
    \end{equation}
    Note that by \cref{thm: affine Hecke algebra at central character is Ext algebra} we also have
    \begin{equation}\label{eq: irr for bar Hecke and bar Sa}
        \Irr(\bar{\mathcal{H}}^{\aff}_{\chi_{a'}}) \overset{1:1}{\leftrightarrow} \Irr(\bar{\textbf{S}}^{a'}).
    \end{equation}
    Hence, we get
    \begin{align*}
        \# \Irr(\mathcal{H}^{\aff}_{\chi_a}) \overset{\cref{thm: affine Hecke algebra at central character is Ext algebra}}&{=} \#\Irr(\textbf{S}^a) \\
        \overset{\cref{lemma: Sa and bar Sa are the same up to dir sum}}&{=} \# \Irr(\bar{\textbf{S}}^a) \\
        \overset{\cref{lemma: condition when a can be replaced with real element}}&{=}\# \Irr(\bar{\textbf{S}}^{a'}) \\
        \overset{\eqref{eq: Irr for bar hecke}, \eqref{eq: irr for bar Hecke and bar Sa}}&{=}\#\mathcal{P}((G^{S_a})^{\circ}, V^{S_a}, a') \\
        \overset{\eqref{eq: parameter spaces at Sa for a and a prime}}&{=}\#\mathcal{P}((G^{S_a})^{\circ}, V^{S_a}, a)  \\
        \overset{\cref{lemma: Sa and bar Sa are the same up to dir sum}}&{=}\# \mathcal{P}(G,V,a).
    \end{align*}
    This shows that the canonical map $\Irr(\mathcal{H}^{\aff}_{\chi_a})\hookrightarrow \mathcal{P}(G, V, a)$ from \eqref{eq: injective DL parametrization} is a bijection.

    Assume now that in addition (A2') holds $((G^{S_a})^{\circ}, V^{S_a})$. Then the projection map
    \begin{equation*}
        \mathcal{P}(G,V,a) \rightarrow \mathfrak{N}(V)^{\hat{T}_a}/G^{\hat{T}_a}
    \end{equation*} 
    is the same as the projection map
    \begin{equation*}
        \mathcal{P}((G^{S_a})^{\circ}, V^{S_a}, a') \rightarrow \mathfrak{N}(V^{S_a})^{\hat{T}_{a'}}/ ((G^{S_a})^{\circ})^{\hat{T}_{a'}}
    \end{equation*}
    by \cref{lemma: Sa and bar Sa are the same up to dir sum}, \eqref{eq: parameter spaces at Sa for a and a prime} and \eqref{eq: G and nilcone at a is the same when taking Sa fixed points}. This is a bijection by \cref{cor: DL correspondonce without simply connected for positive real central character}. Hence, we get
    \begin{equation*}
        \Irr(\mathcal{H}^{\aff}_{\chi_a})\overset{1:1}{\leftrightarrow} \mathcal{P}(G, V, a) \overset{1:1}{\leftrightarrow} \mathfrak{N}(V)^{\hat{T}_a}/G^{\hat{T}_a}.
    \end{equation*}
\end{proof}
\section{The affine Hecke algebra of $G_2$}\label{section: affine Hecke for G2}
In this section we show that the conditions of \cref{thm: general DL correspondence} hold for $G_2$.
\begin{lemma}\label{lemma: conditions A hold for type A}
    Let $G$ be a reductive group such that all simple components of $G$ are of type $A$. Then (A1)-(A3) and (A2') hold for any rooted $G$-representation $V$.
\end{lemma}
\begin{proof}
    By \cref{lemma: Geometric conditions compatible with isogeny} we may replace $G$ first with its derived subgroup and then its simply connected cover. Hence, we may assume $G \cong SL_{n_1} \times ... \times SL_{n_l}$. For each $i = 1,..., l$, let $V_i \subset V$ be the subrepresentation of $V$ generated by the weight spaces corresponding to the roots of $SL_{n_i}$. Then the weights of $V_i$ are the roots of $SL_{n_i}$. The canonical morphism $V_1 \oplus ... \oplus V_l \rightarrow V$ is an isomorphism on all non-zero weight spaces, so by \cref{lem: conditions Ai are equivalent for rooted morphisms} we may replace $V$ with $V_1 \oplus ... \oplus V_l$. Note that each $SL_{n_i}$ acts trivially on $V_j$ for $j \neq i$. Hence, we may assume $G = SL_n$. There is an irreducible representation of $SL_n$ whose weights are precisely the roots of $SL_n$. This irreducible representation appears as a subquotient of both $V$ and $\mathfrak{sl}_n$. Hence, by \cref{lem: conditions Ai are equivalent for rooted morphisms} we may assume $V= \mathfrak{sl}_n$. For $\mathfrak{sl}_n$ the conditions (A1)-(A3) are well known: The nilpotent orbits are determined by Jordan normal form, so there are $\# \Irr(S_n)$ many of them proving (A1). The action of $A_{SL_n}(x)$ on $H_*(\mathcal{B}_x)$ is trivial since it factors through $A_{GL_n}(x) = \{ e \}$. This proves (A2') and (A2). Finally, it is known that nilpotent orbits in type A admit an affine paving in all characteristics: For any $x \in \mathfrak{N}(\mathfrak{sl}_n)$, the element $1 +x \in SL_n$ is unipotent and $\mathcal{B}_x$ is equal to the Springer fiber of $1+x$ which has an affine paving by \cite[5.9]{spaltenstein2006classes}. This implies (A3).
\end{proof}
For the rest of this section, let $G = G_2$ in characteristic $3$. We then have
\begin{equation*}
    \Phi^- = \{ \alpha, \beta, \alpha + \beta , 2\alpha + \beta , 3 \alpha + \beta, 3\alpha + 2\beta \}
\end{equation*}
where the sets of short and long negative roots are given by
\begin{align*}
    \Phi^-_s &= \{ \alpha, \alpha+\beta, 2 \alpha + \beta \} \\
    \Phi^-_l &= \{ \beta, 3\alpha + \beta, 3\alpha +2 \beta\}.
\end{align*}
We pick a Chevalley basis $\{ X_{\gamma}, H_{\alpha}, H_{\beta} \mid \gamma \in \Phi \}$ of $\mathfrak{g}$ and denote the corresponding root group homomorphism by $x_{\gamma} : \mathbb{G}_a \rightarrow U_{\gamma} \subset G$. The adjoint representation of $G$ fits into a short exact sequence
\begin{equation*}
    0 \rightarrow \mathfrak{g}_s \rightarrow \mathfrak{g}_s \rightarrow \mathfrak{g}/ \mathfrak{g}_s
\end{equation*}
where $\mathfrak{g}_s$ and $\mathfrak{g} / \mathfrak{g}_s$  are the simple $G$-modules of lowest weight $2 \alpha + \beta$ and $3 \alpha + 2 \beta$ respectively (c.f. \cite[Hauptsatz]{hiss1984adjungierten}). Then
\begin{equation*}
    V := \mathfrak{g}_s \oplus \mathfrak{g} / \mathfrak{g}_s
\end{equation*}
is an $I$-rooted $G$-representation for $I = \{ 1,2 \}$. Note that for each $\gamma \in \Phi$ there is a canonical isomorphism $V_{\gamma} \cong \mathfrak{g}_{\gamma}$ induced by the inclusion $\mathfrak{g}_s \hookrightarrow \mathfrak{g}$ and the projection $\mathfrak{g} \rightarrow \mathfrak{g}/\mathfrak{g}_s$. We denote the image of $X_{\gamma} \in \mathfrak{g}_{\gamma}$ under this isomorphism by $v_{\gamma} \in V_{\gamma}$. The following formulas are a direct consequence of Chevalley's formulas for the action of $G$ on the Chevalley basis:
\begin{equation}\label{eq: chevalley formulas}
\begin{aligned}
    x_{\beta}(t) v_{\alpha} &= v_{\alpha} +tv_{\alpha + \beta}\\
    x_{\alpha + \beta}(t) v_{\alpha} &= v_{\alpha} - tv_{2\alpha + \beta}\\
    x_{\alpha}(t) v_{\alpha + \beta} &= v_{\alpha+ \beta} + t v_{2 \alpha + \beta}\\
    x_{\alpha}(t) v_{\beta} &= v_{\beta} - t^3 v_{3 \alpha + \beta} \\
    x_{3\alpha + \beta}(t) v_{\beta} &= v_{\beta} -tv_{3\alpha + 2\beta}\\
    x_{\beta}(t) v_{3\alpha+ \beta} &= v_{3\alpha +\beta} +tv_{3\alpha + 2\beta}
\end{aligned}
\end{equation}
and 
\begin{equation}\label{eq: trivial chevalley formula}
    x_{\gamma_1}(t) v_{\gamma_2} = v_{\gamma_2}
\end{equation}
for $\gamma_1,\gamma_2 \in \Phi^-$ whenever $x_{\gamma_1}(t) v_{\gamma_2}$ does not appear in \eqref{eq: chevalley formulas}.
\begin{lemma}\label{lemma: Stabilizers in adjoint and split representations are the same}
    Let $\gamma \in \Phi_l$ and $x \in \mathfrak{g}_s$. Then for any $\lambda \in k$, the elements $x + \lambda  X_{\gamma} \in \mathfrak{g}$ and $x + \lambda v_{\gamma} \in V$ have the same stabilizer in $G$.
\end{lemma}
\begin{proof}
    After acting with an element of the Weyl group, we may assume that $\gamma = 3 \alpha + 2\beta$. If $\lambda = 0$, the claim is clear since $\mathfrak{g}_s \subset \mathfrak{g}$ is a subrepresentation. Let us now assume that $\lambda \neq 0$. Note that $U^- \cdot v_{\gamma } = v_{\gamma}$ by \eqref{eq: trivial chevalley formula}. Hence, if $u'\dot{w}ut \in G$ with $u , u' \in U^-$, $t \in T$ and $\dot{w} \in N_G(T)$ then $u'\dot{w}utv_{\gamma} = v_{\gamma}$ if and only if $w( \gamma) = \gamma$ and $\gamma(t) = 1$. Similarly, it follows from the Chevalley formulas that $u'\dot{w}utX_{\gamma} = X_{\gamma}$ if and only if $w( \gamma) = \gamma$ and $\gamma(t) = 1$. Thus, we get $G_{v_{\gamma}} = G_{X_{\gamma}}$. Since the $G$-equivariant map $\mathfrak{g} \rightarrow \mathfrak{g}/\mathfrak{g}_s$ sends $x+ \lambda X_{\gamma}$ to $ \lambda v_{\gamma}$, we also have
    \begin{equation*}
        G_{x + \lambda X_{\gamma}} \subset G_{\lambda v_{\gamma}} = G_{v_{\gamma}} = G_{X_{\gamma}}.
    \end{equation*}
    Hence, we get
    \begin{equation*}
        G_{x + \lambda X_{\gamma}}  = G_{x + \lambda X_{\gamma}}  \cap G_{X_{\gamma}} = G_x \cap G_{X_{\gamma}} = G_x \cap G_{v_{\gamma}} = G_{x + \lambda v_{\gamma}}.
    \end{equation*}
\end{proof}
\begin{lemma}\label{lemma: exotic nilpotent orbits for G2}
    (A1) holds for $(G,V)$, i.e. the number of $G$-orbits on $\mathfrak{N}(V)$ is finite. The following table lists orbit representatives together with their stabilizer dimensions and their components groups:

    \begin{table}
        \caption{Exotic nilpotent orbits for $G_2$}
        \begin{equation*}
        \begin{array}{c||c|c|c|c|c|c}
            \text{Representative }\xi & 0 & v_{\alpha} & v_{\beta}   & v_{\alpha + \beta}+ v_{\beta}  &  v_{2\alpha+ \beta}+ v_{\beta}  & v_{\alpha}+ v_{\beta} \\ \hline
            \text{dim }G_{\xi} & 14 & 8 & 8 & 6 & 4 & 2\\ \hline
            A(\xi) & 1 & 1 & 1 & 1 & S_2 & 1
        \end{array}
        \end{equation*}
    \end{table}
\end{lemma}
\begin{proof}
A straightforward computation (similar to the one in \cite[II 3)]{stuhler1971unipotente}) using the formulas in \eqref{eq: chevalley formulas} shows that any element in $\mathfrak{N}(V)$ is conjugate to one of listed representatives. Note that all the representatives are of the form $x + \lambda v_{\beta}$ for some $x \in \mathfrak{g}_s$ and $\lambda \in k$. Thus, by \cref{lemma: Stabilizers in adjoint and split representations are the same} their stabilizers agree with the stabilizer of $x + \lambda X_{\beta} \in \mathfrak{g}$. The dimensions of these stabilizers and their component groups have been computed and are is in our table (c.f. \cite[Table 12]{xue2017springer}). Alternatively, the stabilizers of the representatives can also be easily computed directly. For example, if $g \in G_{v_{2 \alpha + \beta} + v_{\beta}}$ is of the form $b\dot{w}b'$ with $b,b' \in B$ and $\dot{w} \in N_G(T)$, then we must have $w(2 \alpha + \beta) = 2 \alpha + \beta$ and $w(\beta) \in \Phi^-$ which implies $w = e$. Hence, we get
\begin{equation}\label{eq: centralizer for the non triv component group}
    G_{v_{2 \alpha + \beta} + v_{\beta}}  = B_{v_{2 \alpha + \beta} + v_{\beta}}  = \{e,t_0 \} \cdot U_{\beta} U_{\alpha + \beta} U_{ 2 \alpha + \beta} U_{3 \alpha + \beta}
\end{equation}
where $t_0 \in T$ is the unique element with $\alpha (t_0) = -1$ and $\beta(t_0) = 1$. In particular, $\dim G_{v_{2 \alpha + \beta} + v_{\beta}} = 4$ and $A(v_{2 \alpha + \beta} + v_{\beta}) \cong S_2$ is generated by $\bar{t}_0$.

It remains to show that the orbit representatives are pairwise not conjugate. This follows immediately by comparing the dimensions of their stabilizers except possibly for $v_{\alpha}$ and $v_{\beta}$. But $v_{\alpha}$ and $v_{\beta}$ are clearly not conjugate since they lie in $\mathfrak{g}_s$ and $\mathfrak{g}/\mathfrak{g}_s$ respectively.
\end{proof}
\begin{lemma}\label{lemma: A3 and A3 hold for G2}
    (A3) and (A2') (and thus also (A2)) hold for $(G,V)$.
\end{lemma}
\begin{proof}
    It is clear that $\mathcal{B}_0 = \mathcal{B}$. Let $g = b\dot{w} b' \in G$ with $b,b' \in B$. If $g^{-1}(v_{\alpha} + v_{\beta}) \in V^-$, then $w^{-1}( \alpha ) \in \Phi^-$ and $w^{-1} ( \beta) \in \Phi^-$ which implies $w = e$. Hence, $\mathcal{B}_{v_{\alpha} + v_{\beta}} = B/B$. It follows from \eqref{eq: trivial chevalley formula} that $g^{-1} v_{3 \alpha + 2 \beta} \in V^-$ if and only if $w^{-1}(3 \alpha + 2 \beta) \in \Phi^-$. Similarly, we have $g^{-1} v_{2 \alpha +  \beta} \in V^-$ if and only if $w^{-1}(2 \alpha +  \beta) \in \Phi^-$. Hence, we get
    \begin{align*}
        \mathcal{B}_{v_{\alpha}} & \cong \mathcal{B}_{v_{2 \alpha + \beta}} = \bigsqcup_{w^{-1} (2 \alpha + \beta) \in \Phi^-} B\dot{w}B/B \\
        \mathcal{B}_{v_{\beta}} & \cong \mathcal{B}_{v_{3 \alpha + 2 \beta}} = \bigsqcup_{w^{-1} (3 \alpha + 2\beta) \in \Phi^-} B\dot{w}B /B .
    \end{align*}
    Note that $v_{\alpha + \beta} + v_{\beta}$ and $v_{3 \alpha + 2 \beta} +v_{2 \alpha + \beta}$ are in the same $G$-orbit. Hence, we have 
    \begin{align*}
        \mathcal{B}_{v_{\alpha + \beta} + v_{\beta}}  & \cong \mathcal{B}_{v_{3 \alpha + 2 \beta} +v_{2 \alpha + \beta}} \\
        & = \mathcal{B}_{v_{3 \alpha + 2 \beta}} \cap \mathcal{B}_{v_{2 \alpha + \beta}}\\
        &=  \bigsqcup_{w^{-1}  (3 \alpha + 2 \beta), w^{-1}  (2 \alpha + \beta) \in \Phi^-} B\dot{w}B/B.
    \end{align*}
    If $g^{-1}(v_{2 \alpha + \beta} + v_{\beta}) \in V^-$ then $w^{-1} ( 2 \alpha + \beta) \in \Phi^-$ and $ w^{-1}(\beta) \in \Phi^-$ which implies $w^{-1} \in \{e, s_{\alpha},  s_{\alpha}s_{\beta} \}$. A straightforward computation with the Bruhat decomposition shows that
    \begin{equation*}
        \mathcal{B}_{v_{2 \alpha + \beta} + v_{\beta}} = B/B \sqcup U_{\alpha} \dot{s}_{\alpha}B/B \sqcup U_{3\alpha + \beta} \dot{s}_{\alpha} \dot{s}_{\beta}B/B.
    \end{equation*}
    Hence, for all $x \in \mathfrak{N}(V)$ the variety $\mathcal{B}_x$ has an affine paving. This implies (A3).
    
    To prove (A2'), we want to determine the action of $S_2 \cong A( v_{2 \alpha + \beta} + v_{\beta})$ on $H_*(\mathcal{B}_{v_{2 \alpha + \beta} + v_{\beta}} )$. Recall from \eqref{eq: centralizer for the non triv component group} that the non-trivial element in $ A(v_{2 \alpha + \beta} + v_{\beta})$ lifts to the unique element $t_0 \in T $ with $\alpha(t_0)=-1$ and $\beta(t_0) = 1$. The vector space $H_*(\mathcal{B}_{v_{2 \alpha + \beta} + v_{\beta}} )$ is three dimensional with basis given by the fundamental classes
    \begin{equation*}
        [B/B], [B/B \sqcup U_{\alpha}\dot{s}_{\alpha}B/B],  [B/B  \sqcup U_{3\alpha + \beta} \dot{s}_{\alpha} \dot{s}_{\beta}B/B].
    \end{equation*}
    The element $t_0$ fixes each of these fundamental classes point-wise. This shows that $A( v_{2 \alpha + \beta} + v_{\beta})$ acts trivially on $H_*(\mathcal{B}_{v_{2 \alpha + \beta} + v_{\beta}} )$. Note that the Weyl group of $G_2$ has $6$ irreducible representations which is equal to the number of orbits of $\mathfrak{N}(V)$. Hence, we have proved (A2').
\end{proof}

\begin{theorem}\label{theorem: dl correspondence for G2}
    Consider the representation $V = \mathfrak{g}_s \oplus \mathfrak{g}/\mathfrak{g}_s$ of $G = G_2$ defined over an algebraically closed field $k$ of characteristic $3$. Let $\mathcal{H}^{\aff}_{q_1, q_2}$ be the affine Hecke algebra of type $G_2$ with two parameters (where short roots have parameter $q_1$ and long roots have parameter $q_2$). Then for any $a = (s, t_1, t_2) \in \hat{\mathcal{T}}$ such that $\langle t_1, t_2 \rangle \subset \mathbb{C}^{\times}$ is torsion-free, there is a bijection
    \begin{equation*}
        \Irr(\mathcal{H}^{\aff}_{\chi_a}) \overset{1:1}{\leftrightarrow} \mathfrak{N}(V)^{\hat{T}_a}/G^{\hat{T}_a}.
    \end{equation*}
\end{theorem}
\begin{proof}
    By \cref{thm: general DL correspondence}, it suffices to show that (A1), (A2') and (A3) hold for $((G^{S_a})^{\circ}, V^{S_a})$. Note that $(G^{S_a})^{\circ} $ and $ G$ share a maximal torus, so the root system of $(G^{S_a})^{\circ}$ is a closed subroot system of the root system of $G$. The proper closed subroot systems of the root system of type $G_2$ only contain components of type $A$. Thus, we either have $G^{S_a} = G$ and $V = V^{S_a}$ or all simple components of $(G^{S_a})^{\circ}$ are of type $A$. In the former case (A1), (A2') and (A3) hold by \cref{lemma: exotic nilpotent orbits for G2} and \cref{lemma: A3 and A3 hold for G2} and in the latter case they hold by \cref{lemma: conditions A hold for type A}.
\end{proof}
Note that the correspondence above does not involve local systems. In particular, it is different from the classical Deligne-Langlands correspondence \cite{kazhdan1987proof,chriss2009representation} for equal parameters.
\begin{example}\label{example: comparison with classical DL corr}
    Let $s \in \mathcal{T}$ such that $\alpha(s) = 1$ and $\beta (s) = q$ for some $q \in \mathbb{C}^{\times}$ not a root of unity. Then for $a = (s, q, q) \in \hat{\mathcal{T}}$, we have
    \begin{equation*}
        \mathfrak{N}(V)^{\hat{T}_a} = V^{\hat{T}_a} = V_{\beta} \oplus V_{\alpha + \beta} \oplus V_{2 \alpha + \beta} \oplus V_{3 \alpha + \beta}.
    \end{equation*}
    Using the formulas in \eqref{eq: chevalley formulas} one can easily check that this space has $5$ orbits under the $G^{\hat{T}_a}$-action with representatives
    \begin{equation*}
        0, v_{3 \alpha + \beta}, v_{2 \alpha + \beta}, v_{2 \alpha + \beta} + v_{3\alpha + \beta}, v_{\alpha + \beta} + v_{3 \alpha + \beta}.
    \end{equation*}
    Thus, there are $5$ simple $\mathcal{H}^{\aff}_{\chi_a}$-modules. Note that component group $A(a, v_{\alpha + \beta} + v_{3 \alpha + \beta})$ is isomorphic to $S_2$ but the corresponding non-trivial local system does not contribute to our parameterization. On the other hand, the fixed points space of the (complex) nilpotent cone
    \begin{equation*}
        \mathfrak{N}(\mathfrak{g}_{\mathbb{C}})^a = \mathfrak{g}_{\mathbb{C}}^a = \mathfrak{g}_{\beta} \oplus  \mathfrak{g}_{\alpha + \beta} \oplus  \mathfrak{g}_{2 \alpha + \beta} \oplus  \mathfrak{g}_{3 \alpha + \beta}
    \end{equation*}
    has $4$ orbits under the $G_{\mathbb{C}}^a$-action. In the classical Deligne-Langlands correspondence, each orbit, together with the trivial local system, corresponds to an irreducible representation which accounts for $4$ of the $5$ irreducible $\mathcal{H}^{\aff}_{\chi_a}$-representations here. The component group $A(a, X_{\beta} + X_{2 \alpha + \beta} )$ is isomorphic to $S_3$ and the remaining fifth irreducible representation corresponds to the non-trivial local system associated to the partition $(21)$ (c.f. \cite[Table 6.1, central character $t_e$]{ram2003representations}).
\end{example}
The assumption that $\langle t_1, t_2 \rangle$ is torsion-free cannot be removed from \cref{theorem: dl correspondence for G2}. In fact, for arbitrary parameters it is no longer true that $\mathfrak{N}(V)^{\hat{T}_a}$ has only finitely many $G^{\hat{T}_a}$-orbits.
\begin{example}\label{example: example where fixed point space has infinitely many orbits}
    Let $q \in \mathbb{C}^{\times}$ be not a root of unity and $s \in \mathcal{T}$ such that $\beta(s) = q$ and $\alpha(s) = \zeta$ for a primitive third root of unity $\zeta \in \mathbb{C}^{\times}$. Then setting $a = (s, \zeta q, q) \in \hat{\mathcal{T}}$, we get
    \begin{equation*}
        \mathfrak{N}(V)^{\hat{T}_a} = V^{\hat{T}_a} =  V_{\beta} \oplus V_{\alpha + \beta} \oplus V_{3\alpha + \beta}.
    \end{equation*}
    Moreover, we have $G^{\hat{T}_a} = T$. Since $\dim T < \dim \mathfrak{N}(V)^{\hat{T}_a}$ there are infinitely many $G^{\hat{T}_a}$-orbits on $\mathfrak{N}(V)^{\hat{T}_a}$. However, since $\mathcal{H}^{\aff}_{\chi_a}$ is finite-dimensional, there are only finitely many simple $\mathcal{H}^{\aff}_{\chi_a}$-modules.
\end{example}

\begin{remark}\label{remark: conjecture of Xi}
    It would be interesting to determine precisely for which $t_1, t_2$ with $\langle t_1, t_2 \rangle$ not necessarily torsion-free the parameterization of simple modules in \cref{theorem: dl correspondence for G2} (and more generally in \cref{thm: general DL correspondence}) still holds. The conjecture in \cite[5.20]{xi2006representations} suggests that for $G_2$ a sufficient condition for this should be that $(t_1, t_2)$ is not a root of the polynomial
    \begin{equation*}
        f(\textbf{u},\textbf{v}) = (1+\textbf{u})(1+\textbf{v})(1+\textbf{u}\textbf{v}+\textbf{u}^2\textbf{v}^2).
    \end{equation*}
    Unfortunately, this turns out to be too optimistic for our geometric parameterization: Taking $(t_1,t_2) = ( \zeta q, q)$ where $q \in \mathbb{C}^{\times}$ is of infinite order, $\zeta \in \mathbb{C}^{\times}$ is a primitive third root of unity, we get $f(\zeta q, q ) \neq 0$ but for $a= (s, \zeta q, q)$ with $\alpha(s) = \zeta$, $\beta(s) = q$ the space $\mathfrak{N}(V)^{\hat{T}_a}$ has infinitely many $G^{\hat{T}_a}$-orbits by \cref{example: example where fixed point space has infinitely many orbits}. Note however, that $f(\zeta q, q^{-1}) = 0$ which suggests that the condition $f(u,v) \neq 0$ for all $u,v \in \langle t_1, t_2 \rangle$ may be sufficient.
\end{remark}

\appendix
\numberwithin{equation}{section}
\section{Convolution, localization and the Chern character}\label{section: appendix}
This appendix summarizes some standard facts about equivariant $K$-theory and Borel-Moore homology. For $k = \mathbb{C}$ most of these results can be found in \cite{chriss2009representation} but the arguments also work in positive characteristic.

\subsection{Equivariant $K$-theory}\label{section: app k theory}
Let $G$ be an algebraic group (i.e. a group object in the category of varieties) or a diagonalizable group scheme defined over an algebraically closed field $k$ (not necessarily of characteristic $0$) and let $X$ be a $G$-variety. All our varieties are assumed to be quasi-projective. We denote by $\Coh^G(X)$ the category of $G$-equivariant coherent sheaves, by $D^b(\Coh^G(X))$ its bounded derived category and by $D^b_{\perf}(\Coh^G(X))$ the category of perfect complexes. Consider the Grothendieck groups
\begin{align*}
    K^G(X) &:= K_0( D^b(\Coh^G(X))) \\
    K_G(X) &:= K_0(D^b_{\perf}(\Coh^G(X))).
\end{align*}
There are the usual pushforward and pullback maps
\begin{align*}
    f_* &: K^G(X) \rightarrow K^G(Y)\\
    f^* &: K^G(Y) \rightarrow K^G(X)
\end{align*}
which are defined for $f: X \rightarrow Y$ proper (resp. flat) and satisfy base change. Moreover, there is a pullback map
\begin{equation*}
    f^* : K_G(Y) \rightarrow K_G(X)
\end{equation*}
which is defined without any assumptions of $f$. The tensor product induces maps
\begin{align*}
    - \otimes - &: K_G(X) \times K_G(X) \rightarrow K_G(X) \\
    - \otimes - &: K^G(X) \times K_G(X) \rightarrow K^G(X).
\end{align*}
This equips $K_G(X)$ with a ring structure and $K^G(X)$ with the structure of a $K_G(X)$-module. For any proper $f: X\rightarrow Y $ and $[\mathcal{F}] \in K^G(X), [\mathcal{E}] \in K_G(Y)$, we have the projection formula
\begin{equation}\label{eq: projection formula for coherent sheaves and vector bundles}
    f_*[\mathcal{F}] \otimes [\mathcal{E}] = f_*([\mathcal{F}] \otimes f^*[\mathcal{E}]).
\end{equation}
If $X$ is smooth, the canonical map
\begin{equation*}
    [\mathcal{O}_X] \otimes - : K_G(X) \rightarrow K^G(X)
\end{equation*}
is an isomorphism. If $V \rightarrow X$ is a $G$-equivariant vector bundle, the pullback map
\begin{equation*}
    \pi^* : K^G(X) \rightarrow K^G(V)
\end{equation*}
is an isomorphism called the Thom isomorphism. If $H \subset G$ is a closed subgroup and $X$ is an $H$-variety such that $G \times^H X$ exists as a variety, then there is an induction isomorphism
\begin{equation}\label{eq: induction iso}
    K^G(G \times^H X) \cong K^H(X)
\end{equation}
which is compatible with proper pushforward and flat pullback.

For any $G$-stable closed subvariety $Z \subset  X$, we denote by $\iota_Z : Z \hookrightarrow X$ the inclusion. Let $\Coh^G_Z(X) \subset \Coh^G(X)$ be the full subcategory consisting of those coherent sheaves $\mathcal{F}$ on $X$ which are (set-theoretically) supported on $Z$ (i.e. the restriction of $\mathcal{F}$ to $X \backslash Z$ vanishes). Similarly, we denote by $D^b_Z(\Coh^G(X)) \subset D^b(\Coh^G(X))$ the full subcategory consisting of those complexes whose cohomology lies in $\Coh^G_Z(X)$. The $K$-group with support is defined as
\begin{equation*}
    K^G_Z(X) := K_0(D^b_Z(\Coh^G(X))).
\end{equation*}
Pushing forward along the inclusion $ Z \hookrightarrow X$ yields an isomorphism
\begin{equation*}
    K^G(Z) \overset{\sim}{\rightarrow} K_Z^G(X).
\end{equation*}
We will from now on implicitly identify $K^G(Z)$ and $ K_Z^G(X)$ via this isomorphism. If $X$ is smooth, we have the derived tensor product
\begin{equation*}
    \otimes_X^L : D^b(\Coh^G(X)) \times D^b(\Coh^G(X)) \rightarrow D^b(\Coh^G(X)).
\end{equation*}
Passing to $K$-theory and taking supports into account, this induces a tensor product
\begin{equation*}
    \otimes^L_X: K^G(Z_1) \times K^G(Z_2) \rightarrow K^G(Z_1 \cap Z_2)
\end{equation*}
for any  $Z_1, Z_2 \subset X$ closed. For locally free sheaves on $X$, the tensor product with support reduces to the ordinary tensor product in the following sense.
\begin{lemma}\label{lemma: cap product with support for vector bundle}
    Let $X$ be smooth and $ \iota_Z: Z \hookrightarrow X$ a closed subvariety. Then
    \begin{equation*}
       [\mathcal{F}] \otimes^L_X [\mathcal{E}] = [\mathcal{F} \otimes \iota_Z^* \mathcal{E}]
    \end{equation*}
    for any $[\mathcal{F}]  \in K^G(Z)$ and $[\mathcal{E}] \in K_G(X)\cong K^G(X)$.
\end{lemma}
The tensor product with support behaves well with respect to transverse intersections.
\begin{lemma}\label{lemma: transversal derived tensor product is trivial}
    Let $Z_1,Z_2 \subset X$ be smooth closed subvarieties of a smooth variety $X$. Assume that $Z_1$ and $Z_2$ intersect transversally, i.e. for all $p \in Z_1 \cap Z_2$, we have the equality of Zariski tangent spaces $T_p Z_1 + T_pZ_2 = T_p X$. Then we have
    \begin{equation*}
        [\mathcal{O}_{Z_1}] \otimes^L_X [\mathcal{O}_{Z_2}] = [\mathcal{O}_{Z_1 \cap Z_2}]
    \end{equation*}
    where the tensor product with support is taken with respect to $[\mathcal{O}_{Z_1}] \in K^G(Z_1)$ and $[\mathcal{O}_{Z_2}] \in K^G(Z_2)$.
\end{lemma}
If $f: \tilde{X} \rightarrow X$ is a morphism between smooth varieties there is a pullback functor
\begin{equation*}
    Lf^* : D^b(\Coh^G(X)) \rightarrow D^b(\Coh^G(\tilde{X})).
\end{equation*}
Passing to $K$-theory and taking supports into account, this induces a `pullback with support' map on $K$-theory
\begin{equation*}
    f^* : K^G(Z) \rightarrow K^G(f^{-1}(Z))
\end{equation*}
for any $Z \subset X$ closed. The pullback with support construction is functorial (i.e. $f^* \circ g^*  = (g \circ f)^*$) and if $f$ is flat, it agrees with the ordinary flat pullback. Moreover, pullback with support essentially commutes with the tensor product with support in the following sense (c.f. \cite[(6.4.1)]{nakajima2001quiver}).
\begin{lemma}\label{lemma: K theory restriction with supp commutes with tensor product with supp}
    Let $f : \tilde{X} \rightarrow X$ be a morphism between smooth varieties and let $Z_1, Z_2 \subset X$ be closed. Then for any $[\mathcal{F}_1] \in K^G(Z_1)$ and $[\mathcal{F}_2] \in K^G(Z_2)$ we have
    \begin{equation*}
        f^*([\mathcal{F}_1] \otimes^L_X [\mathcal{F}_2]) = f^*[\mathcal{F}_1] \otimes^L_{\tilde{X}} f^*[\mathcal{F}_2].
    \end{equation*}
\end{lemma}

There is also the following `projection formula with support' (c.f. \cite{chriss2009representation}, \cite[(6.5.1)]{nakajima2001quiver}).
\begin{lemma}\label{lemma: K-theory proj form with supp}
    Let $f : \tilde{X} \rightarrow X$ be a morphism between smooth varieties and let $Z \subset X$ and $\tilde{Z} \subset \tilde{X}$ be closed subvarieties such that the restriction $f: \tilde{Z} \rightarrow f(\tilde{Z})$ (and hence also $f: \tilde{Z} \cap f^{-1}(Z) \rightarrow f(\tilde{Z}) \cap Z$) is proper. Then
    \begin{equation*}
        f_*[\mathcal{F}] \otimes^L_X [\mathcal{G}] = f_*([\mathcal{F}]  \otimes^L_{\tilde{X}}  f^*[ \mathcal{G} ])
    \end{equation*}
    for any $[\mathcal{F}]  \in K^G(\tilde{Z})$ and $[\mathcal{G}] \in K^G(Z)$.
\end{lemma}
Let $X_1, X_2 , X_3$ be smooth varieties, let $Z_{12} \subset X_1 \times X_2$ and $Z_{23} \subset X_2 \times X_3$ be closed subvarieties such that $p_{13}: p_{12}^{-1}(Z_{12}) \cap p_{23}^{-1}(Z_{23}) \rightarrow X_1 \times X_3$ is proper and define
\begin{equation*}
    Z_{12} \circ Z_{23} := p_{13}(p_{12}^{-1}(Z_{12}) \cap p_{23}^{-1}(Z_{23})).
\end{equation*}
Then there is a convolution operation
\begin{equation}\label{eq: convolution def}
    \star: K^{G} (Z_{12}) \times K^{G}(Z_{23}) \rightarrow K^{G} (Z_{12} \circ Z_{23})
\end{equation}
defined via
\begin{equation*}
    [\mathcal{F}] \star [\mathcal{G}] := (p_{13})_*( p_{12}^* [\mathcal{F}] \otimes^L_{X_1 \times X_2 \times X_3} p_{23}^* [\mathcal{G}]).
\end{equation*}
One can check using \cref{lemma: K-theory proj form with supp} that this operation is associative. Moreover, convolution with the diagonal has the following concrete description (c.f. \cite[Lemma 8.1.1]{nakajima2001quiver}).
\begin{lemma}\label{lemma: convolution along diagonal}
    Let $X_1, X_2$ be smooth, $Z \subset X_1 \times X_2$ closed and $\Delta_{X_1}: X_1 \rightarrow X_1 \times X_1$ the diagonal embedding. Then for any $[\mathcal{E}] \in K_G(X_1) \cong K^G(X_1)$ and $[\mathcal{F}] \in K^G(Z_{23})$ we have
    \begin{equation*}
        (\Delta_{X_1})_* [\mathcal{E}] \star [\mathcal{F}] =  [p_1^*\mathcal{E}] \otimes [\mathcal{F}]
    \end{equation*}
    where $p_1: Z\rightarrow X_1$ is the projection onto the first factor.
\end{lemma}

    Let $T$ be a torus over $k$ and define the corresponding complex torus
    \begin{equation*}
        \mathcal{T} := Spec(\mathbb{C}[X^*(T)]) = \mathbb{C}^{\times} \otimes X_*(T).
    \end{equation*}
    Note that there is a canonical isomorphism $X^*(\mathcal{T}) \cong X^*(T)$. Any $a \in \mathcal{T}$ defines an evaluation map
    \begin{equation*}
       ev_a : R(T)=R(\mathcal{T}) =\mathbb{Z}[X^*(\mathcal{T})] \rightarrow \mathbb{C}. 
    \end{equation*}
    We define
    \begin{align*}
    \rho_a &:= \ker(ev_a)\\
        X_a^* &:= \{ \lambda \in X^*(\mathcal{T}) \mid \lambda - 1 \in \rho_a\} =\{ \lambda \in X^*(\mathcal{T}) \mid \lambda(a) = 1 \}.
    \end{align*}
    Then $X_a^*$ is a subgroup of $ X^*(T)$ and $\rho_a \subset R(T)$ is a prime ideal. Moreover, there is an associated diagonalizable subgroup scheme of $T$
    \begin{equation*}
        T_a := Spec(k[X^*(T)/X_a^*]).
    \end{equation*}
    If $Y$ is smooth, then $Y^{T_a}$ is also smooth (see \cref{lemma: taking fixed points preserves smoothness}) and we can consider the element 
    \begin{equation*}
        \lambda(Y) := \sum_{i \ge 0} (-1)^i [\wedge^i \mathcal{N}^{\vee}_{Y^{T_a}/Y}] \in K_{T_a}(Y^{T_a})
    \end{equation*}
    where $\mathcal{N}^{\vee}_{Y^{T_a}/Y} \in \Coh^{T_a}(Y^{T_a})$ is the conormal sheaf. We then have the following special case of Thomason's localization theorem.
    \begin{theorem}\label{thm: localization in equivariant K theory}\cite{thomason1992formule}
    Let $Y$ be a $T$-variety.
    \begin{enumerate}[label = (\roman*)]
        \item The inclusion $\iota: Y^{T_a} \hookrightarrow Y$ induces an isomorphism of localized $K$-groups
    \begin{equation*}
        \iota_* : K^{T_a}(Y^{T_a})_{\rho_a} \rightarrow K^{T_a}(Y)_{\rho_a}.
    \end{equation*}
        \item If $Y$ is smooth, then $\lambda(Y)$ is invertible in $ K_{T}(Y^{T_a})_{\rho_a}$ and the composition
        \begin{equation*}
            \iota^* \iota_* : K^{T_a}(Y^{T_a})_{\rho_a} \rightarrow K^{T_a}(Y^{T_a})_{\rho_a}
        \end{equation*}
        is given by multiplication with $\lambda(Y)$.
        \item If $Z \subset Y$ is a closed subvariety with $Y$ smooth, then
        \begin{equation*}
            \iota^* \iota_* : K^{T_a}(Z^{T_a})_{\rho_a} \rightarrow K^{T_a}(Z^{T_a})_{\rho_a}
        \end{equation*}
        is still given by multiplication with $\lambda(Y)$ where $\iota^*$ is the restriction with support.
    \end{enumerate}
    \end{theorem}
    \begin{proof}
        The first claim is \cite[Thm. 2.1]{thomason1992formule}. The second claim is proved in \cite[Lemma 3.3]{thomason1992formule} by observing that $\iota^*\iota_*$ is the same as tensoring with the element $\iota^* \iota_* \mathcal{O}_{Y^{T_a}}$ which corresponds to $\lambda(Y)$ in the Grothendieck group. Since $\iota^* \iota_*$ in our third claim is defined by restricting $\iota^* \iota_* : D^b(\Coh^{T_a}(Y^{T_a})) \rightarrow D^b(\Coh^{T_a}(Y^{T_a}))$ to $D^b_{Z^{T_a}}(\Coh^{T_a}(Y))$, the map is  still given by multiplication with $\lambda(Y)$.
    \end{proof}
    
    Localization is also compatible with convolution in the following sense: Let $X_1, X_2, X_3$ be smooth $T$-varieties and $Z_{12}\subset X_1 \times X_2$, $Z_{23} \subset X_2 \times X_3$ closed $T$-stable subvarieties such that $p_{13} : Z_{12} \times X_2 \cap X_1 \times Z_{23} \rightarrow Z_{12} \circ Z_{23}$ is proper. Then by \cref{thm: localization in equivariant K theory} the composition
    \begin{equation*}
        r_a: K^{T_a}(Z_{ij})_{\rho_a} \overset{\iota_{ij}^*(-)}{\longrightarrow} K^{T_a}(Z_{ij}^{T_a})_{\rho_a} \overset{[\mathcal{O}_{X_i}] \boxtimes \lambda(X_j)^{-1} \cdot }{\longrightarrow} K^{T_a}(Z_{ij}^{T_a})_{\rho_a}
    \end{equation*}
    is an isomorphism (where $\iota_{ij} : Z_{ij}^{T_{a}} \rightarrow Z_{ij}$ is the inclusion).
    \begin{lemma}\label{lemma: convolution compatible with K theory localization}
        The following diagram commutes
        \begin{equation*}
            \begin{tikzcd}
                K^{T_a}(Z_{12})_{\rho_a} \otimes K^{T_a}(Z_{23})_{\rho_a} \arrow[d, "r_a \otimes r_a"] \arrow[r, "\star"] & K^{T_a}(Z_{12} \circ Z_{23})_{\rho_a} \arrow[d, "r_a"] \\
                K^{T_a}(Z_{12}^{T_a})_{\rho_a} \otimes K^{T_a}(Z_{23}^{T_a})_{\rho_a} \arrow[r, "\star"]                  & K^{T_a}((Z_{12} \circ Z_{23})^{T_a})_{\rho_a} .        
            \end{tikzcd}
        \end{equation*}
    \end{lemma}
    \begin{proof}
        This is proved exactly as in \cite[Thm 5.11.10]{chriss2009representation} using the results from \cref{thm: localization in equivariant K theory}.
    \end{proof}

    \subsection{Borel-Moore homology}\label{section: app Borel Moore}
    For any $k$-variety $X$ let $D^b_c(X)$ be the constructible derived category with coefficients in $\mathbb{C} \cong \overline{\mathbb{Q}}_{\ell}$ where $\ell$ is coprime to the characteristic of $k$. We denote the constant sheaf by $\textbf{1}_X$ and the dualizing complex by $\omega_X := \mathbb{D} \textbf{1}_X$. Cohomology and Borel-Moore homology are defined as
    \begin{align*}
        H^i(X) &:= \Hom_{D^b_c(X)}^i(\textbf{1}_X, \textbf{1}_X), \\
        H_i(X) &:= \Hom_{D^b_c(X)}^{-i}(\textbf{1}_X, \omega_X).
    \end{align*}
    Composition of morphisms defines cup- and cap-products
    \begin{align*}
        \cup &: H^i(X) \otimes H^j(X) \rightarrow H^{i+j}(X), \\
        \cap &: H_i(X) \otimes H^j(X) \rightarrow H_{i-j}(X).
    \end{align*}
    This equips $H^*(X)$ with the structure of a commutative $\mathbb{C}$-algebra and $H_*(X)$ with the structure of an $H^*(X)$-module. For any $f :X \rightarrow Y$, the canonical map $\textbf{1}_X \rightarrow f_*f^* \textbf{1}_X \cong f_* \textbf{1}_Y$ gives rise to a pull-back map
    \begin{equation*}
        f^*: H^i(X) \rightarrow H^i(Y)
    \end{equation*}
    which is an algebra homomorphism. If $f$ is proper, the map $f_* \omega_X \cong f_!f^! \omega_Y \rightarrow \omega_Y$ induces a pushforward map
    \begin{equation*}
        f_* : H_i(X) \rightarrow H_i(Y)
    \end{equation*}
    and if $f$ is smooth of relative dimension $d$, the map $\omega_X \rightarrow f_*f^* \omega_X \cong f_* \omega_Y [d]$ induces a pullback map 
    \begin{equation*}
        f^*: H_i(Y) \rightarrow H_{i+d}(X).
    \end{equation*}
    If $\iota: Z \hookrightarrow  X$ is closed, the cohomology of $X$ with support in $Z$ is defined as
    \begin{equation*}
        H^i_Z(X) :=  \Hom_{D^b_c(X)}^i(\textbf{1}_X, \iota_! \iota^! \textbf{1}_X).
    \end{equation*}
    For $\iota_i : Z_i \hookrightarrow X$ closed ($i=1,2$), the sheaf $\iota_{1!} \iota_1^! \textbf{1}_X \otimes \iota_{2!}\iota_2^! \textbf{1}_X$ is supported on $\iota_{12}: Z_1 \cap Z_2 \hookrightarrow X$. Thus the canonical map $\iota_{1!} \iota_1^! \textbf{1}_X \otimes \iota_{2!}\iota_2^! \textbf{1}_X \rightarrow \textbf{1}_X \otimes \textbf{1}_X \cong \textbf{1}_X$ gives rise to a map
    \begin{equation*}
        \iota_{1!} \iota_1^! \textbf{1}_X \otimes \iota_{2!}\iota_2^! \textbf{1}_X \cong \iota_{12!}\iota_{12}^!(\iota_{1!} \iota_1^! \textbf{1}_X \otimes \iota_{2!}\iota_2^! \textbf{1}_X) \rightarrow \iota_{12!}\iota_{12}^!\textbf{1}_X.
    \end{equation*}
    This induces the cup product with support
    \begin{equation*}
        \cup_X : H^i_{Z_1}(X) \otimes H^j_{Z_2}(X) \rightarrow H^{i+j}_{Z_1 \cap Z_2}(X).
    \end{equation*}
    If $f:\tilde{X} \rightarrow X$ is a morphism and $Z \subset X$ is closed, we get a cartesian diagram
    \begin{equation*}
        \begin{tikzcd}
            f^{-1}(Z) \arrow[d, "f"] \arrow[r, "\tilde{\iota}"] & \tilde{X} \arrow[d, "f"] \\
            Z \arrow[r, "\iota", hook]                          & X.                       
            \end{tikzcd}
    \end{equation*}
    The map $\iota_!\iota^! \textbf{1}_X \rightarrow \iota_! \iota^! f_*f^* \textbf{1}_X \cong f_* \tilde{\iota}_!\tilde{\iota}^!\textbf{1}_{\tilde{X}}$ gives rise to a restriction with support map
    \begin{equation*}
        f^*: H^i_Z(X) \rightarrow H^i_{f^{-1}(Z)}(\tilde{X}).
    \end{equation*} 
    If $\iota:Z \hookrightarrow X$ is closed with $X$ smooth, then there is a canonical isomorphism $\iota_!\iota^!\textbf{1}_X \cong \iota_* \omega_Z[- 2 \dim X]$ which yields an isomorphism
    \begin{equation*}
        H_i(Z) \cong H^{2\dim X -i}_Z(X).
    \end{equation*}
    In particular, the cup product with support for $Z_1,Z_2 \subset X$ yields a cap product with support
    \begin{equation*}
        \cap_X : H_i(Z_1) \otimes H_j(Z_2) \rightarrow H_{i+j - 2\dim X}(Z_1 \cap Z_2).
    \end{equation*}
    Similarly, if $f: \tilde{X} \rightarrow X$ is a morphism and $Z \subset X$ is closed, the restriction map on cohomology with support induces a restriction with support map
    \begin{equation*}
        f^* : H_i(Z) \rightarrow H_{i +  2 \dim \tilde{X} -  2 \dim X} (f^{-1}(Z)).
    \end{equation*}
    
    If $X_1,X_2,X_3$ are smooth and $Z_{12} \subset X_1 \times X_2$ and $Z_{23} \subset X_2 \times X_3$ are closed with $p_{13}: (Z_{12} \times X_3) \cap (X_1 \times Z_{23}) \rightarrow X_1 \times X_3$ proper, we can define the convolution
    \begin{equation}\label{eq: convol in BM homology}
        \begin{aligned}
        \star : H_i(Z_{12}) \otimes H_j(Z_{23}) & \rightarrow H_{i+j -2\dim X_2 }(Z_{12} \circ Z_{23})\\
         \alpha \otimes \beta & \mapsto p_{13*}(p_{12}^*\alpha \otimes p_{23}^* \beta).
    \end{aligned}
    \end{equation} 
    It can be shown that this operation is associative (c.f. \cite{chriss2009representation,joshua1998modules}). Moreover, it is shown in \cite[Prop. 8.6.35]{chriss2009representation} that for any $f_i:X_i \rightarrow Y$ with $X_i$ smooth, $f_i$ proper and $Z_{ij} := X_i \times_Y X_j$ there is an isomorphism (which shifts the grading)
    \begin{equation}\label{eq: identification of BM homology and Ext algebra}
        H_*(Z_{ij}) \cong \Hom_{D^b_c(Y)}^*(f_{j *} \textbf{1}_{X_j}, f_{i*}\textbf{1}_{X_i})
    \end{equation}
    such that convolution corresponds to composition of morphisms.
    
    Now let $G$ be an algebraic group over $k$ and $X$ a $G$-variety. All the previous constructions in $D^b_c(X)$ can also be applied to the equivariant derived category $D^b_{G,c}(X)$ (c.f. \cite{bernstein2006equivariant} and \cite[Chapter 6]{achar2021perverse}). In particular, we can consider equivariant cohomology and Borel-Moore homology
    \begin{align*}
        H^i_G(X)&:= \Hom^i_{D^b_{G,c}(X)}(\textbf{1}_X, \textbf{1}_X) \\
        H_i^G(X)&: =  \Hom^{-i}_{D^b_{G,c}(X)}(\textbf{1}_X, \omega_X)
    \end{align*}
    which come with the analogous push, pull and convolution operations (see also \cite{lusztig1988cuspidal}).
    
    There also is a localization theorem for Borel-Moore homology similar to the one for $K$-theory. Let $T$ be a torus defined over an algebraically closed field $k$ with corresponding complex Lie algebra
    \begin{equation*}
        \mathfrak{t} :=  \mathbb{C} \otimes_{\mathbb{Z}} X_*(T).
    \end{equation*}
    Any $\sigma \in \mathfrak{t}$ defines an evaluation map
    \begin{equation*}
        ev_{\sigma} : H^*_{T}(pt) \cong Sym^*(\mathfrak{t}^{\vee}) \rightarrow \mathbb{C}
    \end{equation*}
    and thus a 1-dimensional $H^*_{T}(pt)$-module $\mathbb{C}_\sigma $. Let
    \begin{align*}
        \rho_{\sigma} & := \ker(ev_{\sigma}) \\
        X^*_{\sigma} &:= \{ \lambda \in X^*(T) \mid \lambda(\sigma) = 0\}.
    \end{align*}
    Then $\rho_{\sigma}$ is a maximal ideal and $X^*_{\sigma} \subset X^*(T)$ is a subgroup with $X^*(T) / X^*_{\sigma}$ torsion-free. Hence, 
    \begin{equation*}
        T_\sigma := Spec(k[X^*(T) / X^*_{\sigma}])
    \end{equation*}
    is a subtorus of $T$.
    \begin{theorem}\label{thm: localization for BM homology}\cite{lusztig1995cuspidal,evens1997fourier}
        For any morphisms of $T$-varieties $f_i : X_i \rightarrow Y$ with $X_i$ smooth and $Z_{ij} := X_i \times_Y X_j$ such that each $X_i$ admits a locally closed embedding into a smooth projective $T$-variety, there are isomorphisms
        \begin{equation*}
            H_*^T(Z_{ij}) \otimes_{H^*_T(pt)} \mathbb{C}_{\sigma} \cong H_*^{T}(Z_{ij}^{T_{\sigma}}) \otimes_{H^*_T(pt)} \mathbb{C}_{\sigma}
        \end{equation*}
        that are compatible with convolution.
    \end{theorem}
    \begin{proof}
        This is proved for complex varieties in \cite[Thm B.2, Lemma 4.9]{evens1997fourier} (see also \cite[Prop. 4.9]{lusztig1995cuspidal}). The same proof goes through in positive characteristic.
    \end{proof}
    There also is the following criterion for a space to be equivariantly formal.
    \begin{lemma}\cite{lusztig1988cuspidal}\label{lemma: odd homology vanishing implies equivariantly formal}
        Let $T$ be a torus and $X$ a $T$-variety with $H_{odd}(X) = 0$. Then $H_*^T(X)$ is a free $H_T^*(pt)$-module and the canonical map $\mathbb{C} \otimes_{H^*_T(pt)} H_*^T(X) \rightarrow H_*(X)$ is an isomorphism.
    \end{lemma}
    \begin{proof}
        This is proved for complex varieties in \cite[Proposition 7.2]{lusztig1988cuspidal} under the assumption that $H^{odd}_c(X) = 0$. The same proof goes through in positive characteristic. Note that we have $H_i(X) \cong H^i_c(X)^{\vee}$ as vector spaces, so the vanishing conditions for $H_{odd}(X)$ and $H^{odd}_c(X)$ are equivalent.
    \end{proof}
    For any variety $Y$, let $A(Y)$ be its Chow group. Recall that the Riemann-Roch transformation is an isomorphism
    \begin{equation*}
        \tau^Y : \mathbb{C} \otimes_{\mathbb{Z}} K(Y) \rightarrow \mathbb{C} \otimes_{\mathbb{Z}} A(Y)
    \end{equation*}
    which is compatible with proper push-forward.
    \begin{theorem}\label{thm: chern character}\cite{baum1975riemann,iversen1976local,olsson2015borel}
        Let $X$ be smooth and $Y \subset X$ closed. Then there is a map $ch^X_Y : \mathbb{C} \otimes_{\mathbb{Z}}K(Y) \rightarrow H_*(Y)$ with the following properties:
        \begin{enumerate}[label = (\roman*)]
            \item $ch^X_Y$ is additive;
            \item $ch^X_Y$ is functorial (i.e. commutes with pullback with support);
            \item $ch^X_Y$ is multiplicative (i.e. $ch^X_{Y \cap Y'}([\mathcal{F}] \otimes_X^L [\mathcal{G}]) = ch^X_{Y}([\mathcal{F}]) \cap_X ch^X_{Y'}([\mathcal{F}]) $);
            \item $Td(X) \cdot ch^X_Y = cl_Y \circ \tau^Y$ where $Td(X)$ is the Todd class of the tangent bundle of $X$ (which is an invertible element of $H^*(X)$).
        \end{enumerate}
    \end{theorem}
    \begin{proof}
        It is shown in \cite[Theorem 1.3]{iversen1976local} that for complex varieties, there is a transformation $ch^X_Y : K(Z) \rightarrow H^*_Y(X) \cong H_*(Y)$ with the first three properties. As pointed out in the abstract of \cite{iversen1976local} the same argument applies for positive characteristic and étale/$\ell$-adic cohomology (this is spelled out in more detail in \cite[§3.1]{olsson2015borel}). The last property is shown in \cite[Theorem 1.5]{olsson2015borel}.
    \end{proof}
    \begin{lemma}\label{lemma: convolution compatible with Chern character}
        For $X_i,X_j$ smooth and $Z_{ij} \subset X_i \times X_j$, the map
        \begin{equation*}
            RR: \mathbb{C} \otimes_{\mathbb{Z}} K(Z_{ij}) \overset{ch_Z^{X_i \times X_j}}{\longrightarrow} H_*(Z_{ij}) \overset{1 \boxtimes Td(X_j)}{\longrightarrow } H_*(Z_{ij})
        \end{equation*}
        is compatible with convolution.
    \end{lemma}
    \begin{proof}
        This is proved exactly as in \cite[Thm. 5.11.11]{chriss2009representation} using the properties from \cref{thm: chern character}.
    \end{proof}

\bibliographystyle{alpha}
\bibliography{bibliography}
\end{document}